\documentclass[10pt,amssymb]{amsart}
\usepackage{amsfonts,amsmath,latexsym,amsthm,amssymb,xypic,graphicx}
\usepackage[usenames,dvipsnames]{color}
\input xy
\xyoption{all}
\usepackage{amstext}
\usepackage{cite}

\newcommand\ownremark[1]{}

\newcommand{\RR}{\mathbb R}
\newcommand{\ZZ}{\mathbb Z}
\newcommand{\QQ}{\mathbb Q}
\newcommand{\NN}{\mathbb N}
\newcommand{\CC}{\mathbb C}
\newcommand{\calA}{\mathcal A}

\newcommand{\calT}{{\mathcal T}}
\newcommand{\calY}{\mathcal Y}
\newcommand{\calZ}{\mathcal Z}

\newcommand{\calI}{\mathcal I}
\newcommand{\calJ}{\mathcal J}

\newcommand{\del}{\partial}
\newcommand{\dM}{\partial M}
\newcommand{\CN}{\dot{C}^N}

\newcommand{\ff}{{\rm ff}}
\newcommand{\lf}{{\rm lf}}
\newcommand{\rf}{{\rm rf}}
\newcommand{\dx}{\partial_x}
\newcommand{\dyi}{{\partial_{y_i}}}
\newcommand{\dzj}{{\partial_{z_j}}}
\newcommand{\dwk}{{\partial_{w_k}}}

\newcommand{\calV}{\mathcal{V}}
\newcommand{\calE}{\mathcal{E}}
\newcommand{\calF}{\mathcal{F}}
\newcommand{\calG}{\mathcal{G}}
\newcommand{\calS}{\mathcal{S}}
\newcommand{\bfa}{{\bf{a}}}
\newcommand{\bfPhi}{{\bf{\Phi}}}
\newcommand{\Btilde}{\tilde{B}}

\newcommand{\Ytilde}{\tilde{Y}}
\newcommand{\Vtilde}{{\tilde{V}}}

\newcommand{\xtilde}{\tilde{x}}
\newcommand{\ytilde}{\tilde{y}}
\newcommand{\utilde}{\tilde{u}}
\newcommand{\Ttilde}{\tilde{T}}
\newcommand{\Omegatilde}{\tilde{\Omega}}
\newcommand{\Deltatilde}{\tilde{\Delta}}

\newcommand{\fbar}{\overline{f}}
\newcommand{\gbar}{\overline{g}}

\newcommand{\eps}{\epsilon}
\newcommand{\interior}[1]{\mathring{#1}}
\newcommand{\intM}{\interior{M}}
\newcommand{\intX}{\interior{X}}
\newcommand{\Span}{\operatorname{span}}
\newcommand{\id}{\operatorname{id}}
\newcommand{\Red}{\operatorname{Red}}

\newcommand{\Diff}{\operatorname{Diff}}

\newcommand{\cl}{\operatorname{cl}}
\newcommand{\MO}{(M)}  
\newcommand{\MOEF}{(M,E,F)}  
\newcommand{\intff}{\interior{\ff}}
\newcommand{\sus}{{\rm{sus}}}
\newcommand{\sfV}{\mathsf{V}}
\newcommand{\Nhat}{\widehat{N}}
\newcommand{\Ahat}{\hat{A}}
\newcommand{\xPsiinfty}{x^\infty\Psi_\bfa^{-\infty}(M)}

\newcommand{\dvol}{\operatorname{dvol}}
\newcommand{\Cdot}{\dot{C}}
\newcommand{\loc}{{\operatorname{loc}}}
\newcommand{\clos}{\operatorname{clos}}

\newtheorem{theorem}{Theorem}
\newtheorem{lemma}{Lemma}
\newtheorem{proposition}{Proposition}
\newtheorem{corollary}{Corollary}
\newtheorem{definition}{Definition}

\begin{document}

\title{Pseudodifferential operator calculus for generalized $\mathbb{Q}$-rank 1 locally symmetric  spaces, I}
\author{
Daniel Grieser}
\address{Institut f\"ur Mathematik, Carl-von-Ossietzky Universit\"at Oldenburg, 26111 Oldenburg, Germany}
\email{grieser@mathematik.uni-oldenburg.de}
\author{
Eug\'enie Hunsicker}
\thanks{EH was partially supported by the Max Planck Institute of Mathematics, Bonn}
\address{Department of Mathematical Sciences, Loughborough University, Loughborough, Leic. LE11 3TU, United Kingdom}
\email{e.hunsicker@lboro.ac.uk}
\date{September 26, 2008}


\begin{abstract}
This paper is the first of two papers constructing a calculus of pseudodifferential operators suitable
for doing analysis on $\QQ$-rank $1$ locally symmetric spaces and Riemannian manifolds generalizing these.
This generalization is the interior of a manifold with boundary, where the boundary has the structure of a tower of fibre bundles. The class of operators we consider on such a space includes those arising naturally from metrics which   degenerate to various orders at the boundary, in directions given by the tower of fibrations. As well as $\QQ$-rank $1$ locally symmetric spaces, examples include  Ricci-flat
metrics on the complement of a divisor in a smooth variety constructed by Tian and Yau.
In this first part of the calculus construction, parametrices are found for ``fully elliptic differential \bfa-operators'', which are uniformly elliptic operators on these manifolds that satisfy an additional invertibility condition at infinity.
In the second part we will consider operators that do not satisfy this condition.
\end{abstract}

\maketitle

\tableofcontents

\section{Introduction}

In this paper we present the first part of the construction of a  calculus of pseudodifferential operators that will contain  parametrices for Dirac-type operators over manifolds that are  generalizations of $\mathbb{Q}$-rank 1 locally symmetric spaces.   These generalizations are complete noncompact Riemannian manifolds
that off a compact set have the structure $\mathbb{R}^+ \times Y$,  where $Y \rightarrow B_1 \rightarrow B_0$ is a double fibre bundle and  $Y$ is compact.  That is, $Y$ is the total space of a fibre bundle  over $B_1$,  which in turn is the total space of a fibre bundle over  $B_0$.
The metrics we consider on these
spaces, which we call $\bfa$-cusp metrics, are metrics that as $r\in \mathbb{R}^+$ approaches infinity
cause the two fibres to collapse at rates $e^{-a_1 r}$ and $e^{- (a_1+a_2) r}$ for a pair of positive integers
$\bfa=(a_1, a_2)$, while the base $B_0$ does not collapse.   In the  case of $\mathbb{Q}$-rank 1 locally symmetric
spaces, the fibres will be tori, and the $a_i$ will come from the Lie-theoretic weights \cite{Bo}.
These algebraic manifolds have been extensively studied in  \cite{muller}, where generalized eigenfunctions are constructed and an  $L^2$ index theorem is proved.
The calculus we construct will also be relevant to manifolds whose  metrics are conformally
equivalent to $\bfa$-cusp metrics.  In particular, a class of metrics  conformal to the \bfa-cusp metrics, which
we call \bfa-boundary metrics, arises naturally in the analysis of \bfa-cusp metric operators.
Given a manifold as above it is convenient to consider it as the  interior of a manifold with boundary $M$, where $\del M = Y$, and  where $x\in [0,1)$ is a boundary defining function on $M$ equal to
$e^{-r}$ near $\del M$.
This is the approach we will take when these definitions are made  precise in Section \ref{section:definitions}, and we will call  such an $M$ a {\em manifold with double-fibred boundary} or $\bfa$-manifold.
We emphasize here, however, that the operators we consider act only on the {\em interior} of $M$, and since their structure of degeneracy corresponds to a complete metric, we do not need to impose  boundary conditions.

There are several motivations for creating a pseudodifferential  calculus
manifolds with {\bf a}-metrics.  One of these is to develop tools to  explore how analytic results about locally
symmetric spaces, here particularly those of $\mathbb{Q}$-rank 1,  might generalize to manifolds
with similar geometric structures but without strict algebraic  structures.
For example, it is interesting to understand what versions of the  $L^2$ Hodge theorem
apply for such manifolds, generalizing the theorem for arithmetic  locally symmetric spaces
by Saper and Stern in \cite{SaSt} or what $L^2$ index theorem might exist, generalizing
that of M\"uller in \cite{muller}.  A pseudodifferential calculus  would also permit us to construct an extension of the
resolvent for these manifolds, done for $\mathbb{Q}$-rank 1 locally  symmetric spaces in \cite{Gu} and \cite{Mu2}.
There has been recent work on analytic continuation of the  resolvent of the Laplacian on (globally) symmetric spaces
by Mazzeo and Vasy in \cite{MeVa}, who, along with M{\"u}ller, also  continue to work on the higher rank
case for locally symmetric spaces.

A second motivation is to study elliptic operators on, and in  particular, prove a full Hodge-type theorem for, the noncompact  manifolds constructed by Tian and Yau in \cite{TY1} as solutions of  the noncompact Calabi conjecture. This is part of a larger program to  prove Hodge theorems for manifolds with special holonomy, and is  related to Sen's conjecture in string theory \cite{sen2}.  Many  mathematicians have
worked in this area. To name just a few, Hitchin studied the $L^2$  Hodge theorem outside of the middle degree for monopole moduli spaces  in \cite{Hi}.  The second author, together with Mazzeo and Hausel, proved an $L^2$ Hodge theorem for various cases of gravitational  instantons in \cite{HHM}.  Carron has developed techniques for  calculating dimensions of harmonic forms of various degrees in a  variety of geometric situations, \cite{Ca1}, \cite{Ca2}, \cite{Ca3}.   Mazzeo and Degeratu are currently writing up an $L^2$ Hodge result  about certain QALE spaces.   The Hodge theorem outside the middle  degree for the manifolds constructed by Tian and Yau was proved in  \cite{Hu1}, and a conjecture proposed for the middle degree case. The  metrics on the manifolds in question are conformal $\bfa$-metrics, and  the theory in our two papers permits us to confirm the conjecture in  \cite{Hu1}, as will be described in a subsequent paper.

A third motivation is to study spectral properties of manifolds with  metrics that  interpolate between infinite cylindrical ends and fibred  cusp ends, both of which are interesting and well-studied classes of noncompact manifolds.  It is interesting therefore to see how these  can be seen as opposite limiting cases in a family of examples.  A  similar study, resulting in a family of Hodge and signature theorems,
was done for singular manifolds with metrics between finite cylinders  and cone bundles off a compact set in \cite{Hu2}.

Finally, this work is part of a general program to construct pseudodifferential operator calculi on singular and noncompact
manifolds.  Various groups of researchers have used different  approaches to do this, including Melrose and his collaborators, who  construct calculi using operator kernels defined on ``blown up'' double  spaces (more on this below), Schulze and his collaborators, for  instance in \cite{Shul}, using an approach generalizing that of Boutet  de Monvel, and recently, Ammann, Lauter and Nistor
in their paper \cite{ALN} which uses integration of Lie algebroids.
The calculus for double-fibred manifolds that we construct in this  paper is the natural generalization and next step
in the program of the Melrose school after the phi-calculus of  \cite{MaMe}.  It extends easily to the case of more than two nested  fibrations. We restrict the main part of the presentation to the case  of two fibrations in order to keep the notation manageable; the  essentials needed for the extension to any number of fibrations are  given in an appendix. Another way in which our calculus extends \cite{MaMe} is that we allow arbitrary powers of degenerations (called $a_i$ above) for the various fibrations.

This paper and its sequel build on a great deal of previous work on  the construction of pseudodifferential calculi for singular and  noncompact manifolds.  The general approach is taken from the
work by Melrose, starting with his work on the b-calculus of operators  over manifolds with infinite
cylindrical ends \cite{Me-aps}, see also \cite{Me-Men}.
The main idea of this approach is to construct a calculus of operators
through studying distributions on a ``blown up double space''.  That  is, given a manifold $M$ as
above, we identify differential and pseudodifferential operators on $M$ with their Schwartz kernels, which are distributions on $M^2$, and study these by lifting them to a blow-up $M^2_z$.  This  method has been used to develop many different  calculi of pseudodifferential operators, including, for instance, the scattering calculus \cite{Me-sc}
and the 0-calculus and edge calculus, \cite{Ma-edge}, which are all calculi
designed to study geometric operators on noncompact manifolds with different types of geometry.
The original b-calculus is a special case of the calculi that will be constructed in this paper,
where the two boundary fibrations are trivial, as is Wunsch's quadratic scattering calculus \cite{W},
where there is only one boundary fibration, and the scaling is given by $a_1=2$.
Finally, the phi-calculus is suited for manifolds with  fibred cusp metrics, which are metrics like those above, but where the  second fibration,
$B_1 \rightarrow B_0$, is trivial \cite{MaMe}.  In 2001, Boris  Vaillant, a PhD student working
with W. M{\"u}ller and R. Melrose, completed a dissertation supplementing  the construction of the
phi-calculus in a way that permitted him to prove an index theorem for  fibred cusp manifolds.
These two pieces of work on the phi-calculus have been much of the  basis for the work in this paper
and its successor.   Unfortunately,
Vaillant left mathematics after the completion of his dissertation,  and this important work was therefore never revised and submitted to a  journal.  Nevertheless, there has been a fair degree of
interest in his results, which among other things have been used to  prove the Hodge theorem in \cite{HHM} and to
prove a families index theorem for manifolds with hyperbolic cusps in  \cite{AlRo}, of
interest in smooth K-theory of noncompact manifolds (defined in  \cite{MR}) and with an application to Teichm\"uller theory in \cite{AlRo2}.
Thus an additional motivation for the present work is to provide a  revised and edited (as well as generalized) version of Vaillant's work.

The content of this first paper is a generalization of the material in  \cite{MaMe}.  In it we define and
study the classes of \bfa-differential  and \bfa-pseudodifferential operators on an \bfa-manifold $M$.
If $A$ is a geometric operator over $M$ for an
\bfa-boundary metric, then $A$ will be an \bfa-differential operator,  and if $A$ is a geometric
operator of degree $m$ associated to an \bfa-cusp metric on $M$, then  $x^{m(a_1+a_2)}A$
will be an \bfa-differential operator.
These operators are elliptic in the sense that their principal symbols are invertible, uniformly up to the boundary after a suitable rescaling. A stronger condition is full ellipticity, which requires in addition invertibility of an associated family of operators on the fibres of $\partial M\to B_1$. Our first main theorem is:
\begin{theorem}
\label{th:psd-calc}
On any $\bfa$-manifold $M$ there is an algebra of operators, $\Psi^*_\bfa(M)$ that contains the $\bfa$-differential operators and parametrices for fully elliptic elements. The elements of $\Psi^*_\bfa(M)$ are given by conormal distributions on the blown up double space $M^2_z$.
\end{theorem}

Here a parametrix is an inverse up to operators in $x^\infty\Psi^{-\infty}_\bfa(M)$, which are operators whose kernels are smooth on $M^2$ and vanish to infinite order at the boundary of $M^2$. There is an obvious extension of Theorem \ref{th:psd-calc} to operators acting between sections of vector bundles over $M$.

Theorem \ref{th:psd-calc} applies in particular to the resolvent of the Laplacian, and we obtain:
\begin{theorem}
\label{th:resolv}
For a manifold $M$ with a double-fibred boundary endowed with an  \bfa-boundary
metric $g_{\bfa b}$ let $\Delta_{\bfa b}$ be the Hodge Laplace operator on differential forms. For any $\lambda\in \CC \setminus [0,\infty)$  the operator $\Delta_{{\bfa b}} -\lambda$ is an invertible  fully elliptic \bfa-operator of degree 2 and the  resolvent
$(\Delta_{{\bfa b}} -\lambda)^{-1}$ is an element of the small \bfa-calculus.
\end{theorem}

We then apply Theorem \ref{th:psd-calc} to obtain the usual Fredholm and regularity results. In the following theorem, $g$ is a metric of the form $x^{2r}g_{\bfa b}$ on $\intM$ for some \bfa-boundary metric $g_{\bfa b}$ and $r\in\RR$, $\dvol_g$ is the associated volume form, and $L^2(M,\dvol_g)$  and
$H^k_\bfa(M,\dvol_g)$, $k\in\RR$, are the $L^2$ and Sobolev spaces naturally associated with $g$ and the $\bfa$-structure.

\begin{theorem}
\label{th:mainthm}
Let $P\in\Psi^m_\bfa(M)$ be a fully elliptic \bfa-operator on an $\bfa$-manifold, and let $g$ be a metric as described above.
Then $P$ is a Fredholm operator on $L^2(M,\dvol_g)$.  Further, if $Pu=f$ where $u$ is a tempered distribution and
$f \in  x^\alpha H^k_{\bfa}(M,\dvol_g)$ 
then in fact $u \in x^\alpha H^{(k +m)}_{\bfa}(M,\dvol_g)$.
In particular, all elements of the kernel and cokernel of $P$ in any  weighted Sobolev space
are smooth sections that vanish to infinite order on $\partial M$.
\end{theorem}

The space $\Psi^*_\bfa(M)$ may be called the ``small'' $\bfa$-calculus in the philosophy of b-type  calculi, since the operator kernels have non-trivial expansions only  at the highest order front face of $M^2_z$.
Note that a fully elliptic \bfa-operator is Fredholm on any weighted \bfa-Sobolev space. This, and the vanishing to infinite order of elements of the kernel of a fully elliptic operator, is also true in the $\phi$- and scattering calculus. Compare this with the b-calculus, where elliptic operators are Fredholm only off a discrete set of weights, the parametrix depends on the weight and elements of the kernel have polynomial asymptotic behavior at the boundary.

We also define the full $\bfa$-calculus and study its mapping and composition properties. In the sequel to this paper we will show that the full $\bfa$-calculus contains parametrices for a particular class of elliptic but not fully elliptic $\bfa$-operators.  This is
an important extension, as the most geometrically interesting  operators on manifolds
with either $\bfa$-cusp or \bfa-boundary metrics, including Dirac  operators, belong to this class.  This step is a  generalization of  the extension done by Vaillant  in his dissertation  of the phi-calculus. In future papers we will apply these techniques  to obtain Hodge and index theorems for manifolds with $\bfa$-cusp and  boundary metrics.

The rough outline of this paper is as follows.
In Section 2, we make several definitions necessary for the  construction of the small $\bfa$-calculus,
then go on to a description of the geometry of the $\bfa$-double space  and discuss the kernels of
$\bfa$-differential operators lifted to this space.  We then define  the small $\bfa$-calculus and the normal operator of an element of the  small calculus, whose invertibility gives the ellipticity condition at  the boundary that was discussed earlier.  In Section 3 we define and analyze the \bfa-triple space,  introduce the full \bfa-calculus and prove  composition and mapping theorems for it.  Finally, in Section 4 we construct the parametrix of a  fully elliptic \bfa-differential operator and prove the Theorems stated above.





\section{Definitions and basic lemmas}
\label{section:definitions}
\subsection{Double-fibration structures}
\label{subsec:double}
\begin{definition} Let $M$ be a manifold with boundary, where
$\del M = B_2$ is the total space of a fibre bundle
$F_2 \hookrightarrow B_2 \stackrel{\phi_2}{\rightarrow}B_1$
and $B_1$ is again the total space of a fibre bundle
$F_1 \hookrightarrow B_1 \stackrel{\phi_1}\rightarrow B_0$.
Then $M$ is called a {\em manifold with double-fibred boundary}.
\end{definition}
Thus, $B_2$ comes with two fibrations, $\phi_2:B_2\to B_1$ and $\phi_{2,1}=\phi_1\circ\phi_2:B_2\to B_0$, the fibres of the latter being unions of fibres of the former.
Fix the notation that dim$(F_i)=f_i$ and dim$(B_0)=b$, so dim$(M)=f_1+f_2+b+1$.

For convenience we will assume throughout the paper that $M$ is compact. However, this is not needed for many of the considerations.
The prototype of metrics we consider is the following.
\begin{definition}  Let $M$ be a manifold with double-fibred boundary and fix $a_1,a_2 \in \mathbb{N}$.
Then a metric $ds_{{\bf a}c}^2$ on the interior of $M$ is called a {\em product-type} {\bf a}{\em -cusp metric}
if for some trivialization of a neighborhood  $U$ of the boundary,
$U  \cong [0,1)_x \times B_2$, it has the form, for $x>0$:
\begin{equation}
\label{eq:def product metric}
ds_{{\bf a}c}^2 = \frac{dx^2}{x^{2}} + (\phi_1 \circ \phi_2)^*ds_{B_0}^2 + x^{2a_1}\phi_2^*h_1 +
x^{2(a_1+a_2)}h_2,
\end{equation}
\noindent
where
$ds_{B_0}^2$ is a metric on $B_0$, $h_1$ is a nondegenerate 2-form on $B_1$ that restricts to a metric on each fibre, $F_1(p)=\phi_1^{-1}(p)$, $p \in B_0$, and
$h_2$ is a nondegenerate 2-form  on $B_2$ that restricts
to a metric on each fibre, $F_2(q)=\phi_2^{-1}(q)$, $q \in B_1$.

A {\em product type} {\bf a}{\em -boundary metric} is one of the form $ds_{{\bf a}b}^2 =x^{-2(a_1+a_2)}ds_{{\bf a}c}^2$
for some product type {\bf a}-cusp metric $ds_{{\bf a}c}^2$, that is, has the form:
\begin{equation}
\label{eq:def product metric2}
ds_{{\bf a}b}^2 = \frac{dx^2}{x^{2(1+a_1+a_2)}} +
\frac{(\phi_1 \circ \phi_2)^*ds_{B_0}^2}{ x^{2(a_1+a_2)}}
+ \frac{\phi_2^*h_1}{ x^{2a_2}}+
h_2,
\end{equation}

\end{definition}
Fibred cusp metrics are closer to the intended applications, and fibred boundary metrics are more naturally associated with the calculus constructed below.

Note that `product type' refers to the fact that none of $ds_{B_0}^2$, $h_1$ or $h_2$ depends on $x$. The fibrations may be non-trivial.
This is a stronger condition than we need for analysis, although it simplifies things if
we can choose to work with such a metric, as, for instance, we can when proving Hodge
theorems.  However, the analysis below will be carried out for a somewhat
more general class of metrics having the same structure of degeneracy as $x\to 0$ as $ds_{{\bf a}c}^2$.

The degeneracy as $x\to 0$ of these metrics and of the associated differential operators is efficiently encoded in a set of vector fields, $^\bfa\calV(M)$, which serves as the basic object in building the pseudodifferential calculus used to analyze these operators. To make this explicit, we fix a trivialization $M\supset U \cong B_2\times [0,1)$ and extend the fibrations $\phi_i$ to the interior by
\begin{equation}
\label{eq:product structure}
U\cong B_2\times[0,1)\stackrel{\Phi_2=\phi_2 \times \id}\longrightarrow B_1\times[0,1)
\stackrel{\Phi_1=\phi_1\times\id} \longrightarrow B_{0}\times[0,1).
\end{equation}
This defines a boundary defining function (or bdf) on $U$, always denoted $x$, given by the coordinate in $[0,1)$, as well as the two fibrations $\Phi_2$ and $\Phi_{2,1}=\Phi_1\circ\Phi_2$ with total space $U$ and the same fibres as $\phi_1$, $\phi_{2,1}$.
Then $^\bfa\calV(M)$ consists of those smooth vector fields $V$ on $M$ that have {\em bounded length} with respect to $ds^2_{\bfa b}$. From \eqref{eq:def product metric2} it is clear that these are characterized by
\begin{equation}
\label{eq:def a calV'}
(\Phi_2)_* V = O(x^{a_2}),\ (\Phi_{2,1})_* V = O(x^{a_1+a_2}),
\
x_* V = O(x^{1+a_1+a_2}).
\end{equation}
Here, $O(x^a)$, for sections of a bundle $E$ over $M$, denotes an element of the form $x^a f$ with $f$ a smooth (up to the boundary!) section of $E$. Also $(\Phi_2)_*V$ is regarded as a section of the pull-back bundle $(\Phi_2)^*T(B_1\times[0,1))$ over $B_1\times[0,1)$, and similar for the other maps.
Note that condition \eqref{eq:def a calV'} does not depend on $ds_{B_0}^2$, $h_1$ and $h_2$ in the definition of $ds^2_{\bfa b}$; hence we use it to define:
\begin{definition} \label{def:interior dfs}
Let $M$ be a manifold with double-fibred boundary and $\bfa=(a_1,a_2)\in \NN^2$.
An {\em \bfa-structure} on $M$ is a set of vector fields
$$ ^\bfa\calV(M) = \{V\in \Gamma(TM): \, V\text{ satisfies \eqref{eq:def a calV'}}\}$$
for some trivialization $M\supset U \cong B_2\times [0,1)$. $M$ together with an \bfa-structure will also be called an {\em \bfa-manifold}.
\end{definition}
Condition \eqref{eq:def a calV'} says that $V$ is tangent to order $a_2$ (as $x\to 0$) to the fibres $F_2(w)$ of $\Phi_2$, tangent to order $a_1+a_2$ to the fibres $F_{21}(p)$ of $\Phi_{2,1}$, and tangent to order $1+a_1+a_2$ to the fibres of $x$. This shows that $^\bfa\calV(M)$ does not uniquely determine the fibrations $\Phi_2,\Phi_1,x$, or the trivialization $U\cong B_2\times [0,1)$. However, it does determine $\Phi_2,\Phi_1,x$ uniquely at the boundary (up to diffeomorphisms of the bases), and also to some finite orders in the interior. This is discussed in greater detail in the appendix.

We assume from now on that a trivialization as in Definition \ref{def:interior dfs}  is chosen.
However, any object introduced below that is decorated with an upper index $\bfa$  will be determined purely by $^\bfa\calV(M)$ and not depend on the particular choice of trivialization.

Throughout we will use coordinates adapted to the trivialization and the fibrations. These are, near a boundary point, $\{x, y_1, \ldots, y_b, z_1, \ldots , z_{f_1}, w_1, \ldots w_{f_2}\}$ where $x$ is the boundary defining function,
the $y_i$ are coordinates lifted from the base $B_0$, the $z_j$ are $F_1$ coordinates and the
$w_k$ are $F_2$ coordinates. \footnote{More precisely, $y_i$ are coordinates on $B_0$, $z_j$ are local functions on $B_1$ so that combined with the $(\phi_1)^*y_i$ they give coordinates on $B_1$, and $w_k$ are local functions on $B_2$ that combined with the $(\phi_{2,1})^*y_i$ and the $(\phi_2)^*z_j$ give coordinates on $B_2$. For simplicity we write $z_j$ for $(\phi_2)^*z_j$ etc.}
The question of which coordinate changes are admissable given an \bfa-structure is discussed in the appendix.
The set of ${\bf a}$-vector fields is a $C^{\infty}(M)$ module and is locally spanned by:
$$
x^{1+a_1+ a_2} \del_x,\ x^{a_1+a_2} \del_{y_i},\ x^{a_2}\del_{z_j},\ \del_{w_k}.
$$
Thus $^{\bf a} \mathcal{V}(M)$ is the space of sections of a bundle which we will call the ${\bf a}${\em-tangent bundle} over $M$ associated to $^\bfa\calV(M)$ and denote by  $^{\bf a}TM$.

Associated with $^\bfa\calV(M)$ is the space of one-forms
$$ ^\bfa\Omega(M) = \{ \omega\in \Omega^1(\interior{M}):\, \omega(V) \text{ extends smoothly from }\interior{M}\text{ to } M,\text{ for all }V\in {}^\bfa\calV(M)\}.
$$
It is locally spanned by
\begin{equation}
\label{eq:local basis aT*}
\frac{dx}{x^{1+a_1+a_2}},\ \frac{dy_i}{x^{a_1+a_2}},\ \frac{dz_j}{x^{a_2}},\ dw_k
\end{equation}
over $C^\infty(M)$, and is the space of sections of the dual bundle $^\bfa T^*M$.

\begin{definition}
An {\em $\bfa$-boundary metric} (or {\em $\bfa$-metric}) on an \bfa-manifold $M$ is a
smooth positive definite section of the symmetric square of $^\bfa T^*M$.
An {\bf a}{\em -cusp metric} is one of the form $ds_{{\bf a}c}^2 =x^{2(a_1+a_2)}ds_{{\bf a}b}^2$
for some {\bf a}-boundary metric $ds_{{\bf a}b}^2$.\end{definition}
In other words, an $\bfa$-boundary metric is, locally near the boundary, a symmetric positive definite quadratic form in the expressions \eqref{eq:local basis aT*}, with smooth coefficients.


There are several vector bundles over $\partial M$ associated naturally with an \bfa-structure on $M$, which will be needed as spaces carrying  various reduced (or model) operators.
Here we simply define them in terms of their local bases and refer to the appendix, see \eqref{eq:def a-normal bundle}, for an intrinsic characterization. The local bases are:
\begin{align*}
{}^\bfa N\partial M: &\quad x^{1+a_1+a_2}\partial_x,\ x^{a_1+a_2}\partial_{y_i},\ x^{a_2}\partial_{z_j}\\
{}^\phi N\partial M: &\quad x^{1+a_1}\partial_x, x^{a_1}\partial_{y_i}\\
{}^{\rm b} N\partial M:&\quad x\partial_x
\end{align*}
Note that $^{\rm b} N\partial M$ is the b-normal bundle from \cite{Me-aps}. It is canonically trivial since $x\partial_x$ is independent of the choice of bdf $x$. The bundles ${}^\phi N\partial M$ and ${}^\bfa N\partial M$ are in general not trivial. ${}^\phi N\partial M$ is the pull-back of a bundle ${}^\phi N'B_0$ on $B_0$, and in case $a_1=1$ is the phi-normal bundle defined in \cite{MaMe} for the fibration $\partial M \stackrel{\phi_{2,1}}{\to} B_0$. Similarly,
${}^\bfa N\partial M$ is the pull-back of a bundle ${}^\bfa N'B_1$ on $B_1$, which is just ${}^\bfa T_{B_1\times\{0\}}(B_1\times [0,1))$ for the induced \bfa-structure on $B_1\times[0,1)$, which arises by squashing the $F_2$-fibres to points. See again the appendix, equation \eqref{eq:normal bundle as pullback}.



\begin{definition}
An {\em {\bf a}-differential operator} on an \bfa-manifold $M$  is an element of the universal
enveloping algebra of $^{{\bf a}}\mathcal{V}(M)$, that is, an operator of the form $a_0 + \sum_{l=1}^m V_{l,1}\cdots V_{l,l}$ for some $m\in\NN_0$, with $a_0\in C^\infty(M)$ and all $V_{l,i}\in {}^{{\bf a}}\mathcal{V}(M)$. The set of \bfa-differential operators of order at most $m$ will be denoted ${}^\bfa\Diff^m(M)$.
\end{definition}
Note that $^{{\bf a}}\mathcal{V}(M)$ is a Lie-subalgebra of $\mathcal{V}(M)$, i.e.\ closed under brackets.
Therefore, ${}^\bfa\Diff^1(M)= C^\infty(M) + {}^\bfa\calV(M)$.

In terms of the local coordinates above, an {\bf a}-differential operator of order $m$ has the form:
\begin{equation}
P = \sum_{\alpha + |I| + |J| +|K| \leq m}
a_{\alpha, I, J,K}(x, y, z,w) (x^{1+a_1+a_2} D_x)^\alpha (x^{a_1+a_2} D_y)^I (x^{a_2}D_z)^J
(D_w)^K,
\label{eq:difflocal}
\end{equation}
\noindent for multi-indices $I, \, J$ and $K$,
where  $(x^{a_2}D_z)^J= x^{a_2}D_{z_1}^{J_1} \cdots x^{a_2} D_{z_f}^{J_f}$, and we define
$(x^{a_1+a_2} D_y)^I $  and $(D_w)^K$ analogously. Here $D_x = \frac1i \partial_x$ etc.

The definition of an {\bf a}-differential operator can be
extended in an obvious way to operators that map from sections of one vector bundle $E$ over $M$ to sections
of a second vector bundle $F$ over $M$.  The space of such operators will be denoted
by ${}^\bfa\Diff^m(M;E,F),$ or simply by ${}^\bfa\Diff^m(M;E)$ if $E=F$.

\medskip
\noindent{{\bf Densities etc.:}}
For purposes of bookkeeping, it is useful (as in, e.g. \cite{Me-aps}) to consider everything in sight (functions, forms, operator kernels, and so on) as having coefficients in a half-density bundle. This allows us to integrate products of two such objects invariantly, i.e.\ without choosing an extra measure. This occurs, for example, when we define the action of kernels on functions and the composition of kernels, see \eqref{eq:action}, \eqref{eq:comp}.
A density on a manifold $M^n$ is an object that in local coordinates has the form $h(x)|dx_1\cdots dx_n|$ for a smooth function $h$; more invariantly it may be considered as a section of a line bundle (as usual the local form is supposed to indicate the transition maps between local trivializations of this bundle), the density bundle, which is associated to the cotangent bundle and denoted as $|\bigwedge^n T^* M|$. Correspondingly, a half density is locally of the form $h(x)|dx_1\cdots dx_n|^{1/2}$, and a section of the half density bundle $\Omega^{1/2}M := |\bigwedge^n T^*M|^{1/2}$. On an $\bfa$-manifold it is natural to consider $\Omega_\bfa M := |\bigwedge^n{}^\bfa T^*M|$ (not to be notationally confused with
the space of smooth \bfa-one forms, ${}^\bfa\Omega(M)$). That is, in adapted coordinates around a boundary point, a smooth $\bfa$-density is of the form
\begin{equation}
\label{eq: half density}
h(x,y,z,w)\,\left| \frac {dx}{x^{1+a_1+a_2}} \prod_{i}\frac{dy_i}{x^{a_1+a_2}} \prod_j \frac{dz_j}{x^{a_2}} \prod_k dw_k \right|
\end{equation}
with $h$  a smooth function, and the density is called positive if $h$ is positive.
The volume density of an $\bfa$-metric is a positive $\bfa$-density. If $\nu$ is a smooth positive density then $\nu^{1/2}$ is a half-density. Fixing such a $\nu$ allows us to identify arbitrary half-densities $\alpha$ with smooth functions $f$, via $\alpha=f\nu^{1/2}$. In this way a differential operator $P$ as in \eqref{eq:difflocal} can be considered as acting on half-densities, by setting $P'(f\nu^{1/2}) := P(f) \nu^{1/2}$. $P'$ will depend on the choice of $\nu$, but only up to conjugation by a smooth function, which does not affect its main properties. Therefore we will continue to write $P$ instead of $P'$. See Subsection \ref{subsec.proofs} for further discussion of this.

In this paper we will use both regular and \bfa-half densities.

\subsection{Regularity definitions}

Various types of regularity will be considered in this paper, both for functions or sections over
our original manifold, and for functions, sections, and distributions used in describing
the kernels for the space of pseudodifferential operators we will construct.  We lay them
out here for reference.  The definitions here, and more discussion, can be found in various
sources, including \cite{Ma-edge}, \cite{EMM},\cite{Me-aps}, and \cite{Me-mwc}.

\begin{definition}
\label{def:mwc}
An n-dimensional {\em manifold with corners}, $X$,  is a
topological manifold with
boundary $\del X$ that is a submanifold of an n-dimensional smooth manifold without
boundary, $\tilde{X}$ such that in a neighborhood $U$ of any point
on the boundary of $X$, there is a finite set of smooth functions $x_1, \ldots x_k$ on $\tilde{X}$
whose differentials are linearly independent at each point of $\del X \cap U$ and such that
$U$ is the set $\{x_1 \geq 0, \ldots, x_k \geq 0\}$ with $\del X \cap U$ given by
the subset where one or more of the functions vanish. Also, it is assumed that the boundary hypersurfaces, introduced below, are embedded.
\end{definition}
The closures of the maximal connected subsets $F\subset X$ on which exactly $l$ of the functions $x_i$ vanish, in some (and hence any) local representation, are called {\em faces of codimension $l$}, for $l=0,\dots,n$, and for $l=1$ {\em boundary hypersurfaces}. The local condition implies that they are immersed submanifolds of $\tilde{X}$, and it is assumed that they are actually embedded. This implies that, for each boundary hypersurface $H$, there is a globally defined {\em boundary defining function}, i.e.\ a smooth function $x_H:X\to[0,\infty)$ such that $x_H^{-1}(0)=H$ and $dx_{H|p}\neq 0$ at each point $p\in H$, and a trivialization $U\to [0,1)\times H$ of a neighborhood $U\subset X$ of $H$. The ambient manifold $\tilde{X}$ is not really needed anywhere, but its existence is sometimes useful for simple arguments.
Any point of $X$ is contained in a unique smallest face $F$, and if $F$ has codimension $k$, then one has local coordinates $x_1\geq 0,\dots,x_k\geq 0,y_1,\dots,y_{n-k}$ on $X$ centered at that point, with $F=\{x_1=\dots=x_k=0\}$ locally.
\begin{definition}
\label{def:psub}
A {\em p-submanifold} $Y$ of a manifold with corners $X$ is a subset of $X$ for
which near any point on $Y\cap X$ there exist local coordinates
$x_1 \geq 0, \ldots, x_k \geq 0, y_1, \ldots, y_{n-k}$ on $X$ such that $Y$ is given by the
vanishing of some subset of these coordinates.  We call $Y$ an
{\em interior p-submanifold} if it is given everywhere locally by the vanishing of some subset of the $y_i$; otherwise we call it a {\em boundary p-submanifold}.
\end{definition}

As is usual in various b-type calculi, we will consider the action of operators
on bundle sections with various types of regularity:  polyhomogeneous sections, conormal
sections, and sections in various weighted Sobolev spaces.  In addition, the kernels
of {\bf a}-differential operators and later, of {\bf a}-pseudodifferential operators, will also
be distributions with polyhomogeneous type regularity in some regions and
conormal regularity in others.  Thus we briefly recall here the definitions of conormal
and of polyhomogeneous distributions. Weighted Sobolev spaces will be introduced in Subsection \ref{section.sob}.

First we need to define index sets and index families.

\begin{definition}
Let $X$ be a manifold with corners.
An {\em index set} at a boundary hypersurface  $H \subset \del X$ is a discrete subset
$G \subset \mathbb{C} \times \mathbb{N}_0$ satisfying
\begin{enumerate}
\item for every $c\in \mathbb{R}$, the subset $G \cap \{\Re(z) < c\} \times \mathbb{N}_0$
is finite, and
\item if $(z,p) \in G$ and $0 \leq q \leq p$, then $(z,q)$ is also in $G$.
\item if $(z,p) \in G$ then also $(z+1,p) \in G$.
\end{enumerate}
An {\em index family}, $\mathcal{G}$ for $X$ is a choice of index set $\calG(H)$
for each boundary hypersurface $H$ of $X$.
\label{def:indfam}
\end{definition}
Condition 3 ensures the coordinate invariance of the definitions below.
Now we can define polyhomogeneous functions. We first consider the case of a manifold with boundary.
Denote by $\CN(X)$ the space
of functions on $\intX$ that are $N$ times differentiable and whose derivatives up to order $N$ vanish when approaching $\del X$.
\begin{definition}  Let $X$ be a manifold with a single boundary hypersurface, $H= \del X$, and let $x$ be a boundary defining function for $H$ and let $G$ be an index set.  Then
a smooth function $u$ on $\intX$ is called {\em polyhomogeneous} with index set $G$ at $H$,
and we write $u \in \mathcal{A}^G(X)$ if there are smooth functions $u_{z,p}$ on $X$
such that for all $N$,
$$
u-\sum_{\stackrel{(z,p) \in G}{\Re(z)\leq N}} x^z (\log x)^p u_{z,p} \in \CN(X).
$$
We denote this by writing:
\begin{equation}
\label{eq:phg}
u \sim \sum_{(z,p) \in G} x^z (\log x)^p u_{z,p}.
\end{equation}
\end{definition}
It is essential that $u$ need only be defined in the interior but the $u_{z,p}$ are smooth up to (i.e.\ including) the boundary.
Certain index sets are of particular importance.  If a function is smooth on $X$ up to $\del X$
then it will be polyhomogeneous with index set $\mathbb{N}_0 \times \{0\}$.
We will denote this set by $0$.  The set of functions
that vanish to all orders at $\del X$ will be indicated by the index set $\emptyset$.
Now we extend this definition to the case where $X$ may have corners.

\begin{definition}
Let $X$ be a manifold with corners and let $\calG$ be an index family for $X$. For any boundary hypersurface $H$ of $X$, denote by  $\mathcal{G}_H$ the index family for $H$ given by the collection of index sets for the boundary hypersurfaces intersecting $H$.
A smooth function $u$ on $\intX$ is called {\em polyhomogeneous with index family
$\mathcal{G}$}, and we write $u\in\calA^\calG(X)$, if
near each boundary hypersurface $H$, $u$ has an expansion as in equation \eqref{eq:phg}
in which $x$ is a boundary defining function for $H$ and
\begin{enumerate}
\item the coefficients $\sigma_{z,p}$ are required to lie in
$C^{\infty}([0,1);\mathcal{A}^{\mathcal{G}_H}(H))$ and
\item the remainder is in $\CN([0,1),\calA^{\calG_H}(H))$, with respect to some trivialization $[0,1)\times H$ of a neighborhood of $H$.
\end{enumerate}
We take this definition recursively in the codimension of the boundary faces.
\label{def:polyhom}
\end{definition}

Next we want to discuss regularity properties for distributions that will be relevant for
the descriptions of our operator kernels.

\begin{definition}
Let $X$ be a manifold and $Y \subset X$ be an (embedded) submanifold.  A distribution $u$ on $X$
is called a {\em classical conormal distribution of degree $m \in \mathbb{R}$ with respect to $Y$} if
\begin{enumerate}
\item $u$ is smooth on $X-Y$ and
\item in any local coordinate system around $Y$,
$u$ can be represented as the Fourier transform in the transverse
direction of a function $\sigma$ with prescribed asymptotics.
More precisely, let
$\{y_1, \ldots, y_k, z_1, \ldots, z_l \}$ be local coordinates around $Y$ where $Y$ is defined by
$z_{1} = \cdots = z_l = 0$.  Then
$$
u(y,z) = \int_{R^{l}} e^{i z \zeta} \sigma(y, \zeta) d\zeta
$$
where
\begin{equation}
\label{eq:symb}
\sigma(y,\zeta) \sim \sum_{j=0}^{\infty} \sigma_{m'-j}(y, \zeta)
\end{equation}
where the $\sigma_{m'-j}(y, \zeta)$ are homogeneous of degree $m'-j$ in $\zeta$,
with $m' = m + \frac{1}{4}\rm{dim}(X) - \frac{1}{2} \rm{codim}(Y)$.
\end{enumerate}
We denote the space of such distributions by $I^m(X, Y)$.  It is possible to consider a broader
definition of conormal distributions.  In texts where this is done, the classical distributions
are indicated by a subscript ``cl" or ``os".
\label{def:conorm}
\end{definition}
The asymptotic sum in \eqref{eq:symb} is defined as in the standard pseudodifferential calculus, see \cite{Sh}.
It is a non-trivial and important fact that this is independent of the choice of coordinates $(y,z)$ and that condition 2 is independent of the choice of coordinates $(y,z)$ and that $\sigma_{m'} |d\zeta|$, the {\em principal symbol} of $u$,  is defined invariantly when interpreted as (fibrewise) density on the conormal bundle of $Y$.
We can put these definitions together to get

\begin{definition}
A distribution $u$ on a  manifold with corners X is both conormal of degree $m$ with respect to an
interior p-submanifold, $Y$,
and polyhomogeneous at the boundary of $X$ with index family $\calG$ if
$u \in I^m(\interior{X}, \interior{Y})$
and if near each boundary hypersurface $H \subset X$ given by a boundary defining function $x$,
$u$ has an expansion as in equation (\ref{eq:phg}) where
\begin{enumerate}
\item the coefficients $u_{z,p}$ are
required to lie in $C^\infty([0,1); I^{m+1/4, \mathcal{G}_H}(H, H \cap Y))$ and
\item the remainder is in $\CN([0,1),I^{m+1/4, \mathcal{G}_H}(H, H \cap Y))$, with respect to some local trivialization of a neighborhood of $H$.
\end{enumerate}
Here $\mathcal{G}_H$ is as in Definition \ref{def:polyhom}.
If these conditions are satisfied, we write $u\in I^{m, \mathcal{G}}(X,Y)$

\label{def:polyhom+conorm}
\end{definition}

Note here that the change in degree as we move to hypersurfaces is caused by the definition
of degree of conormality of a distribution.  We define the space of conormal distributions that take
values in a vector bundle $E$ over $X$ analogously and denote it by
$$
I^{m, \mathcal{G}}(X,Y;E).
$$

\subsection{Quasihomogeneous blowups}
\label{subsec:qhom}
In the construction of the double space below, we will need quasi-homogeneous blow-ups. These generalize the `standard' blowups discussed, for example, in \cite{Gr} and \cite{Me-mwc}, and will be used to resolve the higher order vanishing of elements in ${}\bfa\Diff(M)$ at $\partial M$.  Quasihomogeneous blowups are
related to the inhomogeneous blowups discussed in \cite{HMV}, where points are blown up
inhomogeneously, and in the quadratic scattering calculus of \cite{W}.
The theory of quasihomogeneous blowups will be discussed in detail in \cite{GH2}. A much more general setup is introduced in \cite{Me-mwc}. We record only the basics here.

In the standard setting, a blowup is the construction of a new manifold $[X;Y]$ from a manifold with corners $X$ and a p-submanifold $Y$. For an invariant definition of a quasihomogeneous blowup of a boundary submanifold $Y$, it needs to have an extension to finite order to the interior. We discuss this notion first. It is also basic for an invariant understanding of $\bfa$-manifolds.
Let $Y$ be a boundary p-submanifold of $X$. An {\em interior extension of $Y$} is an interior p-submanifold $\Ytilde$ such that $Y=\Ytilde\cap\partial X$. In this situation, local coordinates as in Definition \ref{def:psub} may be chosen in terms of which $Y$ is given locally by the
vanishing of $x'= (x_1, \ldots, x_r)$ ($r\geq 1$) and of $y'=(y_1, \ldots, y_m)$, and $\Ytilde$ by the vanishing of $y'$.

\begin{definition} \label{def:submfd finite order}
Let $X$ be a manifold with corners and $a\in\NN$. Let $Y$ be a boundary p-submanifold of $X$.
We say that two interior extensions $\Ytilde,\Ytilde'$ of $Y$ {\em agree to order $a$} if in any local coordinate system for which $\Ytilde$ is given by the vanishing of $y'$ as above one has,
with $y''=(x_{r+1},\dots,x_k,y_{m+1},\dots,y_{n-k})$,
\begin{equation}
\label{eq:Ytilde' local coord}
\Ytilde' = \{(x',y',y''):\, y'= \sum_{\alpha\in \NN_0^r,|\alpha|=a} (x')^{\alpha} F_\alpha(x',y'')\}\ \text{ locally }
\end{equation}
for some smooth functions $F_\alpha$.

An {\em interior extension of $Y$ to order $a$} is an equivalence class of interior extensions of $Y$ under this equivalence relation, and  $Y$ together with an extension to order $a$ is called  a {\em submanifold of the boundary to order $a$}.
\end{definition}

The case $a=1$ corresponds to a `usual' boundary p-submanifold.
If $Y$ is a p-submanifold of
$X$ then there is a unique smallest face $F=F(Y)$ of $X$ containing $Y$. A p-submanifold $Y$ may be characterized by the ideal $\calI(Y)\subset C^\infty(X)$ of functions vanishing on $Y$. This may easily extended to submanifolds to order $a$ as follows. It follows from \eqref{eq:Ytilde' local coord} that interior extensions $\Ytilde,\Ytilde'$ of $Y$ agree to order $a$ if and only if
$ \calI(\Ytilde) + \calI(F)^a = \calI(\Ytilde') + \calI(F)^a$ where $F=F(Y)$, so we define:
\begin{definition} \label{def:ideal order a}
Let $Y_a$ be a boundary submanifold of $X$ to order $a$. The {\em ideal of functions defining $Y_a$} is
$$ \calI(Y_a) := \calI(\Ytilde) + \calI(F)^a$$
where $\Ytilde\subset X$ is any representative of $Y_a$ and $F=F(Y)$.
\end{definition}
Near any point of $Y$ where $\Ytilde$ is given in terms of the local coordinates by
$y_1= \cdots= y_m=0$,
we have
\begin{equation}
\label{eq:order a ideal locally}
\calI(Y_a)_{|U} = \Span_{C^\infty(X)_{|U}} \{y_1,\dots,y_m\}\cup \{(x')^\alpha:\, |\alpha|=a\}
\end{equation}
while near any interior point $\calI(Y_a)$ is trivial, $\calI(Y_a)_{|U} = C^\infty(X)_{|U}$.

If $Y \subset \del X$ is a closed submanifold and we fix an interior extension $Y_a$ to order $a$,
then the natural notion of the blow-up in $X$ of $Y$ to order $a$ at the boundary, is a {\em quasi-homogeneous blow-up}: It resolves the ideal $\calI(Y_a)$, and points on the front face correspond to order $a+1$ information at $Y_a$. To define it, we first consider a local coordinate neighborhood $U$ of a boundary point in which \eqref{eq:order a ideal locally} holds. The guiding idea is to consider the variables as having weights (or degrees of homogeneity), where $y_1,\dots,y_m$ have weight $a$ and $x_1,\dots,x_r$ have weight $1$ (while $y''$ has weight $0$). Then a monomial $(x')^\alpha y^I$ has weight $|\alpha| + a(I_1+\dots+I_m)$, and a function is in $\calI(Y_a)_{|U}$ if and only if all monomials in its Taylor expansion around $Y$ have weight at least $a$.

Let $r_a:\RR_+^r\times \RR^{m}\to\RR_+$, $r_a(x',y') = (x_1^{2a}+\dots+x_r^{2a} + y_1^2 +\dots + y_m^2)^{1/2a}$. The function $r_a$ is weighted homogeneous of degree $1$. Let $S_a^+ = \{(\omega,\nu)\in\RR_+^r\times\RR^m:\, r_a(\omega,\nu)=1\}$ be an `octant' of the (weighted) sphere. Over the neighorhood
$U=r_a^{-1}([0,\eps)) \times U''$, $U'' \subset \RR_+^{k-r}\times\RR^{n-k-m}$, we define the
{\em blow-up of} $Y$ {\em
to order} $a$ {\em in} $X$ {\em associated to the extension} $Y_a$, and its blow-down map, by
\begin{equation}
\label{eq:def weighted blow-up local}
\begin{aligned}{}
\beta_a: [U,Y_a\cap U]  := \left(S_a^+ \times [0,\eps)\right) \times U'' & \to U \\
((\omega,\nu), s;y'') & \mapsto (s\omega,s^a\nu,y'').
\end{aligned}
\end{equation}
As usual it is easy to check that this is defined invariantly, i.e.\ a diffeomorphism of $U$ preserving $Y_a$ induces a diffeomorphism of $[U,Y_a\cap U]$, and therefore patches together to define a manifold with corners $[X;Y_a]$ with blow-down map $\beta_a:[X;Y_a]\to X$. The front face is, as usual, defined as $\ff = \beta_a^{-1} (Y)$, and locally is given by $(S_a^+ \times \{0\})\times U''$.
Points on $\ff$ correspond to equivalence classes of curves ending at $Y$ and tangent to order $a-1$ to an extension $\Ytilde$ in the equivalence class defining $Y_a$. Two such curves are equivalent if they have the same endpoint and if their $a$th derivatives at the endpoint agree modulo a vector tangent to $Y$.

Most convenient for calculations are (weighted) projective coordinates on $[X;Y_a]$.
For each $i=1,\dots,r$  these are given on $x_i>0$ by
\begin{equation}
\label{eq:proj coord quasihom ff}
x_i \text{ (defining $\ff$)},\ W_j= \frac{x_j}{x_i},\ V = \frac{y'}{x_i^a},\ y''
\end{equation}
(use the coordinates $W_j=\omega_j/\omega_i$, $V= \nu/\omega_i^a$ on $S_a^+ \setminus \{\omega_i=0\}$)
in terms of which, say for $i=1$, $\beta_a(x_1, W_2,\dots,W_r,V,y'') = (x_1,x_1 W_2,\dots,x_1 W_r,x_1^a V, y'')$,
and for each $p=1,\dots,m$ they are given on $y_p\neq 0 $  by
\begin{equation}
\label{eq:proj coord quasihom bd}
r=y_p^{1/a} \text{ (defining $\ff$)},\ \xi'=\frac{x'}{y_p^{1/a}},\ \eta_q=\frac{y_q}{y_p}\ \ (q=1,\dots,m,\ q\neq p),\ y''
\end{equation}
(use local coordinates $\frac{\omega}{\nu_p^{1/a}} = \xi'$,
$\frac{\nu_q}{\nu_p}=\eta_q$ on $S_a^+\setminus \{\nu_p=0\}$, and $s\nu_p^{1/a}= r$ as bdf), in terms of which $\beta_a (r,\xi',\eta_2,\dots,\eta_m,y'') = (r\xi', r^a, r^a \eta_2,\dots,r^a\eta_m,y'')$ (say for $p=1$).
The coordinates \eqref{eq:proj coord quasihom ff} extend, as coordinates, smoothly to the front face outside the strict transform of the hypersurface $x_i=0$ of $X$ (the strict transform of a hypersurface $H$ under blowup of $Y$ is $\clos(\beta^{-1}(H\setminus Y))$), and similarly the coordinates \eqref{eq:proj coord quasihom bd} outside the strict transform of $y_p=0$.
This shows that
$\beta_a$ is a b-map whose exponents are all $0$ or $1$ just as in the case of homogeneous blow-ups. The power $a$ occurs only in the $\beta_a^*y'$, i.e.\ with respect to the {\em interior} submanifolds $y_p=0$ ($p=1,\dots,m$).
Also, this shows that
\begin{equation}
\label{eq:pull-back generator}
\beta_a^* \calI(Y_a) = \Span_{C^\infty([X;Y_a ])} \{(r_\ff)^a\}
\end{equation}
where $r_\ff$ is any bdf of $\ff$.

In terms of the coordinates \eqref{eq:proj coord quasihom ff} we calculate pull-backs of some vector fields:
\begin{equation}\label{eq:vf pullback}
\begin{gathered}
\beta_a^*(x_i \partial_{x_i}) = x_i \partial_{x_i} - \sum_j W_j \partial_{W_j} - a V \partial_V \\
\beta_a^*(x_j \partial_{x_j}) = W_j \partial_{W_j},\quad
\beta_a^*(x_i^a \partial_{y'}) = \partial_V
\end{gathered}
\end{equation}

Throughout this paper, we simplify notation by writing simply $Y$ instead of $Y_a$ for a submanifold of $X$ defined to order $a$, and we will denote
by $[X:Y]_{a}$ the quasihomogeneous blowup of $Y$ of degree $a$.  When $a=1$ this is the standard blow-up, simply denoted $[X;Y]$. We will denote by $[X; Y_1, \ldots, Y_k]_{a}$
a sequence of blowups all of the same degree.

\medskip
\noindent{\bf Lifts of submanifolds.} Recall that if $Y$ and $V$ are closed p-submanifolds of $X$ then the {\em lift} of $Y$ under the blow-down map $\beta:[X;V]\to X$, denoted $\beta^*(Y)$, is defined to be $\beta^{-1}(Y)$ if $Y\subset V$, and as the closure $\cl(\beta^{-1}(Y\setminus V))$ if $Y=\cl(Y\setminus V)$ (and undefined otherwise; one of the two cases will always occur in this paper). If $Y_a$ and $V_b$ are boundary submanifolds defined to orders $a,b$, respectively, then this has to be modified. We consider only the case of interest here: If $a>b$ and $Y_a\subset V_b$ in the sense that any interior extension $\Ytilde$ of $Y_a$ is contained in some interior extension $\Vtilde$ of $V_b$ then the lift of $Y_a$ under $\beta:[X;V_b]\to X$ is defined as the boundary submanifold with interior extension $\beta^*(\Ytilde) =\cl(\beta^{-1}(\Ytilde\setminus V))$ for any such $\Ytilde$. For example, if $Y=V$ is defined to order $a$ but blown up only to order $b<a$, then $\beta_b^*Y_a$ is a submanifold to order $a-b$, of the same dimension as $Y$. See \cite{GH2} for a more thorough discussion of this.

\subsection{The geometry of the \bfa-double space}
As in previous b-type calculi, the operators of the \bfa-pseudodifferential calculus will be
defined in terms of their integral kernels, which will live on a space $M^2_z$, which
is obtained from the double space $M^2$ by a sequence of blowups.  Below we depart
from the names given to these blown up spaces and their various submanifolds in
previous works in the area, eg \cite{Me-aps} and \cite{MaMe}, as their somewhat
ad hoc names do not generalize well to our more complicated geometric setting.

Recall that we assume a trivialization of a neighborhood of the boundary of $M$ is fixed, and that the extensions $\Phi_i$ of $\phi_i$ defining ${}^\bfa\calV(M)$ have product structure as in \eqref{eq:product structure}. This also fixes the boundary defining function $x$. This assumption is made for convenience. The double space constructed below is actually independent of such a choice, as is shown in the appendix.

Summarizing the discussion below, we introduce the blowups $M^2_z\stackrel{\beta_z}\to M^2_y \stackrel{\beta_y}\to M^2_x \stackrel{\beta_x}\to M^2$ accomplished by blowing up the lifts of the `partial' (or fibre) diagonals $\Delta_z$, $\Delta_y$, $\Delta_x$ intersected with the boundary. In local coordinates, these (interior) fibre diagonals are given by
$$\Delta_x=\{x=x'\},\quad \Delta_y=\{x=x',\,y=y'\},\quad \Delta_z=\{x=x',\,y=y',\,z=z'\}$$
This is the origin of the notation.

\medskip
\noindent{\bf The b-blowup (or $x$-blowup), $M^2_x$}.
The first blowup we do is the standard b-blowup from \cite{Me-aps}. Denote by $x$ the boundary defining function on $M$, and also its lift to the first factor in $M\times M$, and by $x'$ its lift to the second factor (a similar convention will be used for coordinates $y_i$ etc.). Then $M^2$ has a corner $\dM \times \dM$ at the intersection of the left face,
$\lf = \{x=0\}=\dM\times M$, and the right face, $\rf = \{x^\prime=0 \}=M\times\dM$, and we blow this up to obtain
$$
M^2_x = [M^2; \dM \times \dM] \stackrel{\beta_x}\to M^2.
$$
The space $M^2_x$ has three boundary hypersurfaces: $\lf$, $\rf$ and $\ff_x$.
The front face $\ff_x$ is the preimage $\beta_x^{-1}(\dM\times\dM)$, while
the faces $\lf$ and $\rf$ are the lifts $\beta_x^*(\lf)$, $\beta_x^*(\rf)$.
For example, $\lf = \mbox{cl}(\beta_x^{-1}(\dM \times \stackrel{o}{M}))$. (In \cite{Me-aps} $M^2_x$ is called $X^2_b$ and $\ff_x$ is called bf.)
In terms of coordinates, this blowup is accomplished by changing coordinates to:
$$
t = \frac{x^\prime}{x},\, x, y_i, y_i^\prime,z_j, z_j^\prime,w_k, w_k^\prime
$$
in terms of which $\beta_x(x,t,\dots)=(x,x'=tx,\dots)$. These coordinates are valid outside $\lf$, with similar coordinates (switching $x,x'$) valid outside $\rf$.
\ownremark{One may be tempted to think that these coordinates make sense only in a neighborhood of the diagonal, since both $y_i$ and $y_i'$ are defined only on the same `small' coordinate neighorhood. However, given any two points $p,p'$ one may find a coordinate neighborhood containing both (possibly disconnected), so this is no issue.}
In these coordinates the front face is given by $x=0$, hence one has a diffeomorphism
\begin{equation}
\label{eq:ffx diffeo}
\interior{\ff_x} \cong \dM \times \dM \times (0, \infty),
\end{equation}
with $t$ the coordinate in $(0,\infty)$.

In addition
to the boundary hypersurfaces, there is another important submanifold in
$M^2_x$, namely the lifted diagonal, which we denote by
$\Delta$ just like the diagonal  $\Delta=\Delta_M=\{(p,p):\,p\in M\}\subset M\times M$. In coordinates,
$$ \Delta = \{t=1,\ y=y',\ z=z',\ w=w'\}.$$

\medskip
\noindent{\bf The $y$-blowup, $M^2_y$.}
The second blowup is similar to the one used in the $\phi$ calculus, discussed in
\cite{MaMe} and \cite{Va}. Let $\Delta_y=\{(p,p'):\, \Phi_{2,1}(p) = \Phi_{2,1}(p')\}\subset M\times M$ be the fibre diagonal. Consider its lift under $\beta_x$, which we will denote in the same way. In local coordinates, $\Phi_{2,1}(x,y,z,w)=(x,y)$, so $\Delta_y=\{x=x',\ y=y'\}$ and its lift is $\{t=1,\ y=y'\}$. The boundary of this manifold is its intersection with the front face,
$$\partial\Delta_y = \Delta_y\cap \ff_x $$
and is given locally by $\{t=1,\ x=0,\ y=y'\}$. In terms of the identification \eqref{eq:ffx diffeo}, $\partial\Delta_y$ is the fibre diagonal of $\phi_{2,1}$, a subset of $\dM\times\dM$, times $\{1\}\subset (0,\infty)$.
In the $\phi$-calculus, $\partial\Delta_y$ is blown up in the standard homogeneous way.
In the ${\bf a}$-calculus, this blowup
will be the order $a_1$ quasi-homogeneous blowup.  This is defined only if an extension of $\partial\Delta_y$ to the interior to order $a_1$ is given. Here this extension is $\Delta_y$. We obtain
$$ M^2_y = [M^2_x; \partial\Delta_y]_{a_1} \stackrel{\beta_y} \longrightarrow M^2_x.$$
This blowup generates a new hypersurface, $\ff_y = \beta_y^{-1}(\partial\Delta_y)$.
In terms of coordinates, this blowup is accomplished
by changing near
$\ff_x \cap \Delta$ to the coordinates
$$
T=\frac{1-t}{x^{a_1}}, \, x,Y= \frac{y-y^\prime}{x^{a_1}},\, y,z, z^\prime, w, w^\prime,
$$
valid near the interior of $\ff_y$, which is given by $x=0$. Here $T\in\RR$, $Y\in\RR^b$. As $|T|+|Y|\to\infty$ one approaches the boundary of $\ff_y$.
As in the $\phi$-calculus
case 
we have a diffeomorphism
\begin{equation}
\label{eq:ffy diffeo}
\interior{\ff_y} \cong \ {}^{\phi}N\dM \times_{B_0} \dM,
\end{equation}
which is a fibre bundle over $B_0$ with the
fibre over $y\in B_0$  diffeomorphic to $(\phi_{2,1} \times \phi_{2,1})^{-1}(y,y) \times \RR^{b+1}$, and $(T,Y_1,\dots,Y_b)$ are the coordinates on $\RR^{b+1}$ with respect to the basis $x^{1+a_1}\partial_x$, $x^{a_1}\partial_{y_1},\dots,x^{a_1}\partial_{y_b}$, see the proof of Proposition \ref{prop:a front face general} in the Appendix.

In addition to the new blowup hypersurface, $M^2_y$ has the boundary hypersurfaces $\ff_{yx}=\beta_y^*(\ff_x)$, $\lf = \beta_y^* (\lf)$ and $\rf=\beta_y^*(\rf)$ (the latter two are actually simple preimages since $\lf,\rf$ are disjoint from $\partial\Delta_y$). The lift of the diagonal is in coordinates
$$\Delta = \{T=0,\ Y=0,\ z=z',\ w=w'\}.$$

\medskip
\noindent{\bf The $z$-blowup, $M^2_z$.}
The last blowup is completely analogous to the $y$-blowup, except that we start with $M^2_y$ and with the
`smaller' fibre diagonal $\Delta_z = \{(p,p'):\, \Phi_2(p)=\Phi_2(p')\}\subset M^2$. The lift of $\Delta_z$ to $M^2_x$ and then to $M^2_y$ is denoted in the same way, and in local coordinates is given by $\{T=0,\ Y=0,\ z=z'\}$. Its boundary is the intersection with the front face,
$$ \partial\Delta_z = \Delta_z\cap\ff_y$$
and is given locally by $\{x=0,T=0,\ Y=0,\ z=z' \}$. In terms of \eqref{eq:ffy diffeo}, $\partial\Delta_z$ is the subset of the zero section (i.e.\ $T=Y=0$) given by the fibre diagonal of $\phi_2$, when this zero section is identified with $\dM\times_{B_0}\dM\subset \dM\times\dM$. Now we perform the order $a_2$ quasi-homogeneous blowup of $\partial\Delta_z$, again with respect to the extension $\Delta_z$ of $\partial\Delta_z$. We obtain
$$ M^2_z = [M^2_y; \partial\Delta_z]_{a_2} \stackrel{\beta_z}\longrightarrow M^2_y.$$
This blowup generates a new hypersurface, $\ff_z = \beta_z^{-1}(\partial\Delta_z)$.
In terms of coordinates, this blowup is accomplished
by changing near $\ff_y \cap \Delta$ to the coordinates
$$
\mathcal{T}=\frac{T}{x^{a_2}}, \, x,\mathcal{Y}=\frac{Y}{x^{a_2}}, \, y,
\mathcal{Z}=\frac{z-z^\prime}{x^{a_2}}, \, z, w, w^\prime,
$$
valid near the interior of $\ff_z$, which is given by $x=0$. Here $\calT\in\RR$, $\calY\in\RR^b$, $\calZ\in\RR^{f_1}$. As $|\calT|+|\calY|+|\calZ|\to\infty$ one approaches the boundary of $\ff_z$. One has a diffeomorphism
\begin{equation}
\label{eq:ffz diffeo}
\interior{\ff_z} \cong \ {}^{\bfa}N\del M \times_{B_1} \del M,
\end{equation}
which is a fibre bundle over $B_1$ with the
fibre over $p\in B_1$  diffeomorphic to $F_2 \times F_2 \times \RR^{f_1+b+1}$ and $\calT,\calY,\calZ$ coordinates on $\RR^{1+b+f_1}$ corresponding to the basis $x^{1+a_1+a_2}\partial_x$, $x^{a_1+a_2}\partial_y$, $x^{a_2}\partial_z$. The diffeomorphisms \eqref{eq:ffx diffeo}, \eqref{eq:ffy diffeo} and \eqref{eq:ffz diffeo} are canonical when fixing an $(a_1,a_2+1)$-structure refining ${}^\bfa\calV(M)$, see Proposition \ref{prop:a front face general} in the Appendix.

In addition to the new blowup hypersurface, $M^2_z$ has the boundary hypersurfaces
$\ff_{zy}$, $\ff_{zx}$, $\lf$ and $\rf$, which are the lifts of
$\ff_y$, $\ff_{yx}$, $\lf$ and $\rf$ respectively. The latter three are disjoint from $\partial\Delta_z$, so their lifts are simple preimages. Finally, the lifted diagonal is, in coordinates, $$\Delta=\{\calT=0,\calY=0,\calZ=0,w=w'\}.$$

We denote the total blow-down map by
$$ \beta_\bfa = \beta_x\circ \beta_y \circ \beta_z : M^2_z \to M^2.$$

\begin{figure}
\begin{center}
\includegraphics[width=10cm, height=10cm]{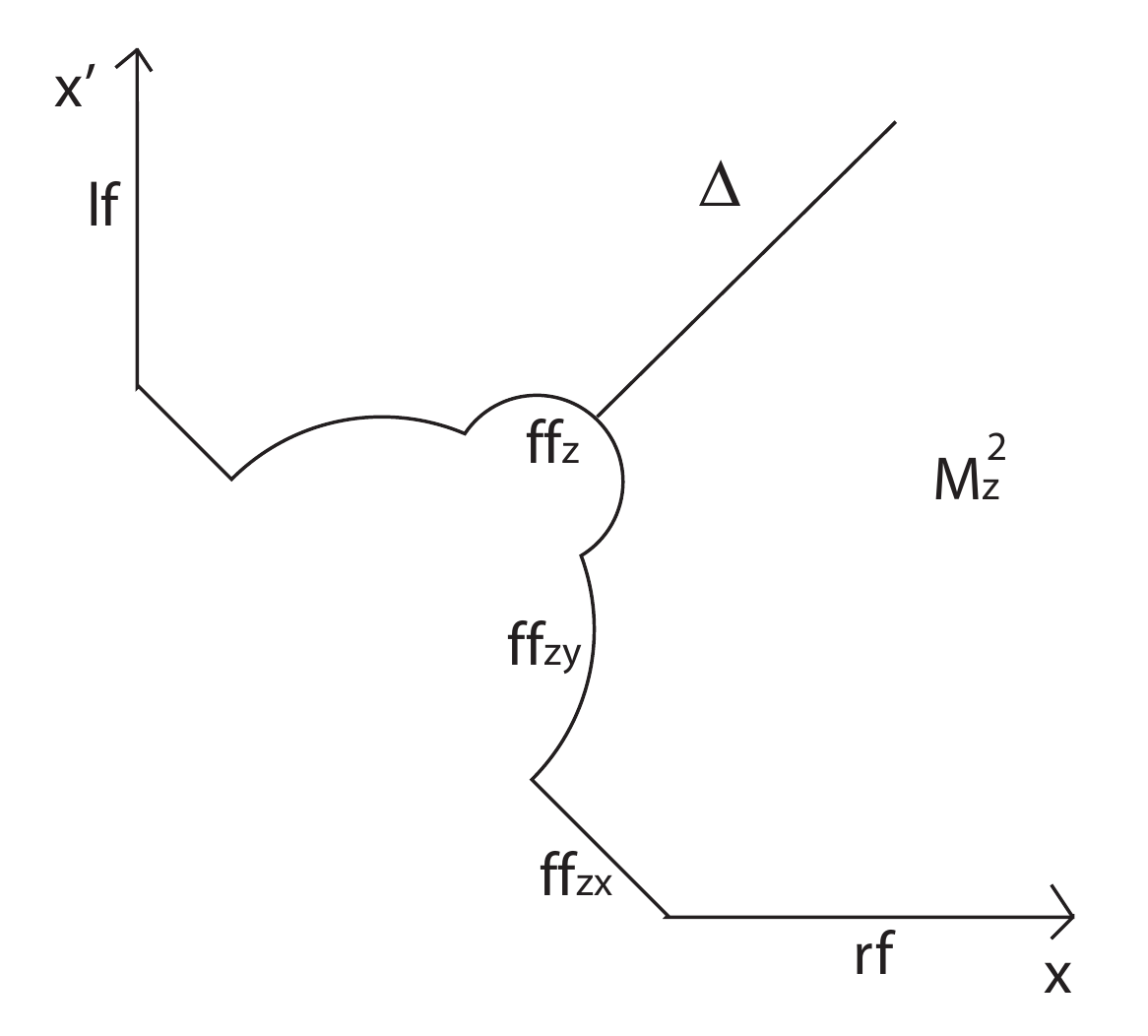}
\end{center}
\caption{The \bfa-double space}
\label{fig:doublespace}
\end{figure}

Instead of constructing the \bfa-double space $M^2_z$ as described above using quasi-homogeneous blowups one may construct a closely related double space, $\tilde{M^2_z}$, using only regular homogeneous blow-ups: The order $a_1$ blowup of $\partial\Delta_y$ is replaced by a sequence of $a_1$ regular blowups, and similarly for the blowup of $\partial\Delta_z$. The finite order extensions of these boundary submanifolds are then needed to have well-defined centers of blowup for the second, third etc. blowup in the sequence.  One way to think about this is to add additional fibrations to the setup, one for each positive integer $\leq a_1+a_2$, with all but those numbered $a_1$ and $a_1+a_2$ being trivial. The resulting space is 'bigger' than $M^2_z$, i.e.\ the identity on the interior extends to a smooth map $\tilde{M^2_z}\to M^2_z$, but not vice versa (unless $a_1=a_2=1$, in which case they are equal). This construction yields the same small calculus $\Psi_\bfa^*(M)$, but a bigger large calculus (see Definition \ref{def.full calculus}). It has the advantage of using only the standard blowup machinery and does not require the considerations in the Appendix.

\subsection{Lifts of \bfa-vector fields and \bfa-differential operator kernels}

We can begin to see the purpose of these blowups by considering the lifts of the vector
fields in $^{{\bf a}}\mathcal{V}(M)$ to these various spaces.  We have the following lemma:

\begin{lemma}
\label{lemma.vectorfields}
Let $p \in \dM$ be a point on the boundary of $M$ and let $x, y_i, z_j, w_k$ be local
coordinates near $p$.  Then all of the basis vector fields $x^{1+a_1+a_2} \dx,
x^{a_1+a_2} \dyi, x^{a_2} \dzj, \dwk$ lift via left or right projection and $\beta_\bfa$ to smooth vector fields on $M^2_z$, and the span of their lifts is transversal to the lifted diagonal $\Delta$.
\end{lemma}
\ownremark{this is stronger than each vector field separately being transversal, and we need this stronger version}
The point is that the transversality holds uniformly up to  the boundary $x=0$.

\begin{proof}
By symmetry it suffices to consider the left projection  $\pi_l:M^2 \rightarrow M,\ (p,p')\mapsto p$
which is given locally by
$$
\pi_l(x,x^\prime, y, y^\prime, z, z^\prime, w, w^\prime)= (x, y, z, w).
$$
Thus $\pi_l^*( \dx) = \partial_{x}$, and so forth.
The blowdown map $\beta_x$ acts by
$\beta_x(t,x,etc.)= (x, tx, etc.)$, that is, $x^\prime = tx$.  Thus (compare \eqref{eq:vf pullback})
$$
\beta_x^*(x\partial_x) = x\partial_x - t\partial_t
$$
while the other vector fields lift to themselves, essentially, i.e. $\beta_x^* (\partial_{y_i}) = \partial_{y_i}$ etc.

Now at the second step,
$T=(1-t)/x^{a_1}$, $Y= (y-y')/x^{a_1}$ so using \eqref{eq:vf pullback} again we get
$$
\beta_y^* ( x\partial_x) = x\partial_x - a_1 T\partial_T - a_1 Y\partial_Y
$$
$$
\beta_y^*( x^{a_1}\partial_t) = -\partial_T,\quad \beta_y^*( x^{a_1}\partial_y) = \partial_Y
$$
Again, the other vector fields lift to themselves.
This gives the intermediate result (use $\beta_x^*x=x$, $\beta_y^*x=x$)
\begin{equation}
\label{eq:vf pullback phi}
\beta_y^*\beta_x^* (x^{1+a_1}\partial_x) = \beta_y^*(x^{a_1}(x\partial_x-t\partial_t)) = \partial_T, \quad
\beta_y^*\beta_x^*(x^{a_1}\partial_y) = \partial_Y\ \text{ at }\ff_y
\end{equation}
since $x=0, t=1$ there.

Finally, at the third step,
$\calT=T/x^{a_2}$, $\calY = Y/x^{a_2}$, $\calZ= (z-z')/x^{a_2}$ so
$$
\beta_z^* (x\partial_x) = x\partial_x - a_2 \calT\partial_\calT - a_2\calY\partial_\calY - a_2\calZ\partial_\calZ
$$
$$
\beta_z^*(x^{a_2}\partial_T) = \partial_\calT,\quad
\beta_z^*(x^{a_2}\partial_Y) = \partial_\calY,\quad
\beta_z^*(x^{a_2}\partial_z) = \partial_\calZ
$$
Once again, the $w$ vector fields lift to themselves.
Putting everything together we get that the pullbacks under $\beta_\bfa^*=\beta_z^*\beta_y^*\beta_x^*$ of
$x^{1+a_1+a_2}\partial_x$, $x^{a_1+a_2}\partial_y$, $x^{a_2}\partial_z$ are smooth and, at $\ff_z$, equal to $\partial_\calT$, $\partial_\calY$, $\partial_\calZ$, respectively, so the span of these together with $\partial_{w}$ is transversal to
$\Delta = \{\calT=\calY=\calZ=0, w=w^\prime\}$.

\end{proof}

The calculation in the proof of the lemma implies the following corollary. Recall the discussion on differential operators acting on half-densities at the end of Subsection \ref{subsec:double}.

\begin{lemma}
\label{lemma.bdiff}
If $P$ is an ${\bf a}$-differential operator of degree $m$ on $M$, acting on smooth half-densities, then its integral kernel lifts under $\beta_\bfa$ to a
smooth Dirac \bfa-half-density on $M^2_z$ of order $m$ with respect to the diagonal $\Delta$. That is, in local
coordinates it is given by
\begin{multline}
\label{eq:diff kernel}
K_P' = \sum_{\alpha + |I| + |J| +|K| \leq m}
b_{\alpha, I, J,K}(x, y, z,w) (D^\alpha \delta)(\mathcal{T})\, (D^I\delta)(\mathcal{Y})\, (D^J\delta)(\mathcal{Z})\, (D^K\delta)(w - w') \cdot \\
\left| \frac {dx}{x^{1+a_1+a_2}} \prod_{i}\frac{dy_i}{x^{a_1+a_2}}\prod_j \frac{dz_j}{x^{a_2}} \prod_k dw_k \,\, d\calT d\calY d\calZ dw' \right|^{1/2}.
\end{multline}
If $P$ is as in \eqref{eq:difflocal} and functions $f$ are identified with half-densities $f|dxdydzdw|^{1/2}$ then $b_{\alpha, I, J,K}=a_{\alpha, I, J,K}$ at $x=0$.
\end{lemma}
\begin{proof}
The kernel of the identity operator on half-densities is
$\delta(x-x')\delta(y-y')\delta(z-z')\delta(w-w')\, |dxdx'\,dydy'\,dzdz'\,dwdw'|^{1/2}$
which under change of coordinates becomes $$\delta(\calT)\delta(\calY)\delta(\calZ)\delta(w-w')
x^{-\gamma} \left| dx\,dx' \prod_i dy_i\,dy_i'\,\prod_j dz_j\,dz_j'\,\prod_k dw_k\,dw_k' \right|^{1/2}$$
where $\gamma=(1+a_1+a_2) + b(a_1+a_2) + f_1a_2$. This half density factor equals the one in \eqref{eq:diff kernel}. Now apply the lifted
vector fields to obtain the lifted kernel of $P$. Identifying functions with half-densities as stated means that the $|dx|$ etc. terms are not differentiated. Since $x^{1+a_1+a_2}\partial_x$, $x^{a_1+a_2}\partial_y$, $x^{a_2}\partial_z$ lift to $\partial_\calT$, $\partial_\calY$, $\partial_\calZ$ at $\ff_z$, respectively, the result follows.
\end{proof}

\subsection{The small \bfa-calculus, the \bfa-principal symbol and the \bfa-normal operator}
In this subsection we define the small calculus of $\bfa$-pseudodifferential operators.
We then introduce the $\bfa$-principal symbol of an operator $P$ in the small calculus. It describes the singularity of the Schwartz kernel of $P$ at the diagonal, as in the case of standard pseudodifferential operators.
However, unlike in the classical case, invertibility of the \bfa-principal symbol is not sufficient to guarantee that $P$ is a Fredholm operator. Therefore we define a second symbol, called the \bfa-normal operator, which describes the behavior of the operator at the boundary $x=0$. In Section \ref{sec:param} we will see that invertibility of both symbols, called full ellipticity, implies existence of a parametrix with compact remainder, hence Fredholmness.

\medskip
\noindent{\bf The small \bfa-calculus.}
The definition of the small calculus is modelled on \eqref{eq:diff kernel}, by replacing the Dirac distribution by an arbitrary distribution conormal with respect to the diagonal and imposing the condition of rapid decay as $|(\calT,\calY,\calZ)|\to\infty$. It is essential to keep the $x$-factor in the half-density. To state this invariantly, we define the half-density bundle
\begin{equation}
\label{eq:def omegatilde}
\Omegatilde_\bfa^{1/2}(M^2_z) = \beta_\bfa^*(\Omega_{\bfa,l}^{1/2}(M) \otimes \Omega_{\bfa,r}^{1/2}(M))
\end{equation}
where $\Omega_{\bfa,l/r}^{1/2}(M)$ are the $\bfa$-half-density bundles on $M$ pulled back to the left resp. right factor of $M^2$. This has the half density factor in \eqref{eq:diff kernel} as smooth non-vanishing section in a neighborhood of $\intff_z$. \ownremark{ but not at $\lf$ and $\rf$ because of the extra factor of $t$ occuring when replacing have of the $x$ factors by $x'$ factors}
\begin{definition}
\label{def:small calculus}
The small \bfa-calculus, $\Psi^m_a(M)$, of \bfa-pseudodifferential operators of
degree $m\in\RR$ acting on half-densities on $M$ consists of those operators whose kernels, when lifted to $M^2_z$,  have values in $\Omegatilde_\bfa^{1/2}(M_z^2)$ and vanish to all orders at the faces $\lf$, $\rf$, $\ff_{zx}$ and $\ff_{zy}$ and are conormal of degree $m$ with respect to $\Delta$, smoothly up to $\ff_z$; that is, kernels in $I^{m,\mathcal{G}}(M^2_z,\Delta; \Omegatilde_\bfa^{1/2}(M^2_z)),$ where $G(\lf)=G(\rf)=G(\ff_{zx})=G(\ff_{zy}) = \emptyset$ and $G(\ff_z)= 0$.
\end{definition}
We will denote both the kernel of $P$ and its lift to $M^2_z$ by $K_P$.
As with the space of conormal distributions, this space of operators corresponds to the so-called classical or one-step polyhomogeneous pseudodifferential operators whose symbols have complete asymptotic expansions rather than just symbol bounds.

If $P\in \Psi^m_\bfa(M)$ has kernel $K_P'(x,\calT,\dots)\left| \frac {dx}{x^{1+a_1+a_2}} \prod_{i}\frac{dy_i}{x^{a_1+a_2}}\prod_j \frac{dz_j}{x^{a_2}} \prod_k dw_k \,\, d\calT d\calY d\calZ dw' \right|^{1/2}$ in local coordinates
and $u = u(x,y,z,w)|dxdydzdw|^{1/2}$ is a smooth half-density on $M$, supported in $\Phi_2^{-1}(U)$ for $U$ a coordinate neighborhood in $B_1\times[0,1)$,
we get the following expression for $Pu$ in $U$:

\begin{multline}
Pu(x,y,z,w) =  \bigg( \int_{\rm{fib}} K_P'(x,\calT, y, \calY, z, \calZ, w, w')
u(x-x^{1+a_1+a_2}\calT, y- x^{a_1+a_2}\calY,z- x^{a_2}\calZ, w') \cdot
\\
\label{eq:Plocal}
d\calT d\calY d\calZ dw' \bigg) \ \mid  dx dy dz dw \mid^{1/2},
\end{multline}
where ${\rm fib}$ are the fibres of the right projection
to $M$ and $d\calY = d\calY_1 \cdots d\calY_b$
and so on.  The condition that $K_P$ vanish to all orders at all faces except $\ff_z$ means
that $K'_P$ and its derivatives must vanish faster than any negative power of $|(\calT, \calY, \calZ)|$ as this norm
goes to infinity.
\medskip

\noindent{\bf The \bfa-principal symbol.}
The \bfa-principal symbol map is analogous to the symbol map in the standard pseudodifferential
operator calculus, as well as to the symbol maps in the b and $\phi$-calculi.
Let $S^{[m]}({}^\bfa T^*M)$ be the space of smooth functions on ${}^\bfa T^*M \smallsetminus 0$ that
are homogeneous of degree $m$ in the fibre.  Then the  \bfa-principal symbol is a map
$$
{}^\bfa \sigma_m: \Psi^m_{\bfa}(M) \longrightarrow S^{[m]}({}^\bfa T^*M),
$$
defined as in the compact setting: By the remark after Definition
\ref{def:conorm}, which also applies to the boundary case, the
principal symbol of a distribution conormal with respect to the
diagonal $\Delta$ lives on the conormal bundle of $\Delta$. This
bundle can be identified, via the map $\beta_\bfa$ followed by left
projection, with ${}^\bfa T^*M$.  In the case that $P \in
\Psi^m_\bfa(M)$ is an \bfa-differential  operator, we have the
standard formula away from $\ff_z$ and
\begin{equation}
\label{eq:smallsymb}
{}^\bfa \sigma_m(P) = \sum_{\alpha + |I| + |J| +|K| = m}
a_{\alpha, I, J,K}(x, y, z,w) \tau^\alpha \eta^I \zeta^J
\theta^K,
\end{equation}
near $\ff_z$.  As in  the compact case, one has a short exact
sequence and asymptotic completeness:
\begin{lemma}
\label{lem.short exact - sigma}
For each $m\in\RR$ the sequence
\begin{equation}
\label{eq:smallsymbses}
0  \longrightarrow \Psi^{m-1}_{\bfa}(M)\hookrightarrow \Psi^m_{\bfa}(M) \stackrel{{}^\bfa \sigma_m}{\longrightarrow} S^{[m]}(^aT^*M)
\longrightarrow 0
\end{equation}
is exact. Furthermore, any sequence $P_k\in \Psi^{m-k}_\bfa(M)$ can be asymptotically summed, i.e.\ there is $P\in \Psi^m_{\bfa}(M) $ with $P-\sum_{k=0}^{N-1}P_k \in \Psi^{m-N}_{\bfa}(M)$ for all $N$.
\end{lemma}
The construction of the \bfa-principal symbol and the proof of the lemma are essentially the
same in this setting as in the b-calculus, see \cite{Me-aps}. After we discuss composition we will see that Lemma \ref{lem.short exact - sigma} and Lemma \ref{lem:symbolcomp} imply the existence of a parametrix, up to an error in $\Psi^{-\infty}_{\bfa}(M)$, of \bfa-operators with invertible principal symbol. However, these error terms are not compact on $L^2$. The obstruction to compactness is the normal operator, which we now discuss.

\medskip
\noindent{\bf The \bfa-normal operator.}
The normal operator $N_\bfa(P)$ of $P\in \Psi_\bfa^m(M)$ encodes the leading behavior of $P$ at $\partial M$ and serves as a second (non-commutative) symbol of $P$. We will give it in local coordinates, in terms of its Schwartz kernel, as operator on a suitable space and in terms of a family of operators on the fibers $F_2$, the normal family.

For an \bfa-differential operator, given in coordinates as in \eqref{eq:difflocal}, we define the normal operator as
\begin{equation}
N_\bfa(P) = \sum_{\alpha + |I| + |J| + |K| \leq m}
a_{\alpha, I, J,K}(0, y, z,w)\, D_\calT^\alpha D_\calY^I D_\calZ^J D_w^K.
\label{eq:indlocal}
\end{equation}
That is, we set $x=0$, replace differentiation in $x,y,z$ by differentiation in dummy variables $\calT,\calY,\calZ$ (leaving out the $x$-scaling factors) while regarding $y,z$ (the $B_1$-variables) as parameters. Differentiation in $w$ (i.e. in the fibre $F_2$) remains as in $P$.
We need to make sense of this invariantly and extend it to \bfa-pseudodifferential operators.

This is done by considering the Schwartz kernel.
Since \eqref{eq:indlocal} has constant coefficients in $\calT,\calY,\calZ$, its Schwartz kernel can be given as a convolution kernel in these variables. This kernel is precisely \eqref{eq:diff kernel} at $x=0$, i.e.\ the restriction of the kernel of $P$ to $\ff_z$, which is an invariant notion. The restriction to  $\ff_z$ is also defined for the kernel of an $\bfa$-{\em pseudo}differential operator $P$, so we would like to use this to define the normal operator $N_\bfa(P)$. For this we need to interpret a distribution on $\ff_z$ as kernel of an operator on a suitable space. This will be done using the identification \eqref{eq:ffz diffeo} and the concept of {\em suspended fibre pseudodifferential operators} introduced in \cite{MaMe} which we now recall.

Given a compact manifold $F$ and a real vector space $V$, let $\Psi^m_{\sus(V)}(F)$ be the subspace of $\Psi^m(F\times V)$ of operators $A$ with Schwartz kernels of the form
(neglecting density factors for the moment)
$$ K_A(q,q',v,v') = \overline{K}_A (q,q',v-v'),\quad \overline{K}_A \in C^{-\infty}_c (F^2\times V) + \calS(F^2\times V)$$
where $\calS$ is the Schwartz space. That is, $A$ is translation invariant with respect to $V$ and its kernel decays rapidly at infinity with all derivatives. The singularities of $\overline{K}_A$ are conormal and at $\Delta_F\times\{0\}$.
Via Fourier transform in the $V$ variable, i.e.\ by defining $\Ahat(\zeta)\in \Psi^m(F)$ for $\zeta\in V^*$ by the kernel $K_{\Ahat(\zeta)}(q,q') = \int_V e^{-i<\zeta,w>} \overline{K}_A(q,q',w) \, dw$, operators in $\Psi^m_{\sus(V)}(F)$ correspond exactly to pseudodifferential operators with parameter on $F$ (as defined in \cite{Sh}, for example). These notions can be extended to a family of operators parametrized by a base manifold $B$, where $V$ may now vary with the base point, i.e.\ arise from a (possibly non-trivial) vector bundle over $B$:
\begin{definition}
\label{def:suspended algebra}
Let $\phi:X\to B$ be a fibration of compact manifolds and $\sfV\to B$ a real vector bundle. Let $X\times_B \sfV\to B$ be the fibre product. An operator
$$ A: \calS (X\times_B \sfV) \to \calS(X\times_B \sfV)
$$
is a {\em $\sfV$-suspended fibre $\Psi$DO}, and we write $A\in \Psi^m_{\sus(\sfV)-\phi} (X)$, if it is a family of operators $A_b \in \Psi^m_{\sus(\sfV_b)}(\phi^{-1}(b))$ depending smoothly on $b\in B$ as parameter.
\end{definition}
Clearly, $\Psi^*_{\sus(\sfV)-\phi} (X)=\bigcup_{m\in\ZZ}\Psi^m_{\sus(\sfV)-\phi} (X)$ is an algebra, that is, closed under composition, and the definition can be extended to operators acting on sections of a vector bundle $E\to X$.
It is possible to replace the union over $m\in \ZZ$ with a union over $m \in \RR$ if we also allow
polyhomogeneous expansions with any discrete set of orders in Definition \ref{def:conorm}.

The (convolution) Schwartz kernel of $A\in \Psi^*_{\sus(\sfV)-\phi} (X)$ is a distribution on $X\times_B X\times_B \sfV$ (one base variable, one $\sfV$-fibre variable, two $\phi$-fibre variables). As usual we consider the kernels as half densities without expressing this in the notation.

We apply this definition to the fibration $\phi_2:\partial M\to B_1$ and the vector bundle ${}^\bfa N'B_1\to B_1$, which pulls back to ${}^\bfa N\partial M$ under $\phi_2$. Then $\partial M \times_{B_1} {}^\bfa N'B_1 = {}^\bfa N\partial M$, so the Schwartz kernel is a distribution on $\partial M \times_{B_1} {}^\bfa N\partial M$, which by \eqref{eq:ffz diffeo} can be identified with the interior of the front face, with $\calT,\calY,\calZ$ the variables on the vector space fibre. Therefore, the following definition is consistent with \eqref{eq:indlocal}.
\begin{definition}
\label{def:normal operator}
The {\em normal operator} of $P\in \Psi^m_\bfa(M)$ is the element $N_\bfa(P) \in \Psi^m_{\sus({}^\bfa N'B_1)-\phi_2} (\partial M)$ whose Schwartz kernel is the restriction of $K_P$ to the front face $\ff_z$, where the identification described above is used.
\end{definition}
Summarizing, the normal operator acts on $\calS ({}^\bfa N\partial
M,\Omega^{1/2})$ as a family of pseudodifferential operators on the
fibres of ${}^\bfa N\partial M \to B_1$ that are translation
invariant in the ${}^\bfa N$ fibres.

%

\begin{lemma}
\label{lemma.normalop-exactseq}
For any $m\in\RR\cup\{-\infty\}$ the map assigning to an operator its normal operator fits into an exact sequence
\begin{equation}
\label{eq:normexseq}
0 \longrightarrow x\Psi^m_{\bfa}(M) \hookrightarrow \Psi^m_{\bfa}(M)
\stackrel{N_\bfa}{\longrightarrow} \Psi^m_{\sus(^\bfa N'B_1)-\phi_2}(\del M)
\longrightarrow 0.
\end{equation}
Also, any sequence $P_k\in x^k\Psi^m_\bfa(M)$, $k=0,1,2,\dots$, can
be asymptotically summed, i.e.\  there is a $P\in\Psi^m_\bfa(M)$
such that $P-\sum_{k=0}^{N-1}P_k \in x^N\Psi^m_\bfa(M)$ for all $N$.
\end{lemma}
\begin{proof}
This follows from the fact that $\ff_z$ has a boundary defining function that is equivalent to $x$ in the interior of $\ff_z$ and to $x$ times the $1/a_2$ power of $\calT$, $\calY$ or $\calZ$ at boundary points of $\ff_z$, see \eqref{eq:proj coord quasihom bd} (with $y_p$ replaced by $T$, $Y_i$, or $z_j-z'_j$ and $a=a_2$). These factors are irrelevant since the kernels are rapidly decaying in $\calT,\calY,\calZ$. This together with the standard Borel lemma argument also gives the second claim.
\end{proof}

\begin{definition}
An operator $P\in \Psi_\bfa^m(M)$ is called {\em elliptic} (or \bfa-elliptic) if ${}^\bfa\sigma_m(P)$ is invertible. It is called {\em fully elliptic} if it is elliptic and if $N_\bfa(P)$ is invertible with inverse in $ \Psi^{-m}_{\sus(^\bfa N'B_1)-\phi_2}(\del M)$.
\end{definition}
In Lemma \ref{lem.fullyelliptic} we will give an alternative characterization of full ellipticity.

When we discuss composition we will see that $N_\bfa$ preserves composition. For this, the following
additional characterization of $N_\bfa(P)$ will be useful, which expresses it more directly in terms of the action of $P$, rather than its kernel. For $P$ an $\bfa$-differential operator \eqref{eq:difflocal}, the Fourier transform in $\calT,\calY,\calZ$ of the kernel of its normal operator, i.e.\ \eqref{eq:diff kernel} at $x=0$, is (leaving out density factors for the moment)
\begin{equation}
\label{eq:diff normalfamily}
\Nhat_\bfa(P) = \sum_{\alpha + |I| + |J| +|K| \leq m}
a_{\alpha, I, J,K}(0, y, z,w)\, \tau^\alpha \eta^I \zeta^J D_w^K
\end{equation}
This is similar to the \bfa-principal symbol \eqref{eq:smallsymb}, but here we don't transform in $w$  and we keep lower order terms. Again we want to define this invariantly and extend it to \bfa-pseudodifferential operators.
Write $a=a_1+a_2$ and fix $p=(y_0,z_0)\in B_1$ and $\mu=(\tau,\eta,\zeta)\in\RR^{1+b+f_1}$.
The operators $x^{1+a} D_x$, $x^a D_y$, $x^{a_2}D_z$ have eigenfunctions
$e^{-i\tau a^{-1}x^{-a}}$, $e^{i\eta (y-y_0) x^{-a}}$, $e^{i\zeta(z-z_0) x^{-a_2}}$ with eigenvalues $\tau,\eta,\zeta$, respectively. Their product is $e^{ig}$ with $g = -\tau a^{-1}x^{-a} + \eta\cdot(y-y_0)x^{-a} + \zeta \cdot(z-z_0) x^{-a_2}$. Using this we can express
\eqref{eq:diff normalfamily} applied to $u\in C^\infty(\partial M)$, supported near $y=y_0, z=z_0$, as
\begin{equation}
\label{eq:diff normalfamily2}
\Nhat_\bfa (P) u = \left[ e^{-i g} P(e^{i g} u)\right]_{|x=0} \ \text{ at } (y,z)=(y_0,z_0).
\end{equation}
Note that $D_x$ also differentiates the $D_y$ and $D_z$ eigenfunctions but that this contribution vanishes at $y=y_0$, $z=z_0$ for this choice of eigenfunctions. This generalizes similar formulas in the scattering and $\phi$-calculus setting (see \cite{MaMe}) and is similar to the standard `oscillatory testing' characterization of the principal symbol, but here one needs no extra large parameter since the negative $x$-powers in $g$ already create fast oscillations near the boundary.

Equation \eqref{eq:diff normalfamily2} makes sense invariantly.  Note that $dg = \tau \frac{dx}{x^{1+a}} + \eta \frac{dy}{x^a} + \zeta \frac{dz}{x^{a_2}}$ at $y=y_0,z=z_0,x=0$, so $\tau,\eta,\zeta$ should be regarded as local fibre coordinates on ${}^\bfa N^* \partial M$  with respect to the local basis $\frac{dx}{x^{1+a}}, \frac{dy}{x^a}, \frac{dz}{x^{a_2}}$, see also \eqref{eq:local basis aT*}. Recall that ${}^\bfa N\partial M$ is the pull-back of ${}^\bfa N'B_1={}^\bfa T_{B_1\times\{0\}}(B_1\times [0,1))$.
\begin{definition}[and Proposition]
\label{def.prop.normal family}
Let $P\in \Psi^m_\bfa(M)$. The {\em normal family} of $P$ is the family of operators
on the fibres $F_{2,p}=\phi_2^{-1}(p)$, parametrized by $(p,\mu)\in ({}^\bfa N')^* B_1 $,
defined as follows. Given $(p,\mu)$ choose $\gbar\in C^\infty(B_1\times (0,1))$ such that $d\gbar$ extends to a section of ${}^\bfa T^*_{B_1\times\{0\}}(B_1\times [0,1))$ and  $d\gbar (p)= \mu$, and let $g=\Phi_2^*\gbar$. For $u\in C^\infty(\partial M,\Omega^{1/2})$ let $\utilde$ be a smooth extension to $M$. Then $e^{-ig} P (e^{ig} \utilde)$ extends smoothly to the boundary of $M$. Set
\begin{equation}
\label{eq:def normal family 2}
( \Nhat_\bfa(P)  u ) (p) = \left[ e^{-ig} P (e^{ig} \utilde) \right]_{|F_{2,p}}.
\end{equation}
This depends only on $u_{|F_{2,p}}$ and $\mu$, so defines an operator $\Nhat_\bfa(P)(p,\mu)$ on $C^\infty(F_{2,p})$, and
\begin{equation}
\label{eq:def normal family 1}
\Nhat_\bfa(P)(p,\mu) \in \Psi^m (F_{2,p}).
\end{equation}
The Schwartz kernel of $\Nhat_\bfa(P)$ is the Fourier transform, in the fibres of $^\bfa N' B_1$ (that is, in the variable $\mu$), of the Schwartz kernel of $N_\bfa (P)$.
\end{definition}
The last part of this definition should be interpreted using the identifications made in the remarks before Definition \ref{def:normal operator}. The proof will show that the Fourier transform is well-defined without an additional choice of measure on the fibres.  It makes use of Proposition \ref{prop:a front face general} from the
appendix.

Note that $\gbar$ and hence $g$ may blow up at $x=0$, as in the function used in \eqref{eq:diff normalfamily2}. The condition on $d\gbar$ amounts to controlling this blow up.
\begin{proof}
First, note that $P\in\Psi^m_\bfa(M)$ defines a ``boundary operator''
$$ P_\partial: C^\infty(\partial M) \to C^\infty(\partial M), u\mapsto (P\utilde)_{|\partial M}  $$
for $\utilde$ an arbitrary smooth extension of $u$ to the interior. We leave out density factors.
That $P_\partial u$ does not depend on the choice of extension follows from its expression in coordinates, which from \eqref{eq:Plocal} is seen to be
\begin{equation}
\label{eq:Pbdlocal}
\begin{aligned}
(P_\partial u)(y,z,w) &= \int_{F_{2,(y,z)}} K'_{P_\partial} (y,z,w,w') u(y,z,w')\, dw' \ \text{ where} \\
K'_{P_\partial}(y,z,w,w') &= \int_{\RR^{1+b+f_1}} K'_P (0,\calT,y,\calY,z,\calZ,w,w') \, d\calT d\calY d\calZ
\end{aligned}
\end{equation}
for $u$ supported in $\phi_2^{-1}(U)$ for a coordinate neighborhood $U$ in $B_1$.
Observe that as there is no integration in $y$ and $z$ this can be considered as a family of operators in the $F_2$ fibres parametrized by $p=(y,z)\in B_1$.
Now by \eqref{eq:def normal family 2} we need to consider $\left(e^{-ig} P e^{ig}\right)_\partial$ and then restrict this to $F_{2,p}$. The operator $e^{-ig} P e^{ig}$ has kernel $e^{-ig} e^{ig'} K_P$. Its lift to $M^2_z$ is $e^{-iG}K_P$ where $G=\beta_\bfa^*(\pi_l^*g-\pi_r^*g)$. By \eqref{eq:dg meaning} the function $G$ is smooth near the interior of $\ff_z$, with only finite order singularities at the other faces of $M^2_z$, so $e^{-ig}Pe^{ig}\in\Psi^m_\bfa$. This proves the smoothness claim before \eqref{eq:def normal family 1}. Also, $G_{|\intff_z} = dg$, which, over $p$, equals $\tau\calT+\eta\calY+\zeta\calZ$ in coordinates (as in the note before Definition \ref{def.prop.normal family}), where $\mu=(\tau,\eta,\zeta)$. Finally, from \eqref{eq:Pbdlocal} we obtain that $\Nhat_\bfa(P)(p,\mu)$ has kernel
$$K'_{\Nhat_\bfa(P)(y,z,\tau,\eta,\zeta)}(w,w') = \int_{\RR^{1+b+f_1}} e^{-i(\tau\calT+\eta\calY+\zeta\calZ)}K'_P (0,\calT,y,\calY,z,\calZ,w,w') \, d\calT d\calY d\calZ
$$
from which the remaining claims follow.
\end{proof}
\begin{lemma}
\label{lem.fullyelliptic} Suppose $P\in\Psi_\bfa^m(M)$ is elliptic.
Then $N_\bfa(P)$ is elliptic. Furthermore, $P$ is fully elliptic if
and only if $\Nhat_\bfa(P)(p,\mu)$ is invertible, with inverse in $\Psi^{-m}(F_{2,p})$, for all $(p,\mu)$.
\end{lemma}
Note that invertibility of $\Nhat_\bfa(P)(p,\mu)$ as an operator on $L^2(F_{2,p})$ already implies that its inverse is in $\Psi^{-m}(F_{2,p})$. This is a standard result about pseudodifferential operators on compact manifolds, see \cite{Sh}, since $\Nhat_\bfa(P)(p,\mu)$ is elliptic. See Lemma \ref{lem.inverse in calculus} for the analogous result in the \bfa-calculus.

\begin{proof}
Interpret the principal symbol ${}^\bfa\sigma_m(P)$ of $P$ as a function on
the conormal bundle of the diagonal, $N^*\Delta$, see the remarks
before Lemma \ref{lem.short exact - sigma}. To prove the first statement of the lemma, we claim that the principal symbol of $N_\bfa(P)$ is, under the identifications explained before Definition \ref{def:normal operator}, just the restriction of ${}^\bfa\sigma_m(P)$ to the boundary $N^*_{\partial\Delta}\Delta$ of $N^*\Delta$. To prove this claim, note that in the setting of
operators $P\in \Psi^m_{\sus(V)}(F)$, the principal symbol of $P$ is naturally a function on $N^*(\Delta_F\times \{0\})=N^*(\Delta_F)\times V^*$.  Thus in the $\phi$-calculus setting, the principal symbol of an element of $\Psi^m_{\sus(\sfV)-\phi}(X)$ as in Definition \ref{def:suspended algebra} is a function on $N^*(\Delta_F)\times_B \sfV^*$, where $N^*(\Delta_F)$ is the bundle over $B$ whose fibre over $b\in B$ is the normal bundle of the diagonal $\Delta_{F_b}\subset F_b^2$. So finally in the $\bfa$-calculus setting, the principal symbol of the normal operator is a function on $N^*(\Delta_{F_2})\times_{B_1}({}^\bfa N')^*B_1$, which is identified with the boundary
$N^*_{\partial\Delta}\Delta$ of $N^*\Delta$ (in local coordinates, both are parametrized by the variables $y,z,w,\tau,\eta,\zeta,\theta$, see also \eqref{eq:smallsymb} with $x=0$).  For the second claim in the lemma, note that, by  the discussion
before Definition \ref{def:suspended algebra}, ellipticity of
$N_\bfa(P)$ is equivalent to ellipticity of $\Nhat_\bfa(P)(p,\mu)$
in the sense of pseudodifferential operators with parameter $\mu$.
Therefore, the second claim follows from the fact that if such an
operator is invertible for every value of the parameter then its
inverse (taken for each parameter separately) is again a
pseudodifferential operator with parameter, see \cite{Sh}. There it is also shown that invertibility for large $\mu$ already follows from ellipticity.
\end{proof}

\section{Composition and mapping properties}

Both the action and the composition of operators can be expressed, as is laid out in \cite{Gr}, through pullbacks and pushforwards:  If the operator $P$ has kernel $K_P\in C^{-\infty}(M^2)$ and $u$ is a function on $M$ then $P(u) = (\pi_l)_*[K_P\cdot (\pi_r)^*u]$ where $\pi_l$ and $\pi_r$ are the left and right projections from $M^2$ to $M$. Since operators in $\Psi^*_\bfa$ are characterized by the properties of their kernels lifted to $M^2_z$ we should rewrite this as
\begin{equation}
\label{eq:action}
P(u) = (\pi_{z,l})_* \left[
K_P\cdot(\pi_{z,r})^*u\right]
\end{equation}
where $\pi_{z,l/r} = \beta_{\bfa} \circ \pi_{l/r}$. Recall that we consider $K_P$ and $u$ as half-densities. Strictly speaking, pushforward is defined invariantly only on full densities, so one should choose an auxiliary non-vanishing half-density $\nu$ on $M$ and define $P(u)$ by $P(u)\cdot\nu = (\pi_{z,l})_* [
K_P \cdot(\pi_{z,l})^*\nu \cdot(\pi_{z,r})^*u]$, but we will suppress such auxiliary densities from the notation.
Clearly, since pushforward involves integration, $P(u)$ will be defined only under certain integrability conditions.
Similarly, the composition $P\circ Q$ of two operators $P$ and $Q$ with kernels
$K_P$ and $K_Q$ on $M^2_z$ has kernel
\begin{equation}
\label{eq:comp}
K_{P\circ Q} = (\pi^3_{z,2})_*\left[(\pi^3_{z,3})^*K_P\cdot (\pi^3_{z,1})^*K_Q\right],
\end{equation}
where now the $\pi^3_{z,i}:M^3_z\to M^2_z$ are lifts of the projections $\pi_i^3:M^3\to M^2$, with $\pi_i^3$ projecting off the $i$th factor, $i=1,2,3$. Here, the blow-up $M^3_z$ of $M^3$, called the \bfa{\em-triple
space}, has to be constructed in such a way that these lifts actually exist and satisfy appropriate conditions.
To check that $P,Q\in\Psi^*_\bfa$ implies $P\circ Q\in\Psi^*_\bfa$ one needs to understand the behavior of distributions on manifolds with corners under pull-back and push-forward. This is expressed in the pullback and pushforward theorems which we recall below. These theorems require the maps involved to be of a certain type (b-maps resp.\ b-fibrations). We then check that for our construction of $M^3_z$ the maps $\pi^3_{z,i}$ are b-fibrations. Finally we define the full calculus of \bfa-pseudodifferential operators, an extension of the small calculus, and apply this construction to prove mapping and composition theorems for operators in the full calculus, where again integrability conditions will appear.

Our construction of the \bfa-triple space is a fairly straightforward generalization of the construction of the $\phi$-triple space in \cite{MaMe}. However, checking that the triple to double space projections are b-fibrations is rather more involved in the $\bfa$-calculus, see Lemma \ref{lem:atripmaps}.

Before embarking on all of this, it is worth considering if it is really necessary.  In \cite{ALN}, Ammann, Lauter and Nistor construct an algebra of pseudodifferential operators on a broad family of
complete manifolds with a ``Lie structure at infinity," which includes $\bfa$-manifolds with $\bfa$-boundary
metrics.  They prove that this algebra contains the algebra of $\bfa$-differential operators and satisfies a composition theorem, a Sobolev mapping theorem and is closed under adjoints.  The proof of their
composition result uses Lie groupoids and avoids our messy triple-space construction altogether.
However, although this is an elegant paper, it does not achieve the goal of this work and its upcoming
sequel, which is to construct parametrices for the standard geometric operators on an $\bfa$-manifold.  It
is possible to construct inverses for these operators up to smoothing operators in the ALN-algebra,
but it is not generally possible to construct inverses up to compact operators.  This is due to the
proper support condition they impose, which requires that kernels be supported in a compact
region around the diagonal in the double space.  The kernels of parametrices of the fully-elliptic
$\bfa$-operators considered in this paper will decay away from the intersection of the
lifted diagonal with $\ff_z$ in the blown-up double space (see Figure 1), but will not be
compactly supported on $\ff_z$.
So for instance, the resolvent family of the Laplacian on an $\bfa$-manifold with an $\bfa$-boundary
metric will not be contained in the Ammann, Lauter, Nistor algebra, whereas it is contained in the small
$\bfa$-calculus constructed here, as proved in Theorem \ref{th:resolv}.

Thus, to construct parametrices for fully elliptic operators, it is necessary to consider behavior
near the intersection of
$\ff_z$ and $\ff_{zy}$ in the double space, and for composition, it will be necessary to consider behavior
at the intersections of several faces in the triple space.  Again, there are faces and intersections
in the double and triple spaces that do not come into play when restricting consideration to
fully elliptic operators, and it might be possible to construct somewhat simpler spaces and still
study these.  For instance, in \cite{V}, Vasy successfully used a simplified blow-up construction
to study three-body scattering.  However, when we go on to construct parametrices for
``split elliptic \bfa-operators,"  which, for instance, include the Laplacian and Dirac operators on
$\mathbb{Q}$-rank one locally symmetric spaces, we will need to consider behavior at all of
the boundary faces.  Thus, complicated though it is, we do really need the full double and triple
space constructions and theorems we prove in this section.

\subsection{Pullback and pushforward theorems}

Before we can state the necessary pullback and pushforward theorems, we need
several definitions from \cite{Me-mwc} and \cite{EMM}.

\begin{definition}
Two p-submanifolds, $Y$ and $W$ of a manifold with corners $X$ are said to intersect
{\em cleanly} if their intersection is again a p-submanifold and if at each point $p \in Y \cap W$,
$$
T_pY \cap T_pW = T_p(Y\cap W).
$$
They intersect {\em transversally} if, as usual, for all $p \in Y \cap W$
$$
T_pY + T_pW = T_p X.
$$
\end{definition}

\begin{definition}
\label{def:b-map}
A map $f:X \rightarrow Z$ between manifolds with corners is called a {\em b-map} if
it is of product type near the boundaries of $X$ and $Z$.  That is, for any local identification
of $X$ with $\RR^k_+ \times \RR^{n-k}$ and of $Z$ with
$\RR^{k^\prime}_+ \times \RR^{n^\prime-k^\prime}$, $f$ is given by $f=(f_1, \ldots, f_{n^\prime})$
where
$$
f_i (x_1, \ldots, x_k, x_{k+1}, \ldots, x_n) = a_i(x) \Pi_{j=1}^{k} x_j^{\alpha_{i,j}}
$$
for some smooth $a_i$ nonvanishing at $0$ and for non-negative integers $\alpha_{i,j}$.
\end{definition}
It follows from the connectedness of boundary hypersurfaces (bhs) that, given a b-map $f:X\to Z$, one may associate an integer $\alpha_{GH}$ to each bhs $G$ of $X$ and each bhs $H$ of $Z$, such that the $\alpha_{i,j}$ above equal $\alpha_{GH}$ if $G$ is locally given by $x_i=0$ and $H$ by $x_j'=0$. The array of numbers $\alpha_{GH}$ is called the {\em exponent matrix } of $f$.

\begin{definition}
A b-map $f:X \rightarrow Z$ between manifolds with corners is called {\em transversal}
to an interior p-submanifold $Y \subset Z$ if the associated map on the
b-tangent bundles, ${}^{\rm b}\!f_*:{}^{\rm b}T_p X \rightarrow {}^{\rm b}T_{f(p)} Z$, satisfies
$$
{}^{\rm b}\!f_*(^{\rm b}T_p X) + {}^{\rm b}T_{f(p)}Y = {}^{\rm b}T_{f(x)}Z
$$
for all $p \in f^*(Y)$.
\end{definition}

A b-map $f:X\rightarrow Z$ induces a map $\fbar$ from the faces (assumed connected) of $X$ to the faces of $Z$ defined by the requirement that for each face $F$ of $X$, interior points of $F$ are mapped by $f$ to interior points of $\fbar(F)$.

\begin{definition}
\label{def:b-fibration}
A b-map $f:X\to Z$ is called a {\em b-fibration} if for each boundary hypersurface $H$ of $X$,
$codim \overline{f}(H) \leq 1$ and
for each face $F$ of $X$,
the map $f$ restricted to the interior of $F$ is a fibration over the interior of $\overline{f}(F)$.
\end{definition}
Index families of polyhomogeneous functions transform under pullback as follows.

\begin{definition}
\label{def:pullbackindex}
Let $f:X\rightarrow Z$ be a b-map between manifolds with corners.  Let $\mathcal{E}$ be
an index family on $Z$. Define the index family $f^{\#}(\calE)$ on $X$ as follows. If $G$ is a bhs of $X$ then
$$
f^{\#}(\mathcal{E})(G) =
 \left\{
     (q+ \sum_H e_f(G,H) z_H, \sum_H p_H) :  \right.
     $$
     $$ \left.
     q \in \mathbb{N}_0 \
         \mbox{ and for each bhs } H \subset Z,
             \left\{
                 \begin{array}{ll}
                 (z_H, p_H) \in \mathcal{E}(H) & \mbox{ if } e_f(G,H) \neq 0 \\
                 (z_H, p_H) = (0,0)                    & \mbox{ if } e_f(G,H) = 0.
                 \end{array}
                 \right.
 \right\}.
$$
\end{definition}

Now we may state the pullback theorem (Proposition 6.6.1 in \cite{Me-mwc},
Theorem 4 in \cite{GH2}).

\begin{theorem}
\label{thm:genpullback}
Let $f:X \rightarrow Z$ be a b-map of compact manifolds with corners, and let $Y \subset Z$ be
an interior p-submanifold to which $f$ is transversal.  Then for any index family $\mathcal{E}$ and any
$m \in \RR$, there is a natural continuous pull-back map
$$
f^*: I^{m, \mathcal{E}}(Z,Y) \rightarrow
I^{m-\frac{1}{4}(\dim X - \dim Z), f^{\#}\mathcal{E}}(X, f^{-1}(Y)).
$$
\end{theorem}

The second theorem we use is a combination of a product theorem for distributions and
a pushforward theorem.  Again, we start by defining the pushforward of an index family.

\begin{definition}
\label{def:pushforwardindex}
Let $f:X\rightarrow Z$ be a b-fibration between manifolds with corners.  Let $\mathcal{E}$ be
an index family on $X$. Define an index family $f_{\#}(\calE)$ by setting, for $H\subset Z$ a bhs,
$$
f_{\#}(\mathcal{E})(H)= \underset{G}{\overline{\bigcup}}
\left\{ (\frac{z}{e_f(G,H)}, p) : (z,p) \in \mathcal{E}(G) \right\},
$$
where the union is over bhs $G$ of $X$ satisfying $e_f(G,H)>0$ and the symbol $\overline{\cup}$ denotes the {\em extended union} of these index
sets, defined inductively from
$$
\overline{\cup}(E,F) := E \cup F \cup \{(z, p^\prime + p^{\prime\prime} +1): (z,p^\prime) \in E, \:
(z, p^{\prime\prime}) \in F \}.
$$
\end{definition}

Now we can state the pushforward theorem (Proposition B7.20 in \cite{EMM}, Theorem 6 in \cite{GH2}):

\begin{theorem}
\label{thm:genpushforward}
Let $f:X \rightarrow Z$ be a b-fibration and $Y$ and $W$ be interior p-submanifolds
satisfying the following conditions:
\begin{itemize}
\item $Y$ and $W$ intersect transversally
\item $f$ restricted to each of $Y$ and $W$ is a diffeomorphism onto $Z$
\end{itemize}
Consider now $u \in I^{m, \mathcal{E}}(X, Y;\Omega_b^{1/2})$ and $v \in I^{m^\prime, \mathcal{F}}(X,W;\Omega_b^{1/2})$ for some index families $\mathcal{E}$ and $\mathcal{F}$ on $X$.
If for all boundary hypersurfaces $H$ of $ X$ with $\fbar(H)=Z$ we
have  $\inf(\mathcal{E}(H) + \mathcal{F}(H))>0$,
then the pushforward
by $f$ of $uv$, denoted $f_*(uv)$, is in $I^{m'', \mathcal{G}}(Z,Y\cap W;\Omega_b)$, where
$\mathcal{G} = f_{\#} (\mathcal{E} + \mathcal{F})$ and $m''= m+m' + \frac{1}{4} \dim(Z)$.
\end{theorem}

Here by $\mathcal{E} + \mathcal{F}$ we mean that for each boundary hypersurface $H \subset X$,
we take the indicial set $\mathcal{E}(H) + \mathcal{F}(H)$ as a subset of
$\mathbb{C} \times \NN_0$. Note that for the stated simple transformation rule of index sets and integrability condition to be correct it is important that everything is expressed as b-density. Here, a b-density on a manifold with corners is locally of the form $| \frac{dx_1}{x_1}\cdots \frac{dx_k}{x_k}\,dy_1\cdots dy_{n-k} |$.

In order to prove that the various maps and submanifolds we will use in our composition and
mapping results satisfy the conditions of these theorems, we will use the following lemma
(proof of Proposition 6 in \cite{MaMe}):

\begin{lemma}
\label{lem:b-map}
The composition of a b-fibration with the blowdown map for a blowup of a boundary p-submanifold is again a b-fibration.
\end{lemma}


and the lemma (Proposition 5.7.2 in \cite{Me-mwc}, Lemma 4 in \cite{GH2}):

\begin{lemma}
\label{lem:boundaryp}
If $Y$ and $V$ are cleanly intersecting p-submanifolds then the lift of $Y$ to $[X;V]$ is again a p-submanifold. If $Y_a,V_b$ are boundary p-submanifolds defined to higher order then this remains true for the blowup $[X;V]_b$ if the lift of $Y_a$ is defined as at the end of Subsection \ref{subsec:qhom}.
\end{lemma}

It is clear that a map with a global product structure is a b-fibration.
For instance the maps $\pi_{l/r}:M^2 \rightarrow M$ have global product structures, so they
are b-fibrations.  Then to get the maps $\pi_{\bfa,l/r}$ we compose with three blow-down maps
which by the second lemma above all come from blowups of boundary p-submanifolds.
Thus by the first lemma, the overall maps are b-fibrations.
We will show that all of the relevant maps for the composition theorems are also compositions
of blow-down maps for boundary p-submanifolds and maps with global product
structures, and thus also b-fibrations.

In order to apply the pullback theorem, we also need to check that these b-fibrations are
transversal to certain submanifolds in
their image.  But this will be  true automatically from the definitions of b-fibration and transversality.
Finally, to apply the pushforward theorem we need to know that certain submanifolds intersect
transversally.  For this, we will use the following lemma (Lemma 5 in \cite{GH2}):

\begin{lemma}
\label{lem:transversal}
If $Y$ and $W$ are transversally intersecting interior p-submanifolds  of $X$ and $V_a$ is a submanifold to order $a$ of $X$ having an interior extension contained in $Y\cap W$, then the lifts
of $Y$ and $W$ to $[X;V]_a$ also intersect transversally.
\end{lemma}

\subsection{The \bfa-triple space}

Here we construct the triple space $M^3_z$ for the \bfa-calculus.
The construction of the \bfa-triple space is an orgy of notation, so bear with us. See Figure \ref{fig:triplespace}.
We start with the triple space $M^3$.  This space has three boundary hypersurfaces,
corresponding to the boundaries in each of the three factors.  Call these
$H_1 = \del M \times M^2$, $H_2 = M \times \del M \times M$ and
$H_3 = M^2 \times \del M $.
We also have three edges, $E_3 = H_1 \cap H_2$,
$E_2 = H_1 \cap H_3$ and $E_1 = H_2 \cap H_3.$
Finally, there is the `corner', $V=(\del M)^3$.

The interior fibre diagonals $\Delta_y,\,\Delta_z \subset M^2$ lift under the three projection maps to the double fibre diagonals $\Delta_{i,y} =(\pi_i^3)^{-1} (\Delta_y) \subset M^3$, and similarly we define the smaller fibre diagonals $\Delta_{i,z} =(\pi_i^3)^{-1} (\Delta_z) \subset M^3$. The intersection of the three double fibre diagonals $\Delta_{i,y}$ is the triple fibre diagonal $\Delta_y^3$,  and similarly $\Delta_z^3$ is the intersection of the $\Delta_{i,z}$.

We will want to trace these various submanifolds through the several blowups we
will do to $M^3$.  In any of the blowups where a given submanifold has not itself been
blown up, we will abuse notation and use the same notation for the original submanifold and for its lift under the blowup.

To construct the \bfa-triple space, we start by constructing the $x$-(or b-)triple space $M^3_x$.
That is obtained by first blowing up the vertex $V \subset M^3$ to obtain
$[M^3;V]$.  Call the resulting new front face $\mathcal{V}_x$.
Then blow up the (now disjoint) three axes to obtain
\begin{equation}
\label{eq:xtriplespace}
M^3_x = [[M^3;V] ; E_1; E_2; E_3].
\end{equation}
By Theorem \ref{th:commblowups} in the next subsection, the fact that the inverse images of the axes
$E_1, E_2$ and $E_3$ are disjoint in $[M^3;V]$ means it doesn't matter in which order
we blow them up.
Also recall again here that the $E_i$ in
this blowup really represent the closures in $[M^3;V]$ of the inverse images of the original
$E_i \setminus V$.  In the future, we will not point this out.
Call the faces resulting from these blowups $\mathcal{E}_{i,x}$.
Denote the overall blow down map for this step by $\beta^3_x:M^3_x \rightarrow M^3$.

The next step is to construct the $y$-triple space $M^3_y$, but using quasihomogeneous
rather than ordinary blowups. In the case $a_1=1$ this is the $\phi$-triple space from \cite{MaMe}. The first step is to blow up the intersection of
the $y$ triple diagonal with $\mathcal{V}_x$, that is, $\del \Delta^3_y$.  Call the new face
$\mathcal{V}_y$.  Now blow up the
three submanifolds that are the intersections with $\mathcal{V}_x$ of
the double fibre diagonals $\Delta_{i,y}$.  Call these new faces
$\mathcal{G}_{i, y}$.  Finally, blow up the intersections of the
double fibre diagonals with the submanifolds $\mathcal{E}_{i, x}$, whose inverse images
are disjoint in $[M^3_x;\del \Delta^3_y]$.
Call these resulting faces $\mathcal{E}_{i, y}$.  All of these
blowups will be quasihomogeneous blowups of degree $a_1$.  The result will be essentially
the $\phi$ triple space for the fibration $\phi_{2,1}$ constructed in \cite{MaMe},
and will be exactly that when $a_1=1$.  So overall from this step we have
\begin{equation}
\label{eq:ytriplespace}
M^3_y = [[[M^3_x; \del\Delta^3_y]_{a_1}; \Delta_{i,y} \cap \mathcal{V}_x ]_{a_1};
\Delta_{i,y} \cap \mathcal{E}_{i, x}]_{a_1}.
\end{equation}
Here to save notation we have written a blowup of $\Delta_{i,y} \cap\mathcal{V}_x$
to indicate a blowup of each of the three such (disjoint) manifolds for $i=1,2,3$, and so on.  We will
continue also to use this space saving notation.
Denote the overall blow down map for this step by $\beta^3_y:M^3_y \rightarrow M^3_x$.

Finally we construct the \bfa-(or $z$-)triple space $M^3_z$ by blowing up submanifolds from the $y$-
triple space.  Again, start by blowing up the intersection of the triple diagonal,
this time the $z$ triple diagonal, with the new front face, $\mathcal{V}_y$, that is $\del \Delta^3_z$.
Call the resulting face $\mathcal{V}_z$.  Now blow up the three now disjoint preimages
of the intersections
of the $z$ double diagonals with $\mathcal{V}_y$.  Call the resulting faces
$\mathcal{F}_{i,z}$.  Now blow up the intersections of the $z$ double diagonals
with the faces $\mathcal{G}_{i,y}$.  Call the resulting faces $\mathcal{G}_{i,z}$.
Finally, blow up the intersections of the $z$ double diagonals with the
faces $\mathcal{E}_{i, y}$, and call the resulting faces $\mathcal{E}_{i,z}$.
All of these blowups will be of degree $a_2$.
Then overall we have
\begin{equation}
\label{eq:ztriplespace}
M^3_z = [[[[M^3_y; \del {\Delta}^3_{z}]_{a_2};
{\Delta}_{i,z} \cap \mathcal{V}_y ]_{a_2};
{\Delta}_{i,z} \cap \mathcal{G}_{i, y}]_{a_2};
{\Delta}_{i,z} \cap \mathcal{E}_{i, y}]_{a_2}.
\end{equation}
Denote the overall blow down map for this step by $\beta^3_z:M^3_z \rightarrow M^3_y$.
This completes the construction of the \bfa-triple space, $M^3_z$. The total blow-down map will be denoted
$\beta^3_\bfa = \beta^3_x\circ\beta^3_y\circ\beta^3_z :M^3_z\to M^3$.

\begin{figure}
\begin{center}
\includegraphics[width=15cm]{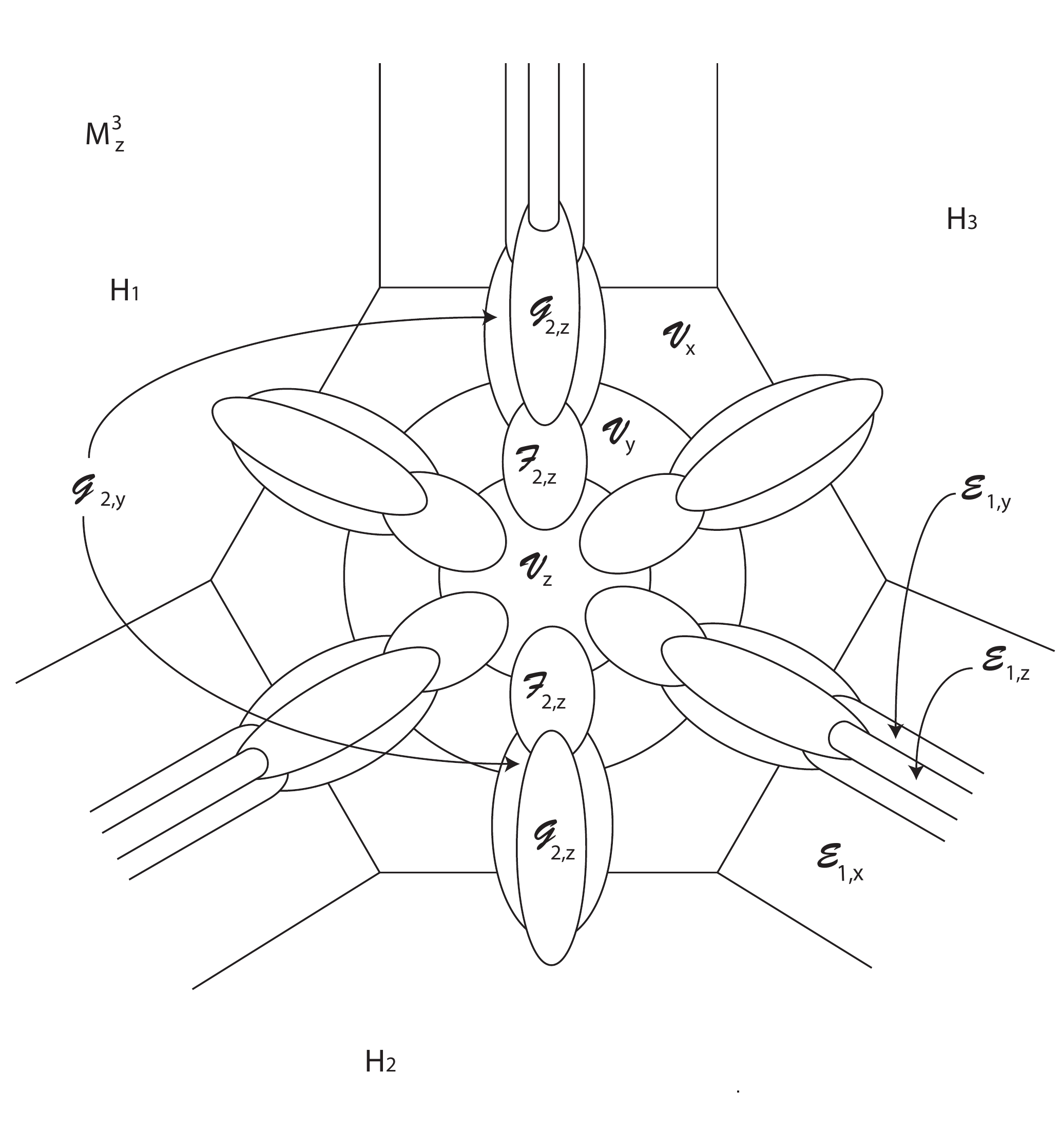}
\end{center}
\caption{The \bfa-triple space}
\label{fig:triplespace}
\end{figure}

\subsection{The \bfa-double and \bfa-triple space maps}
We need to check that the maps $\pi^3_{z,2}$ and
$\pi_{z,l}$ are  b-fibrations,
and $\pi^3_{z,1}$,  $\pi^3_{z,3}$and
$\pi_{z,r}$ are b-maps.  It is of course sufficient to show they are all b-fibrations using
Lemmas \ref{lem:b-map} and \ref{lem:boundaryp}.
After Lemma \ref{lem:boundaryp} we have already indicated how to do this in the case of the double space maps, so we
can state:

\begin{lemma}
\label{lem:bdouble}
The \bfa-double space maps,
$\pi_{z,l/r} := \beta_{\bfa}\circ \pi_{l/r}$
are b-fibrations.
\end{lemma}

Above, we constructed the \bfa-triple space as the result of a sequence of blowups
of the triple space $M^3$ which shows that it is symmetrical in the three
sets of variables.  Here we will show that it is also obtained by a sequence
of blowups of $M\times M^2_z $, which will imply that the relevant maps are b-fibrations.

In order to do this we first need some rules for commuting two blowups.  These
are contained in the following theorem, (Propositions 5.8.1 and 5.8.2 in \cite{Me-mwc},
Theorem 1 in \cite{GH2}):
\begin{theorem}
\label{th:commblowups}
Suppose that $Y$ and $W$, $A$ and $B$ are p-submanifolds of a mwc $X$, with given extensions to the interior (so blow-ups to any orders are defined), and let $a,b\in\NN$.  Then
blowups can be commuted according to the following rules:
\begin{enumerate}
\item If $Y \cap W= \emptyset$ then $[[X;Y]_a;W]_b = [[X;W]_b;Y]_a$.
\item If $A \subset B \subset X$ and $A$ is a p-submanifold of $B$, then
$[[X;A]_a;\tilde B]_a = [[X;B]_a;\tilde A]_a$, where
$\tilde{A}$ and $\tilde{B}$ denote the lifts of $A$ and $B$ under the $B$ and $A$
blowdown maps, respectively.
\item If $Y, W$ intersect cleanly in $A$, then
$$[[[X;Y]_a;\tilde{A}]_b; \tilde{\tilde{W}}]_b = [[[X;W]_b;\tilde{A}^\prime]_a;\tilde{\tilde{Y}}]_a,
$$
where $\tilde{A}$ and $\tilde{A}^\prime$ are the lifts of $A$ under the $Y$-blowdown
and $W$-blowdown maps, respectively, etc.
If in addition $a=b$, then we have further that both of these
$$= [[[X;A]_a;\tilde{Y}^\prime]_a;\tilde{\tilde{W}}^\prime]_a.$$
\end{enumerate}
\end{theorem}

Note that the statement corresponding to (2) is false for different orders for $A,B$. There is a corresponding statement if one introduces a quasi-homogeneous blowup with different weights in different boundary directions. However, this will not be necessary for our proof, and we can now prove the required lemma about the \bfa-triple space maps.

\begin{lemma}
\label{lem:atripmaps}
For $i=1,2,3$, there exist b-fibrations $\pi^3_{z,i}$,
$\pi^3_{y,i}$ and $\pi^3_{x,i}$
fixed by the demand that they make the following diagram with the projection
maps $\pi^3_i$ commutative.
$$
\xymatrix{
M^3_z   \ar[d]^{\pi^3_{z,i}  } \ar[r]^{\beta^3_z} &   M^3_y    \ar[d]^{\pi^3_{y,i}} \ar[r]^{\beta^3_y}&
           M^3_x             \ar[d]^{\pi^3_{x,i}} \ar[r]^{\beta^3_x}&     M^3     \ar[d]^{\pi^3_i } \\
M^2_z    \ar[r]^{\beta_z   } &  M^2_y    \ar[r]^{\beta_y}& M^2_x
                      \ar[r]^{\beta_x}&      M^2
}
$$
\end{lemma}

The existence of the $x$-projections $\pi^3_{x,i}$  is known from the
b-calculus, see, eg \cite{Gr}.  Similarly, the existence of the $y$-projections $\pi^3_{y,i}$, at least in
the case where $a_1=1$, is known from the $\phi$-calculus
and is contained, eg, in \cite{Va}.  Nevertheless, in order to study the $z$-projections,
we need to start by considering the $x$- and $y$-projections.  The proof,
unfortunately, involves a very large number of submanifolds in a large number of various
blown up spaces.  Thus although the basic ideas are not complicated, and just consist
of applying the various parts of the above theorem recursively, the notation is very
cumbersome, despite our attempts to simplify and rationalize it as much as possible.

\begin{proof}
By symmetry it suffices to consider the case $i=1$.
The uniqueness of the lifted projections follows from the fact that the blow-ups are diffeomorphisms in the interior. In order to show that they exist (i.e. extend continuously, even smoothly to the boundary) and are b-fibrations, we show that we can construct the triple spaces
as a series of blowups of boundary p-submanifolds starting with $M \times M^2_{x/y/z}$.
This is so that we can define the maps  $\pi^3_{x/y/z,1}$ as the composition of
the projection $M \times M^2_{x/y/z} \rightarrow M^2_{x/y/z}$ with a sequence of
blowups of boundary p-submanifolds and thereby show they are b-fibrations.
We do this in each case by commuting blowups to change the construction of the triple
spaces through the string of blowups we described above into a string of blowups
starting with $M\times M^2_{x/y/z}$.

Start with the $x$-triple space, $M^3_x = [[M^3;V];E_1;E_2;E_3]$.  Since $V \subset E_1$,
we can apply part (2) of Theorem \ref{th:commblowups} to get this equals
$[[M^3;E_1];V;E_2;E_3]$ (recall we are using the same notation for an original
submanifold as for its inverse image in a blowup).  Now we notice that $[M^3;E_1] = M \times M^2_x$,
so $M^3_x = [M \times M^2_x;V;E_2;E_3]$ (where here $V$ indicates the inverse
image of the original $V=E_1 \int E_2 \int E_3$ in the blow-up $[M^3;E_1] = M
\times M^2_x$).  Thus we have a natural map
$\pi^3_{x,1}:M^3_x \rightarrow M^2_x$ that is the composition of the three blow down maps
from $M^3_x$ to $M \times M^2_x$ with the natural projection $M \times M^2_x \to M^2_x\rightarrow M^2$.
Thus this map is a b-fibration.  Further, by construction, this map fits into the necessary commutative
diagram.

Now consider the $y$-triple space
$$
M^3_y = [[[M^3_x; \del \Delta^3_y]_{a_1};
\Delta_{i,y} \cap \mathcal{V}_x ]_{a_1};
\Delta_{i,y} \cap \mathcal{E}_{i, x}]_{a_1}
$$
Notice that the inverse images in $M^3_x, \del \Delta^3_y]_{a_1}$ of the submanifolds
$\Delta_{i,y} \cap \mathcal{V}_x$ are disjoint
for $i=1,2,3$, as are the inverse images of the submanifolds $\Delta_{j,y} \cap \mathcal{E}_{j, x}$
for $j=1,2,3$. Further, if $i\neq j$ then the two types are disjoint from each other.  Thus by
part (1) of Theorem \ref{th:commblowups}, we can put the blow-ups with $i=j=1$ first, then
put the others at the end:
$$
M^3_y = [[[[M^3_x; \del \Delta^3_y]_{a_1};
\Delta_{1,y} \cap \mathcal{V}_x ]_{a_1};
\Delta_{1,y} \cap \mathcal{E}_{1, x}]_{a_1}; \cdots]_{a_1}.
$$
Now since $\del\Delta^3_y = \Delta^3_y \cap \mathcal{V}_x \subset \Delta_{1,y} \cap \mathcal{V}_x$, we can use part (2) of the theorem to switch the first two blowups in this sequence:
$$
= [[[[M^3_x;\Delta_{1,y} \cap \mathcal{V}_x ]_{a_1};
\del \Delta^3_y ]_{a_1};
\Delta_{1,y} \cap \mathcal{E}_{1, x}]_{a_1}; \cdots]_{a_1}.
$$
After this switch, we see that the inverse images of the two submanifolds $ \del \Delta^3_y$
and $\Delta_{1,y} \cap \mathcal{E}_{1, x}$ in the first blowup are disjoint, so we can switch them
again by part (1) of the theorem to get:
$$
= [[[[M^3_x;\Delta_{1,y} \cap \mathcal{V}_x ]_{a_1};
\Delta_{1,y} \cap \mathcal{E}_{1, x}]_{a_1};
\del \Delta^3_y]_{a_1}; \cdots]_{a_1}.
$$
Now going to the $x$-triple space part of this proof, we can substitute in for $M^3_x$ to get:
$$
= [[[[[[M \times M^2_x;V];E_2;E_3];
\Delta_{1,y} \cap \mathcal{V}_x ]_{a_1};
\Delta_{1,y} \cap \mathcal{E}_{1, x}]_{a_1};
\del \Delta^3_y]_{a_1}; \cdots]_{a_1}.
$$
Now notice that $E_2$ and $E_3$ are disjoint from
$\Delta_{1,y} \cap \mathcal{E}_{1, x}$ and
$\del \Delta^3_y$, so again by part (1) of the theorem,
we can commute these blowups and lump the
$\del \Delta^3_y$, $E_2$ and $E_3$ blowups together with the
other ones at the end to get
\begin{equation}
\label{eq:midphiblowup}
= [[[[M \times M^2_x; V];
\Delta_{1,y} \cap \mathcal{V}_x ]_{a_1};
\Delta_{1,y} \cap \mathcal{E}_{1, x}]_{a_1};
\cdots]_{*}.
\end{equation}
Here we use the shorthand notation $[A;\cdots]_*$ to indicate that there are different degrees to the
various blowups at the end.

Now we notice that in $M \times M^2_x$, the submanifolds
$R=(M \times \Delta_y)\cap \mathcal{E}_{1,x}$ and $Q=$ inverse image of $V$
intersect cleanly to give a manifold we will call $A$, as in part (3) of the
commuting blowups theorem.
Thus $[M \times M^2_x;Q;\tilde{A};\tilde{\tilde{R}}] = [M \times M^2_x;R;\tilde{A}^\prime;\tilde{\tilde{Q}}]$.
If we consider these submanifolds, we can see that
$\tilde{A}=\Delta_{1,y} \cap \mathcal{V}_x$ and
$\tilde{\tilde{R}} =  \Delta_{1,y} \cap \mathcal{E}_{1, x}$.
So before the switch we have exactly the blow-up in equation \eqref{eq:midphiblowup}.
So that sequence of blowups is equal to
\begin{equation}
\label{eq:tripleMy1}
= [[[M \times M^2_x;(M \times \Delta_y)\cap \mathcal{E}_{1, x}];
(M \times \Delta_y)\cap \mathcal{E}_{1, x} \cap V;
V]_{a_1}; \cdots]_{*}.
\end{equation}
When we consider it, we
then notice that $[M \times M^2_x;(M \times \Delta_y)\cap \mathcal{E}_{1, x}] =
M \times M^2_y$.  With this in mind, we can
also rewrite $\tilde{A}^\prime= \mathcal{E}_{1, y} \cap (\partial{M} \times M^2_y)$
and
$\tilde{\tilde{Q}} = (\tilde{\mathcal{E}}_{1,x} \smallsetminus \stackrel{\circ}{\mathcal{E}}_{1,y})\cap (\partial{M} \times M^2_y)$
So we have
\begin{equation}
\label{eq:tripleMy2}
= [[M \times M^2_y;
\mathcal{E}_{1, y} \cap (\partial{M} \times M^2_y);
(\tilde{\mathcal{E}}_{1,x} \smallsetminus\stackrel{\circ}{\mathcal{E}}_{1,y})\cap (\partial{M} \times M^2_y)]_{a_1}; \cdots]_{*}.
\end{equation}
Thus we have a natural map
$\pi^2_{y,1}:M^3_y \rightarrow M^2_y$ that is the composition of several blow down maps
from $M^3_y$ to $M \times M^2_y$ with the natural projection $M \times M^2_y \rightarrow M^2$.
Thus this map is a b-fibration.  Again by construction this map fits into the necessary commutative
diagram.

Finally we consider the \bfa-triple space.
$$
M^3_z = [[[[M^3_y; \del\Delta^3_z]_{a_2};
\Delta_{i,z} \cap \mathcal{V}_y ]_{a_2};
\Delta_{i,z} \cap \mathcal{G}_{i, y}]_{a_2};
\Delta_{i,z} \cap \mathcal{E}_{i, y}]_{a_2}
$$
As for the $b$ and $\phi$ triple spaces, notice that the submanifolds of each type $\Delta_{i,z} \cap \mathcal{V}_y$,
$\Delta_{i,z} \cap \mathcal{G}_{i, y}$ and $\Delta_{i,z} \cap \mathcal{E}_{i, y}$
are disjoint from others of the same type
for $i=1,2,3$. Further, if the indices $i$ and $i^\prime$ are distinct then manifolds of two
different types are disjoint from each other.  Thus by
part (1) of Theorem \ref{th:commblowups}, we can put the blow-ups with $i=1$ first, then
put the others at the end:
$$
M^3_z = [[[[[M^3_y; \del\Delta^3_z]_{a_2};
\Delta_{1,z} \cap \mathcal{V}_y ]_{a_2};
\Delta_{1,z} \cap \mathcal{G}_{1, y}]_{a_2};
\Delta_{1,z} \cap \mathcal{E}_{1, y}]_{a_2}, \cdots ]_{a_2}.
$$
Now, since $\del\Delta^3_z \subset
\Delta_{1,z} \cap \mathcal{V}_y $, by part (2) of the theorem,
we can switch these blowups.  Then in addition, the inverse image in the $\Delta_{1,z} \cap \mathcal{V}_y$ blowup of
$\del\Delta^3_z$
is disjoint from submanifolds blown up in the next two blowups,
so we can use part (1) to move it to the end:
$$
M^3_z = [[[[[M^3_y; \Delta_{1,z} \cap \mathcal{V}_y ]_{a_2};
\Delta_{1,z} \cap \mathcal{G}_{1, y}]_{a_2};
\Delta_{1,z} \cap \mathcal{E}_{1, y}]_{a_2};
\del\Delta^3_z]_{a_2};\cdots ]_{a_2}.
$$
To simplify notation, absorb this last blowup into the $\cdots]_{a_2}$.
Now we can substitute for $M^3_y$ using the step above to get
$$
M^3_z = [[[[[[M \times M^2_y;
\mathcal{E}_{1, y} \cap (\partial{M} \times M^2_y);
(\tilde{\mathcal{E}}_{1,x} \smallsetminus\stackrel{\circ}{\mathcal{E}}_{1,y})\cap (\partial{M} \times M^2_y)]_{a_1}; \cdots]_{*} ; \hspace{1in}
$$ $$ \hspace{1in}
\Delta_{1,z} \cap \mathcal{V}_y ]_{a_2};
\Delta_{1,z} \cap \mathcal{G}_{1, y}]_{a_2};
\Delta_{1,z} \cap \mathcal{E}_{1, y}]_{a_2}; \cdots ]_{a_2}.
$$
Here the $\cdots$ coming from $M^3_y$ include the blowup of
$\del\Delta^3_y$ as well as blowups
of submanifolds of $\mathcal{E}_{i,x}$ for $i=1,2$.
Now again using part (1) of the theorem, we can move the blowups after $\del \Delta^3_y$
to the end of our listed blowups to obtain:
\begin{equation}
\label{eq:aqar1}
M^3_z = [[[[[M \times M^2_y;
\mathcal{E}_{1, y} \cap (\partial{M} \times M^2_y);
(\tilde{\mathcal{E}}_{1,x} \smallsetminus\stackrel{\circ}{\mathcal{E}}_{1,y})\cap (\partial{M} \times M^2_y)]_{a_1};
\del\Delta^3_y]_{a_1}; \hspace{1in}
\end{equation}
$$ \hspace{1in}
\Delta_{1,z} \cap \mathcal{V}_y ]_{a_2};
\Delta_{1,z} \cap \mathcal{G}_{1, y};
\Delta_{1,z} \cap \mathcal{E}_{1, y}]_{a_2}; \cdots ]_{*}.
$$
At this point, we need to again apply part (3) of Theorem \ref{th:commblowups} to the blowups
of $\del\Delta^3_y$, $\Delta_{1,z} \cap \mathcal{V}_y$ and $\Delta_{1,z} \cap \mathcal{G}_{1, y}$.
It is easier to describe the relevant submanifolds of $S=[M \times M^2_y;
\mathcal{E}_{1, y} \cap (\partial{M} \times M^2_y);
(\tilde{\mathcal{E}}_{1,x} \smallsetminus\stackrel{\circ}{\mathcal{E}}_{1,y})\cap (\partial{M} \times M^2_y)]_{a_1}$ if we use the equivalence of this space to the first blowups in equation (\ref{eq:tripleMy1}):
$$
[M \times M^2_x;V;\Delta_{1,y} \cap \mathcal{V}_x;
\Delta_{1,y} \cap \mathcal{E}_{1,x}].
$$
The new faces created by the three blowups here are, in order, $\calV_x$, $\calG_{1,y}$ and $\calE_{1,y}$.
Let $Q$ be the inverse image in $S$ of $\del \Delta^3_y$. This is the submanifold whose blow-up creates $\calV_y$. Let $P=$ the inverse image of $\Delta_{1,z}\cap \calG_{1,y}$. Then $P,Q$ are submanifolds of $\calG_{1,y}\subset S$ that intersect cleanly in $R=P\cap Q$.
By part (3) of Theorem \ref{th:commblowups} $[S;Q]_{a_1}; R']_{a_2};P']_{a_2} = [S;P]_{a_2};R'']_{a_1};Q'']_{a_1}$. The left hand side corresponds to the first five blowups in \eqref{eq:aqar1}.  So we may
replace these with the right hand side.  Further, the submanifolds $R''$, $Q''$ are disjoint from $\Delta_{1,z} \cap \mathcal{E}_{1, y}$,
so we may replace (\ref{eq:aqar1}) by
$$
M^3_z =
[[[M \times M^2_y;
\mathcal{E}_{1, y} \cap (\partial{M} \times M^2_y);
(\tilde{\mathcal{E}}_{1,x} \smallsetminus\stackrel{\circ}{\mathcal{E}}_{1,y})\cap (\partial{M} \times M^2_y)]_{a_1};
\Delta_{1,z} \cap \calG_{1,y};
\Delta_{1,z} \cap \mathcal{E}_{1, y}]_{a_2}, \Delta_{1,z} \cap \mathcal{V}_y; \cdots ]_{*}.
$$
Here $\calG_{1,y}$ is the front face arising from the blowup of $\calE_{1,y}\cap (\partial M \times M_y^2)$.
Notice that $ (\tilde{\mathcal{E}}_{1,x} \smallsetminus\stackrel{\circ}{\mathcal{E}}_{1,y})\cap (\partial{M} \times M^2_y)$ is disjoint from the next two submanifolds blown up in this list, so we can
use part (1) of the theorem to put this blowup at the end and absorb it and the blowup of
$\Delta_{1,z} \cap \mathcal{V}_y$ into the
dots.  So now our space looks like:
$$
M^3_z =
[[[M \times M^2_y;
\mathcal{E}_{1, y} \cap (\partial{M} \times M^2_y)]_{a_1};
\Delta_{1,z} \cap \calG_{1,y};
\Delta_{1,z} \cap \mathcal{E}_{1, y}]_{a_2}, \cdots ]_{*}.
$$
We are almost done now.  The next step is to use part 3 of Theorem \ref{th:commblowups}
again to rearrange these blowups from the form $[M \times M^2_y;Q]_{a_1};\tilde{A}]_{a_2};\tilde{\tilde{R}}]_{a_2}$
to the form $[M \times M^2_y;R]_{a_2};\tilde{A}^\prime]_{a_1};\tilde{\tilde{Q}}]_{a_1}$.  This yields
$$
M^3_z =
[[[M \times M^2_y;\Delta_{1,z} \cap \mathcal{E}_{1, y}]_{a_2};
\mathcal{E}_{1,z} \cap (\partial M \times M_y^2);
\mathcal{E}_{1, y} \cap (\partial{M} \times M^2_y)]_{a_1};  \cdots ]_{*}.
$$
Finally, we notice that
$M \times M^2_z = [M \times M^2_y;\Delta_{1,z} \cap \mathcal{E}_{1, y}]_{a_2}$,
so we have:
$$
M^3_z =
[[M \times M^2_z;
\mathcal{E}_{1,z}]_{a_2}; \calE_{1,z} \cap (\partial{M} \times M^2_y);
{\mathcal{E}}_{1, y} \cap (\partial{M} \times M^2_y)]_{a_1};  \cdots ]_{*}.
$$
That is, we again have a natural map $\pi^3_{z,1}:M^3_z\rightarrow M^2_z$ that is the
composition of several blow down maps to $M \times M^2_z$ followed by the
projection on to $M^2_z$, and is therefore a b-fibration.  It again fits naturally into the
commutative diagram by construction.
\end{proof}





\medskip\noindent{\bf Exponent matrices.}

In order to prove the main mapping theorem we need to calculate
the exponent matrices for the projections from the various blown up double spaces.
Since $M$ has only one boundary hypersurface, the exponent matrices for maps
from the blown up double spaces to $M$ are just vectors indexed by the  boundary hypersurfaces of the blown up space.
\begin{lemma}
\label{lem:e2matrix}
The exponent vectors of the $\pi_x$-projections are, corresponding to the hypersurfaces $\rf$, $\lf$, $\ff_x$ of $M^2_x$,
given by $e_{\pi_{x,l}}=(0,1,1)$ and $e_{\pi_{x,r}}=(1,0,1)$.\\
The exponent vectors of the $\pi_y$-projections are, corresponding to the hypersurfaces
$\rf$, $\lf$, $\ff_{yx}$, $\ff_{y}$ of $M^2_y$, given by
$e_{\pi_{y,l}}=(0,1,1,1)$ and $e_{\pi_{y,r}}=(1,0,1,1)$.\\
The exponent vectors of the $\pi_z$-projections are, corresponding to the hypersurfaces
$\rf$, $\lf$, $\ff_{zx}$, $\ff_{zy}$, and $\ff_z$ of $M^2_z$, given by
$e_{\pi_{z,l}}=(0,1,1,1,1)$ and $e_{\pi_{z,r}}=(1,0,1,1,1)$.
\end{lemma}
\begin{proof}
We'll just consider the left projections, since the right ones are similar.  We know
$\pi_{*,l}(\rf) \not\subset \partial M$, so the first index is 0.  The other bhs of $M^2_*$ do
map to $\partial M$, and in each new set of local coordinates in the blowups, these
faces all have bdf $x$.  The maps $\pi_{*,l}$ just write the original coordinates
in terms of the blown up coordinates, so $\pi_{*,l}(x, \cdots) = (x, \cdots)$ in
each case, and thus the indices are all equal to 1.
\end{proof}

In order to prove the main composition theorem we need to calculate the exponent
matrices for the projections from the triple to the double spaces.  By the same proof
as for the projections from the double spaces to $M$, these will be matrices consisting
of just 1's and 0's that record simply if a given bhs $G$ in the domain maps into
the given bhs $H$ in the range. Thus, the exponent matrix is completely determined by the map $\overline{\pi}_{*,i}^3$ defined before Definition \ref{def:b-fibration}.
For instance, for $\overline{\pi}_1^3$ we get:
\begin{equation}
\label{eq:facemap}
\begin{aligned}
(\overline{\pi}^3_{z,1})^{-1}(M^2_z)&= \{ M^3_z , H_1\} \\
(\overline{\pi}^3_{z,1})^{-1}(\rf )&= \{ H_3, \mathcal{E}_{2,x} , \mathcal{E}_{2,y}, \mathcal{E}_{2,z} \} \\
(\overline{\pi}^3_{z,1})^{-1}( \lf)&=\{   H_2, \mathcal{E}_{3,x} , \mathcal{E}_{3,y}, \mathcal{E}_{3,z}\}\\
(\overline{\pi}^3_{z,1})^{-1}( \ff_{zx})&=\{\mathcal{V}_{x} ,
\mathcal{E}_{1,x}, \mathcal{G}_{2,y} , \mathcal{G}_{2,z}, \mathcal{G}_{3,y}, \mathcal{G}_{3,z}\} \\
(\overline{\pi}^3_{z,1})^{-1}( \ff_{zy})&=\{\mathcal{V}_{y}, \mathcal{E}_{1,y},\mathcal{G}_{1,y},
\mathcal{F}_{2,z},\mathcal{F}_{3,z}   \} \\
(\overline{\pi}^3_{z,1})^{-1}( \ff_{z})&=\{  \mathcal{V}_{z}, \mathcal{E}_{1,z},\mathcal{G}_{1,z},
\mathcal{F}_{1,z}\}
\end{aligned}
\end{equation}

For reference, we record the analogous results
for the b and $\phi$ triple space maps here.
\begin{equation}
\label{eq:facemap2}
\begin{aligned}
(\overline{\pi}^3_{x,1})^{-1}(M^2_x)&= \{ M^3_x , H_1\} \\
(\overline{\pi}^3_{x,1})^{-1}(\rf )&=\{   H_3, \mathcal{E}_{2,x} \} \\
(\overline{\pi}^3_{x,1})^{-1}( \lf)&=\{ H_2, \mathcal{E}_{3,x} \} \\
(\overline{\pi}^3_{x,1})^{-1}( \ff_x)&=\{\mathcal{V}_{x} , \mathcal{E}_{1,x}\}
\end{aligned}
\end{equation}

\begin{equation}
\label{eq:facemap3}
\begin{aligned}
(\overline{\pi}^3_{y,1})^{-1}(M^2_y)&= \{M^3_y , H_1\} \\
(\overline{\pi}^3_{y,1})^{-1}(\rf )&=\{   H_3, \mathcal{E}_{2,x} , \mathcal{E}_{2,y}\} \\
(\overline{\pi}^3_{y,1})^{-1}( \lf)&=\{ H_2, \mathcal{E}_{3,x} , \mathcal{E}_{3,y}\} \\
(\overline{\pi}^3_{y,1})^{-1}( \ff_{yx})&=\{\mathcal{V}_{x} ,
\mathcal{E}_{1,x}, \mathcal{G}_{2,y}, \mathcal{G}_{3,y}\} \\
(\overline{\pi}^3_{y,1})^{-1}( \ff_y)&=\{\mathcal{V}_{y}, \mathcal{E}_{1,y},\mathcal{G}_{1,y}\}
\end{aligned}
\end{equation}

\subsection{Full $\bfa$-calculus, polyhomogeneous mapping and composition theorems}
\label{subsec.polyhom map}
We will formulate the mapping and composition theorems for operators in the full calculus since it is almost no extra work and will be used in the sequel to this paper. Before we introduce the full calculus, we need to define certain half density bundles on $M^2_z$. Choosing a suitable half density bundle will make the bookkeeping easier.

We first discuss how densities behave under quasihomogeneous blowup. This and the formulas in Theorems \ref{th:basicmap} and \ref{th:basiccomp} below will be most transparent if we start our discussion with b-densities rather than smooth densities.
Recall that we denote the set of b-densities on an $n$-dimensional manifold with corners $X$ by $C^\infty(X,\Omega_b)$.
In local coordinates near a codimension $k$ corner in $X$ given by $x_1=\cdots = x_k=0$, a b-density is of the form $|\frac{dx_1}{x_1}\dots \frac{dx_k}{x_k}\,dy_1\dots dy_{n-k}|$ times a smooth function.  Under quasihomogeneous blowup $\beta:[X;Y]_a\to X$ of the boundary submanifold $Y=\{x_1=\dots=x_r=0,\,y_1=\dots=y_m=0\}$ to order $a$ this lifts, in the coordinates \eqref{eq:proj coord quasihom ff}, to $x_i^{am}|\frac{dx_i}{x_i}\prod_{j\neq i}\frac{dW_j}{W_j}\prod_l dV_l|$.
Together with a similar calculation in the coordinates \eqref{eq:proj coord quasihom bd} this means that
$$C^\infty([X;Y]_a,\beta^*\Omega_b) = \rho^{am} C^\infty([X;Y]_a,\Omega_b)$$
naturally, for any boundary defining function $\rho$ of the front face.
We apply this to the successive blowups of $M^2$ in $\partial\Delta_x=(\partial M)^2$ (with $a=1, m=0$), in $\partial\Delta_y$ (with $a=a_1, m=1+b$) and then in $\partial\Delta_z$ (with $a=a_2,m=1+b+f_1$), where in the latter two blow-ups the previous vanishing factors are also lifted. This yields the following result. Denote
\begin{equation}
\label{eq:half densities pullback}
\Omega_{\bfa,0}=\Omega_{\bfa,0}(M_z^2) := \beta_\bfa^* (\Omega_b(M^2)),
\end{equation}
that is, sections of $\Omega_{\bfa,0}(M_z^2)$ are lifts of b-densities on $M^2$. Let
\begin{equation}
\label{eq:def gamma}
\gamma_y = (1 + b) a_1, \ \gamma_z = (1+b)(a_1+a_2) + f_1 a_2
\end{equation}
and let $\rho_x,\rho_y,\rho_z,\rho_{\lf},\rho_{\rf}$ be defining functions for the faces $\ff_{zx},\ff_{zy},\ff_z,\lf,\rf$ of $M^2_z$ respectively. Then we obtain
\begin{equation}
\label{eq:Omega a 0}
 C^\infty(M^2_z, \Omega_{\bfa,0}) = \rho_y^{\gamma_y}\rho_z^{\gamma_z} C^\infty(M^2_z,\Omega_b).
\end{equation}
For the density bundle used for the small calculus, \eqref{eq:def omegatilde}, we have an extra factor $\rho_{\lf}^{-\gamma_z}\rho_{\rf}^{-\gamma_z}$ coming from the $\bfa$-densities on $M\times M$. This factor lifts also, and we obtain
\begin{equation}
\label{eq:omegatilde powers}
C^\infty(M^2_z,\Omegatilde_\bfa) = \rho_{\lf}^{-\gamma_z} \rho_{\rf}^{-\gamma_z} \rho_x^{-2\gamma_z}\rho_y^{-2\gamma_z+\gamma_y}\rho_z^{-\gamma_z} C^\infty (M^2_z,\Omega_b).
\end{equation}
For the mapping and composition theorems in the full calculus it is more natural to use half-densities associated with a density bundle we denote $\Omega_\bfa=\Omega_\bfa (M^2_z)$ and which satisfies
\begin{equation}
\label{eq:half densities on M2z}
C^\infty(M^2_z, \Omega_\bfa) = \rho_y^{-\gamma_y}\rho_z^{-\gamma_z}
C^\infty(M^2_z,\Omega_b).
\end{equation}
Invariantly it may be given as $\Omega_\bfa = (\Omega_b)^2\otimes  \Omega_{\bfa,0}^{-1}$.
Note that kernels of operators in the small calculus may be thought of as having coefficients in $\Omega_\bfa^{1/2} (M^2_z)$
rather than $\Omegatilde_\bfa^{1/2} (M^2_z)$ since the scaling powers in \eqref{eq:omegatilde powers} and \eqref{eq:half densities on M2z} are the same at $\ff_z$ and since the kernels for the small \bfa-calculus vanish to infinite order at all other faces of $M^2_z$.
\begin{definition}
\label{def.full calculus}
Let $M$ be an $\bfa$-manifold and let $\calJ$ be an index family for $M^2_z$. The {\em full \bfa-pseudodifferential calculus on $M$ with index family $\calJ$} is defined as the set of operators with kernels in (for $m\in\RR$)
$$ \Psi^{m,\calJ}_\bfa (M) := I^{m,\calJ} (M^2_z,\Delta; \Omega_\bfa^{1/2} (M^2_z)).$$
\end{definition}
Note that a different choice of half density bundle would have resulted only in a shift of the index sets in $\calJ$. The given normalization is particularly convenient below, and is natural
in that it behaves in the same way with respect to the $x$ and $y$ front faces as with respect to the $z$ front face. For the index set $\calG$ in Definition \eqref{def:small calculus}, $\Psi^{m,\calG}_{\bfa}(M)=\Psi^m_{\bfa}(M)$.

We can now analyze when and how operators in the full calculus act on polyhomogeneous functions.
\begin{theorem}
\label{th:basicmap}
Let $u$ be a b-half-density on $M$ that is polyhomogeneous at $\del M$ with
index set $I$.  Let $P\in \Psi^{m, \mathcal{J}}_\bfa(M)$.  If $\inf(\mathcal{J}(\rf)+I) >0$
then the pushforward in \eqref{eq:action} converges and defines $Pu$, which will be a
b-half-density on $M$ polyhomogeneous at $\del M$ with index set
$$
\calJ(\lf)\, \overline{\cup}\,\left(\calJ(\ff_{zx})+I\right)\,  \overline{\cup}\, (\calJ(\ff_{zy})+I)\, \overline{\cup}\, (\calJ(\ff_{z})+I).
$$
\end{theorem}
\begin{proof}
We list index families for $M^2_z$ in the order $\rf,\lf,\ff_{zx},\ff_{zy},\ff_{z}$ as in Lemma \ref{lem:e2matrix}. Choose an auxiliary $\nu\in C^\infty(M,\Omega_b^{1/2})$. Then for $u\in\calA^I (M,\Omega_b^{1/2})$ we have
$(\pi_{z,l})^*\nu\cdot (\pi_{z,r})^*u \in \calA^{\calI}(M^2_z,\Omega_{\bfa,0}^{1/2})$, where $\calI = (I,0,I,I,I)$ since $e_{\pi_{z,r}} = (1,0,1,1,1)$ and by definition of $\Omega_{\bfa,0}^{1/2}$. Multiplying this with $K_P\in I^{m,\calJ} (M^2_z,\Delta; \Omega_\bfa^{1/2})$ we obtain an element of
$I^{m,\calI+\calJ} (M^2_z,\Delta; \Omega_b)$ since the $\rho$-factors in the densities cancel out.
The result now follows from the Pushforward Theorem \ref{thm:genpushforward}.
\end{proof}



\begin{theorem}
\label{th:basiccomp}
Let $P \in \Psi^{m, \mathcal{I}}_\bfa(M)$ and $Q\in \Psi_\bfa^{m', \mathcal{J}}(M)$.
If $\inf(\mathcal{I}(\rf) + \mathcal{J}(\lf))>0$ then the pushforward in \eqref{eq:comp} converges, so $P\circ Q$ is defined, and
$P\circ Q \in \Psi_\bfa^{m+m', \mathcal{K}}(M)$, where
\begin{equation}
\label{eq:calK calc}
\mathcal{K} = (\pi^3_{z,2})_\#[(\pi^3_{z,3})^\#(\mathcal{I}) + (\pi^3_{z,1})^\#(\mathcal{J})+W_\bfa] - w_\bfa
\end{equation}
for certain weight vectors $W_\bfa$ on $M^3_z$ and $w_\bfa$ on $M^2_z$ as explained in the proof below.
In particular,
$$
\mathcal{K}(\ff_z) = (\calI(\ff_z) + \calJ(\ff_z))\overline{\cup}
(\calI(\lf) + \calJ(\rf)+ \gamma_z)
\overline{\cup}(\calI(\ff_{zx}) + \calJ(\ff_{zx}) + \gamma_z)
\overline{\cup}(\calI(\ff_{zy}) + \calJ(\ff_{zy}) + \gamma_z - \gamma_y).
$$
If $P\in \Psi^{m, \mathcal{I}}_\bfa(M)$ then its adjoint $P^*\in \Psi^{m, \mathcal{I}'}_\bfa(M)$ where $\calI'$ is obtained from $\calI$ by switching the index sets at $\lf,\rf$ and leaving all other index sets the same.
\end{theorem}

\begin{proof}
The triple projection maps, $\pi^3_{z,i}$, are b-fibrations and transversal to the diagonal $\Delta \subset M^2_z$.
Thus we may apply the pullback theorem to $K_P$ and $K_Q$ and obtain distributions conormal with respect to the double diagonals in $M^3_z$.
The double diagonals in $M^2$
intersect transversally at the triple diagonal in $M^3$.  Then by Lemma \ref{lem:transversal}, the lifted
double diagonals in $M^3_z$ still intersect transversally at the lifted triple diagonal.  Further,
each of the three double diagonals in $M^3$ maps diffeomorphically to
$M^2$ under the opposite two
triple projection maps. This remains true for their lifts in the interior, since the blow-down maps are diffeomorphism there. Furthermore, the lifted diagonal (say the 13-diagonal) hits $\calE_{2,z},\calV_z,\calF_{2,z},\calG_{2,z}$ and $H_2$ in their interior and transversally, so from \eqref{eq:facemap} and the fact that all exponents are one we get that the whole lifted double diagonals map diffeomorphically to the whole
of the \bfa-double spaces under the opposite projection maps.
So we can apply the pushforward Theorem \ref{thm:genpushforward}.

To calculate the shifts in the index sets first recall the precise rule for composition: Choose a non-vanishing section $\nu_\bfa^{1/2}$ of $\Omega_\bfa^{1/2}(M^2_z)$, then $P\circ Q$ is determined by $ (P\circ Q)\cdot \nu_\bfa^{1/2} = (\pi^3_{z,2})_* [(\pi^3_{z,3})^*P\cdot (\pi^3_{z,1})^*Q \cdot (\pi^3_{z,2})^*\nu_\bfa^{1/2} ]$; in other words, with
$\mu_{\bfa} :=  (\pi^3_{z,1})^*\nu_{\bfa}^{1/2} \cdot (\pi^3_{z,2})^*\nu_{\bfa}^{1/2} \cdot (\pi^3_{z,3})^*\nu_{\bfa}^{1/2} $ and $P= p\nu_\bfa^{1/2}$, $Q=q\nu_\bfa^{1/2}$ one has
$ (P\circ Q)\cdot \nu_\bfa^{1/2} = (\pi^3_{z,2})_* [(\pi^3_{z,3})^*p\cdot (\pi^3_{z,1})^*q \cdot \mu_\bfa ] $. In order to apply the pushforward theorem, we need to relate $\mu_\bfa$ to b-densities on $M^3_z$. For a density bundle $\Omega_*(X)$ on a manifold with corners $X$ with $C^\infty (X,\Omega_*(X)) = \rho^w C^\infty(X,\Omega_b(X))$, where $\rho$ is a family of boundary defining functions and $w$ assigns an integer to each boundary hypersurface, call $w=w(\Omega_*(X))$ its {\em weight vector}. Now $\mu_\bfa$ is a smooth non-vanishing section of
$\Omega_\bfa(M^3_z) := (\pi^3_{z,1})^*\Omega_\bfa^{1/2}(M^2_z) \otimes (\pi^3_{z,2})^*\Omega_\bfa^{1/2}(M^2_z) \otimes (\pi^3_{z,3})^*\Omega_\bfa^{1/2}(M^2_z) $, so if we set $W_\bfa =w(\Omega_\bfa(M^3_z))$ and $w_\bfa = w(\Omega_{\bfa}(M^2_z))$ then the push-forward theorem gives
\eqref{eq:calK calc} where adding a weight vector to an index family means shifting the $z$-component in Definition \ref{def:indfam} by the weights.

To determine $W_{\bfa}$  we compare it with $W_{\bfa,0}= w(\Omega_{\bfa,0}(M^3_z))$
where $\Omega_{\bfa,0}(M^3_z) := (\beta^3_\bfa)^*\Omega_b(M^3)$. On the one hand, this can be determined from the definition of the \bfa-triple space using the considerations leading to \eqref{eq:Omega a 0}, and this gives
$W_{\bfa,0} = (2\gamma_y,\gamma_y,\gamma_y; 2\gamma_z,\gamma_y+\gamma_z,\gamma_z,\gamma_z)$, with respect to the ordering $\calV_y,\calG_{i,y},\calE_{i,y}; \calV_z, \calF_{i,z},\calG_{i,z},\calE_{i,z}$ of boundary hypersurfaces of $M^3_z$ (the weights are zero at the other faces).
On the other hand, we have $\Omega_{\bfa,0}(M^3_z) := (\pi^3_{z,1})^*\Omega_{\bfa,0}^{1/2}(M^2_z) \otimes (\pi^3_{z,2})^*\Omega_{\bfa,0}^{1/2}(M^2_z) \otimes (\pi^3_{z,3})^*\Omega_{\bfa,0}^{1/2}(M^2_z) $
since this results from pulling back the obvious identity $\Omega_b(M^3) := (\pi^3_{1})^*\Omega_b^{1/2}(M^2) \otimes (\pi^3_{2})^*\Omega_b^{1/2}(M^2) \otimes (\pi^3_{3})^*\Omega_b^{1/2}(M^2) $ by $\beta^3_\bfa$.
Now the pull-back theorem shows that with $w_0 = \frac12 (w_{\bfa}-w_{\bfa,0})$, where $w_{\bfa,0}= w(\Omega_{\bfa,0}(M^2_z))$, we have $W_{\bfa}-W_{\bfa,0} = (\pi^3_{z,1})^\# w_0 + (\pi^3_{z,2})^\# w_0 + (\pi^3_{z,3})^\# w_0 $. With $w_{\bfa,0} = (\gamma_y,\gamma_z)$ (with respect to the faces $\ff_{zy},\ff_z$ of $M^2_z$, the weights being zero at the other faces) from \eqref{eq:Omega a 0} and $w_\bfa = (-\gamma_y,-\gamma_z)$ from \eqref{eq:half densities on M2z} we get $w_0=(-\gamma_y,-\gamma_z)$ and then $W_{\bfa}-W_{\bfa,0} = -(3\gamma_y,\gamma_y,\gamma_y; 3\gamma_z,2\gamma_y + \gamma_z, \gamma_z,\gamma_z)$ from \eqref{eq:facemap}. Together, we obtain $W_\bfa = (W_{\bfa}-W_{\bfa,0}) + W_{\bfa,0} = -(\gamma_y,0,0;-\gamma_z,-\gamma_y,0,0).$ This gives the result.

The statement about adjoints is clear since the kernel of the adjoint is obtained by switching the factors in $M\times M$ and taking the complex conjugate.
\end{proof}

In particular, we have:

\begin{corollary}
\label{th:smallcomp}
Operators in the small \bfa-calculus may be composed. For $c,c'\in\RR\cup\{\infty\}$ and $m,m'\in\RR\cup\{-\infty\}$ we have
$$
x^c\Psi^m_\bfa \MO \cdot x^{c^\prime}\Psi^{m'}_\bfa \MO \subset x^{c+c^\prime}\Psi^{m+m'}_\bfa \MO.
$$
Also, $x^c\Psi^m_\bfa \MO$ is closed under taking adjoints.
\end{corollary}
\begin{proof}
This follows directly from Theorem \ref{th:basiccomp} by choosing the index sets at $\ff_z$
to be $(c+\mathbb{N}_0)\times\{0\}$ and $(c' + \mathbb{N}) \times \{0\}$, respectively,
and all other index sets to be $\emptyset$.
\end{proof}

Notice also

\begin{corollary}
\label{cor:invunderxconj}
The small \bfa-calculus is invariant under conjugation by powers
of the boundary defining function $x$.
\end{corollary}
\begin{proof}
This can be seen by applying the composition theorem, Theorem \ref{th:basiccomp}, or more simply by observing that multiplication by $(x'/x)^\alpha = t^\alpha$ leaves $I^{m,\mathcal{G}}(M^2_z,\Delta; \Omegatilde_\bfa^{1/2}(M^2_z))$ in Definition \ref{def:small calculus} invariant since $t$ has only a finite order zero and pole at $\rf$ and $\lf$, respectively, and is non-vanishing otherwise.
\end{proof}

The principal symbol and
normal operator behave well under composition:

\begin{lemma}
\label{lem:symbolcomp}
The symbol map $\Psi^m_{\bfa}(M) \stackrel{{}^\bfa \sigma_m}
{\longrightarrow} S^{[m]}({}^\bfa T^*M)$ and the normal operator map $\Psi^m_{\bfa}(M) \stackrel{N_\bfa}
{\longrightarrow} \Psi^m_{\sus(^\bfa N'B_1)-\phi_2}(\del M)$ are star algebra homomorphisms.  That is,
if $P \in \Psi^m_{\bfa}(M)$
and $Q \in \Psi^{m'}_{\bfa}(M)$, then
${}^\bfa\sigma_{m+m'} (PQ) = {}^\bfa\sigma_m(P){}^\bfa \sigma_{m'}(Q)$ and $N_\bfa(PQ)=N_\bfa(P)N_\bfa(Q)$, and ${}^\bfa\sigma_m(P^*) = \overline{{}^\bfa\sigma_m(P)}$ and $N_\bfa(P^*)=N_\bfa(P)^*$.
\end{lemma}
\begin{proof}
Consider first the symbol map. By continuity, it suffices to show that for any interior diagonal point
$p\in \interior{\Delta}_z \subset M^2_z$ we have
${}^\bfa \sigma_m(PQ)(p,\cdot)= {}^\bfa\sigma_m(P)(p, \cdot) {}^\bfa \sigma_{m'}(Q)(p, \cdot).$
Let $\chi(x)$ be a cutoff function on $M^2_z$ that equals one near the front face, is supported near the front face, and equals zero near the point $p$.
Use this to break the kernels $K_P$ and $K_Q$ down into pieces $K_{P,1}$ and $K_{Q,1}$
supported near the front face and $K_{P,2}$ and $K_{Q,2}$ away from it.
When we lift these pieces to the double space to compose them, the lifts from the right and
left respectively of $K_{P,1}$ and $K_{Q,1}$ will vanish in a neighborhood of the preimage of $p$
along the lift of the diagonal from the middle. Thus of the four terms in the composition of the
two sums, only the composition of  $K_{P,2}$ and $K_{Q,2}$ will
contribute to the symbol at $p$.  But then we are away from the boundary, and there the standard
argument implies that the symbol is multiplicative.
%

For the normal operator we know from Definition \ref{def.prop.normal family} that for $u \in C^\infty(\del M)$
and for any smooth extension $\tilde{u}$ to the interior, we have
\begin{align*}
\Nhat_\bfa (PQ)(p, \mu) u&= \left[ e^{-ig} PQ (e^{ig} \utilde) \right]_{|F_{2,p}} \\
{}&=  \left[ e^{-ig} Pe^{ig}\left[e^{-ig}Q (e^{ig} \utilde)\right] \right]_{|F_{2,p}}.
\end{align*}
But $\left[e^{-ig}Q (e^{ig} \utilde)\right]$ is an extension to the interior of
$\left[e^{-ig}Q (e^{ig} \utilde)\right]_{|F_{2,p}} = \Nhat_\bfa(Q)(p,\mu)u$,
so overall,
$$
\Nhat_\bfa (PQ)(p, \mu) u = \left( \Nhat_\bfa(P)(p,\mu) \circ \Nhat_\bfa(Q)(p,\mu)\right) u.
$$
Here the composition is simply multiplication in the transform variables $\tau, \eta$ and $\xi$,
and is composition of the differential operators in $w$.  This
implies after transforming back that $N_\bfa(PQ) = N_\bfa(P) \circ N_\bfa(Q)$.

The statements on adjoints are clear from the definitions and the standard symbol formulas.
\end{proof}

\section{Parametrix construction and proof of main theorems}
\label{sec:param}
We can now construct the desired parametrix for a fully elliptic \bfa-pseudodifferential operator.

\begin{theorem}
\label{th:param}
Let $P\in \Psi^m_\bfa (M)$ be a fully elliptic \bfa-operator over an \bfa-manifold $M$.  Then there exists an \bfa-operator $Q$ of order $-m$ over $M$ such that
$PQ = I +R_1$ and $QP = I + R_2$, where $R_i \in x^{\infty} \Psi^{-\infty}_{\bfa} \MO$.
\end{theorem}

\begin{proof}
This is proved by the `usual' parametrix construction, applied twice.
First, since $P$ is elliptic, its principal symbol ${}^\bfa\sigma_m(P)$
is invertible, with inverse in  $S^{[-m]}(^\bfa T^*M)$. By exactness of the symbol sequence for order $-m$, there is  $Q_1 \in \Psi^{-m}_\bfa$
such that ${}^\bfa\sigma_{-m}(Q_1) = {}^\bfa\sigma_m(P)^{-1}$.  By Lemma \ref{lem:symbolcomp},
we have ${}^\bfa\sigma_0(PQ_1-I) = {}^\bfa\sigma_m(P) {}^\bfa\sigma_m(P)^{-1} - 1 = 0$, so by the exact sequence for order zero, $PQ_1 = I + R_1$ for some $R_1 \in \Psi^{-1}_\bfa$.
The remainder can now be improved using a kind of Neumann series. That is, let $S$ be an asymptotic sum for $I-R_1+R_1^2-+\dots$, which exists since $R_1^n\in \Psi^{-n}_\bfa$, and set $Q_s=Q_1 S$. Then $PQ_s = (I+R_1)S = I + R_s$ with $R_s \in \Psi^{-\infty}_\bfa $, so $Q_s$ is a `small parametrix'.

To obtain a full parametrix, we first want to improve $Q_s$ to $Q_s+Q'$ such that the remainder $R_{f,1}$ in
$P(Q_s+Q')=I+R_{f,1}$ is both in $\Psi^{-\infty}_\bfa $ and has $N_\bfa(R_{f,1})=0$.
Subtracting $PQ_s =I+R_s$ from this equation gives $PQ'=R_{f,1}-R_s$, and applying $N_\bfa$ to this and using that $N_\bfa$ is an algebra homomorphism we see that the condition $N_\bfa(R_{f,1})=0$ is equivalent to $N_\bfa(P) N_\bfa(Q') = - N_\bfa(R_s)$.
Now $N_\bfa(P)$ is invertible by full ellipticity, so we can rewrite this as $N_\bfa(Q')=-N_\bfa(P)^{-1} N_\bfa(R_s)$.
Using the  exact sequence for $N_\bfa$ (with $m=-\infty$) we see that $Q'\in \Psi_\bfa^{-\infty}$ can be chosen to have this normal operator.

We conclude that, with $Q_{f,1}=Q_s+Q'$, we have $PQ_{f,1} = I + R_{f,1}$ with $R_{f,1} \in x\Psi^{-\infty}_\bfa$, again by the exact sequence for $N_\bfa$.
Now by Corollary \ref{th:smallcomp}, we know that $R_{f,1}^n \in x^n\Psi^{-\infty}_\bfa$
for all $n$, so if we choose an asymptotic sum $S_f\in\Psi^0_\bfa$ of $I- R_{f,1} + R_{f,2}^2 -+ \cdots$ using Lemma \ref{lemma.normalop-exactseq}, then $Q_f = Q_{f,1}S_{f}$ satisfies $PQ_f = I+R_f$ where $R_f\in x^{\infty} \Psi^{-\infty}_{\bfa}$, so $Q_f$ is a right parametrix for $P$.

Finally, we do the same construction again except that we multiply the various $Q$'s from the left and find a left parametrix $Q_f'$, so that $Q_f'P=I+R_f'$ with  $R_f' \in x^{\infty} \Psi^{-\infty}_{\bfa}$. Then evaluating $Q_f'PQ_f$ in two ways we get
$Q_f'(I+R_f) = Q_f'PQ_f = (I+R_f')Q_f$, so $Q_f-Q_f' = Q_f'R_f - R_f' Q_f \in x^{\infty} \Psi^{-\infty}_{\bfa}$, which implies that $Q_fP$ differs from $Q_f'P$ by an element of $x^\infty\Psi^{-\infty}_\bfa$, hence is again of the form $Q_fP=I+R$ with $R\in x^\infty\Psi^{-\infty}_\bfa$, so $Q_f$ is a two-sided parametrix.
\end{proof}

\subsection{Sobolev mapping theorems}\label{section.sob}
In order to  obtain the Fredholm and regularity results in Theorem \ref{th:mainthm}
from Theorem \ref{th:param} we need to define \bfa-Sobolev spaces and see how \bfa-pseudodifferential
operators map between them.  First consider how to define $L^2$ spaces.
Denote by $L^2(M,\Omega^{1/2})$ the space of measurable half densities $u$  on $\intM$
with finite norm  $\|u\|_{L^2}:=\left(\int |u|^2 \right)^{1/2}$. Note that this integral is well-defined since $|u|^2$ is a density. Because of this invariance, and since in the interior of $M$ any density can be viewed as an \bfa-density and vice versa, putting a different half-density bundle here, for example $\Omega_\bfa^{1/2}$, would make no difference. In Subsection \ref{subsec.proofs} we will explain how to translate the boundedness results from half-densities to functions.

The mapping properties of $\bfa$-pseudodifferential operators are best expressed in terms of \bfa-Sobolev
spaces.
Before we introduce these we need a lemma.

\begin{lemma}
\label{lem.inverse in calculus}
Assume $P\in\Psi^m_\bfa(M)$ is fully elliptic and is invertible as an operator in $L^2(M,\Omega^{1/2})$. Then $P^{-1}$ is in $\Psi^{-m}_\bfa(M)$.
\end{lemma}
\begin{proof}
This is a standard argument.
Let $Q$ be a parametrix for $P$, so that $PQ=I+R$, $QP=I+R'$ with $R,R'\in\xPsiinfty$. Multiplying by $P^{-1}$ and reordering gives $P^{-1}=Q-P^{-1}R$ and $P^{-1}=Q-R'P^{-1}$. Plugging the second equation into the first yields
$P^{-1}=Q-QR + R'P^{-1}R$. Now by a simple calculation we have $R'LR\in \xPsiinfty$ for any $R,R'\in\xPsiinfty$ and any operator $L$ that is bounded on $L^2$, and this gives the claim. Here it is important that $\xPsiinfty$ can be characterized without reference to the $\bfa$-structure, as the set of operators with Schwartz kernels smooth on $M^2$ and vanishing to infinite order at $\partial(M^2)$. It is standard that these operators are continuous maps $C^{-\infty}(M) \to \Cdot^\infty(M)$.
\end{proof}

To introduce the Sobolev spaces as normed spaces we use the following lemma:

\begin{lemma}
\label{lem.exist ell inv}
For every $m \geq 0$, there is an invertible fully elliptic element $G_m \in \Psi_\bfa^m \MO$ with inverse in $\Psi_\bfa^{-m}\MO$.
\end{lemma}
\begin{proof}
Choose an $\bfa$-metric $g$ and let  $E_{m/2} \in  \Psi_\bfa^{m/2}
\MO$ be an operator with principal symbol
${}^\bfa\sigma(p,\xi)=|\xi|_{g^*(p)}^{m/2}$. Then $E_{m/2}$ is
elliptic. By Lemma \ref{lem.fullyelliptic} $N_\bfa (E_{m/2})$ is
elliptic. Hence $\Nhat_\bfa(E_{m/2})(p,\mu)$ is an elliptic element
of $\Psi^{m/2}(\phi_2^{-1}(p))$ for all $(p,\mu)\in ({}^\bfa N')^*
B_1$. Now let $G_m = E_{m/2}^*E_{m/2} + 1$.  This is  a positive
operator, and is still elliptic of order $m$ since $m\geq 0$.  By
Definition \ref{def:normal operator}, $N_\bfa(E_{m/2}^*) =
N_\bfa(E_{m/2})^*$.  Further, since the normal operator map is an
algebra homomorphism, $N_\bfa(G_m) =
N_\bfa(E_{m/2})^*N_\bfa(E_{m/2}) + I$. Hence $\Nhat_\bfa(G_m)
(p,\mu)= \Nhat_\bfa(E_{m/2})(p,\mu)^*\Nhat_\bfa(E_{m/2})(p,\mu) + I$
is elliptic and positive, hence invertible in
$\Psi^*(\phi_2^{-1}(p))$ for each $(p,\mu)$, by the analogue of
Lemma \ref{lem.inverse in calculus} in the standard calculus. By
Lemma \ref{lem.fullyelliptic} it follows that $G_m$ is fully
elliptic. Since it is also positive, it is invertible. By Lemma
\ref{lem.inverse in calculus} the inverse is in $\Psi_\bfa^{-m}(M)$.
\end{proof}

We use these operators to define weighted Sobolev spaces on $M$.
\begin{definition}
\label{def:sobolev}
Let $G_m$ be chosen as in Lemma \ref{lem.exist ell inv} for $m> 0$, $G_0=I$, and let $G_{m}=G_{-m}^{-1}$ for $m<0$.
For $m,\alpha \in \RR$, the weighted Sobolev space $x^\alpha H^m_\bfa (M,\Omega^{1/2})$
is defined by completing $C_0^\infty(\intM, \Omega^{1/2})$ with respect to the norm:
$$
||u||_{x^\alpha H^m_\bfa} = ||G_m(x^{-\alpha} u)||_{L^2}.
$$
\end{definition}
The definition extends in an obvious way to sections of $\Omega^{1/2}\otimes E$ for a hermitian vector bundle $E$ over $M$.
It follows from the next theorem that any other choice of invertible fully elliptic operator $G'_m$
defines the same Sobolev spaces and
endows them with an equivalent norm.
\begin{theorem}
\label{th:bounded sobolev}
Let $M$ be an \bfa-manifold and let $E,F$ be hermitian vector bundles over $M$.
An operator $P \in \Psi_\bfa^m \MOEF$ defines a bounded map
$$
P: x^\alpha H^k_\bfa(M, \Omega^{1/2}\otimes E) \rightarrow x^\alpha H^{k-m}_\bfa(M,  \Omega^{1/2}\otimes F)
$$
for any $m,k, \alpha \in \RR$.
\end{theorem}
\begin{proof}
We prove this only for the case where $E$ and $F$ are trivial line
bundles, as the general case follows by localization from this case.
By Corollary \ref{cor:invunderxconj} it suffices to consider the case that $\alpha=0$.
Further reduce the proof to the case that $k=m=0$ as follows.
Assume that the theorem is true if $k=m=0$.  The operator $P'=G_{k-m}P G_{-k}$
has order zero, so it is bounded on $L^2$.  By definition, $G_l$ is an isomorphism $H_\bfa^l\to L^2$ for all $l$.
Therefore, $P=G_{k-m}^{-1}P'G_k$ is bounded $H^k_\bfa\to L^2\to L^2\to H^{k-m}_\bfa$.

So we need to show boundedness on $L^2$ of zero order operators $P$. This can be reduced to boundedness of operators in $\xPsiinfty$ by adapting a standard argument of H\"ormander.  Using exactness of the
standard symbol sequence, multiplicativity and asymptotic
completeness of the standard symbol, one uses an iterative procedure to find an approximate square root of $CI-P^*P$ for a sufficiently large constant $C$, i.e.\ a selfadjoint operator $Q\in\Psi_\bfa^0 \MOEF$ satisfying
$Q^2= CI -P^*P + R$ for some $R \in \Psi_\bfa^{-\infty} \MOEF$. Using the normal operator in a similar way, $Q$ can even be found so that $R\in \xPsiinfty$; the details can be done in the same way as in \cite{MaMe}. Reordering gives $P^*P=CI-Q^2+R$, hence  $||Pu||^2 \leq C||u||^2 + \langle Ru, u\rangle$
for any $u\in L^2$ since $\|Qu\|^2\geq 0$, so boundedness of $P$ follows from boundedness of $R$.

Finally, operators in $\xPsiinfty$ are Hilbert-Schmidt, hence bounded, since their kernels are square integrable (see Lemma \ref{lem:compactop} below), so we are done.
\end{proof}

We now consider compactness.

\begin{lemma}
\label{lem:compactop}
Let $P\in x^p \Psi_\bfa^k \MO$. If $p>0$ and $k<0$ then $P$ is a compact operator on  $x^\alpha H^m_\bfa (M,\Omega^{1/2})$ for any $\alpha, m$. If $p>\frac12(\gamma_z-1)$, where $\gamma_z$ is defined in \eqref{eq:def gamma}, and $k<- \dim M /2$ then $P$ is a Hilbert-Schmidt operator.
\end{lemma}
\begin{proof}
By a conjugation argument as in the proof of Theorem \ref{th:bounded sobolev} it suffices to prove this on $L^2$. We first prove the second claim.   To do this, we need to show that
the kernel of $P$ is square integrable. Since $k<- \dim M/ 2$ the kernel of $P$ is locally square integrable in the interior of $M^2_z$. Also, by definition it vanishes to infinite order at all faces except $\ff_z$, so it suffices to
check square integrability in a neighborhood of $\ff_z$.
In local coordinates the kernel is
$$K'_P(x,\calT, y, \calY,z, \calZ, w, w') |d\calT\, d\calY \,
d\calZ \, dw'|^{1/2} |x^{-\gamma_z} dxdydzdw|^{1/2},$$
where $K'_P=x^p h$ with $h$ vanishing faster than any negative power as $|(\calT, \calY, \calZ)|\to\infty$ and square integrable in $\calT,\calY,\calZ,w,w'$ near the diagonal $\calT=\calY=\calZ=0$, $w=w'$, uniformly up to $x=0$. Therefore, $|K'_P|^2$ is integrable since the exponents of $x$ add up to $2p-\gamma_x>-1$ and $y,z,w$ and $w'$ vary only over compact sets.

If $P\in x^p \Psi_\bfa^k \MO$ then its adjoint $P^*$ is also in $x^p \Psi_\bfa^k \MO$  and $(P^*P)^N \in x^{2Np}\Psi_\bfa^{2Nk}\MO$ by Corollary \ref{th:smallcomp}, so if $p>0$, $k<0$ then for some $N\in\NN$ the operator $(P^*P)^N$ is Hilbert-Schmidt, hence compact, by the first part of the proof. Since $P^*P$ is selfadjoint the spectral theorem then implies that $|P|=\sqrt{P^*P}$ is compact, and the polar decomposition of $P$ then shows that $P$ is compact.
\end{proof}

\subsection{Proofs of the main theorems}
\label{subsec.proofs}

\noindent{\bf From half-densities to functions.}
In the body of the paper we have chosen to state our results in terms of half-densities since in this way we did not have to choose an auxiliary measure for the $L^2$ and Sobolev spaces. The theorems in the introduction were stated in terms of functions, not half-densities. Therefore, we now translate the results to statements about functions which are in $L^2$ or Sobolev spaces with respect to any fixed density.

As before we identify half-densities with functions by choosing a fixed positive density $\nu$ on the interior of $M$ and identifying $f\,\nu^{1/2}$ with $f$. Thus, define
$$ x^\alpha H^m_\bfa (M,\nu) = \{f\in C^{-\infty}(M):\, f\,\nu^{1/2} \in x^\alpha H^m_\bfa(M,\Omega^{1/2})\},$$
with norm $\|f\|_{x^\alpha H^m_\bfa (M,\nu)} = \|f\,\nu^{1/2}\|_{x^\alpha H^m_\bfa(M,\Omega^{1/2})}$, and similarly for sections of a hermitian vector bundle over $M$.
For the $L^2$ space, i.e.\ $\alpha=m=0$, this is standard notation since it gives $\|f\|_{L^2(M,\nu)}= \left(\int |f|^2\,\nu\right)^{1/2}$. These Sobolev spaces of functions depend on the choice of $\nu$, but multiplying $\nu$ by a smooth positive function will not change them. For example, the $\bfa$-Sobolev spaces associated to any two \bfa-metrics will be equivalent. Multiplying the density by $x^s$ corresponds to a shift in weight:
$$ x^\alpha H^m_\bfa (M,x^s\nu) = x^{\alpha-\frac s 2} H^m_\bfa(M,\nu). $$
We have the following corollary of the Sobolev mapping theorem.
\begin{corollary}
\label{lem.sobolev mapping functions}
Let $\nu$ be a positive density on $\intM$. If $P\in \Psi^m_\bfa(M,E,F)$ then $P_\nu=\nu^{-1/2}P\nu^{1/2}$ defines a bounded map
$$
P_\nu: x^\alpha H^k_\bfa(M,E,\nu) \rightarrow x^\alpha H^{k-m}_\bfa(M,F,\nu)
$$
for any $m,k, \alpha \in \RR$.
\end{corollary}
Here $\nu^{\pm 1/2}$ denotes the multiplication operator.
\begin{proof}
This follows from  Theorem \ref{th:bounded sobolev} by considering the sequence of bounded maps
$$x^\alpha H^k_\bfa(M,E,\nu)\stackrel{\nu^{1/2}}\longrightarrow x^\alpha H^k_\bfa(M,\Omega^{1/2}\otimes E)
\stackrel{P}\longrightarrow x^\alpha H^{k-m}_\bfa(M,\Omega^{1/2}\otimes F)
\stackrel{\nu^{-1/2}}\longrightarrow x^\alpha H^{k-m}_\bfa(M,F,\nu)$$
\end{proof}
Note that, if $\nu=x^s \nu_0$ for $s\in\RR$ and $\nu_0$ a smooth positive density on $M$, then for an \bfa-differential operator $P$ the corresponding $P_\nu$ is of the form \eqref{eq:difflocal}. This follows from $x^{-s}(x^{1+a_1+a_2}\partial_x)x^s = x^{1+a_1+a_2}\partial_x + x^{a_1+a_2}s$ . Also, by arguments similar to those in the case of closed manifolds, it is straightforward to show that $f\in L^2(M,\nu)$ is in
$ H^m_\bfa (M,\nu)$ if and only if $f\in H^m_\loc (\interior{M})$ and in any adapted local coordinates near the boundary,
$P_\nu f \in L^2(M,\nu)$ for all $P_\nu$ as in \eqref{eq:difflocal}.

For the proof of Theorem \ref{th:mainthm} we need to relate tempered distribution and weighted Sobolev spaces. For the following lemma see also \cite{Me-mwc}, Section 3.5.
\begin{definition}[and Lemma]
\label{def:Schwartz space}
Schwartz space on a compact manifold with boundary $M$ is defined as
$$ \Cdot^\infty (M) := \{f\in C^\infty(M):\,f \text{ vanishes to infinite order at }\partial M\} $$
where $\nu$ is a smooth positive density. This is the intersection of all weighted b-Sobolev spaces:
\begin{equation}
\label{eq:Schwartz Sobolev}
\Cdot^\infty (M) = \bigcap_{\alpha\in\RR, m\in\RR} x^\alpha H_b^m (M,\nu).
\end{equation}
The Frechet topology on $\Cdot^\infty (M)$ is defined using the norms of these Sobolev spaces.

The space of {\em tempered distributions}  is defined as the dual space
$$ \Cdot^{-\infty} (M) = (\Cdot^\infty(M))^* .$$
\end{definition}
\begin{proof}
We only need to check \eqref{eq:Schwartz Sobolev}. Recall that, for $m\in \NN_0$, by definition $u\in H^m_b(M,\nu)$ iff $u$ is in the standard $H^m(M,\nu)$ Sobolev space away from the boundary and, in any local coordinate patch $U$ near the boundary, $ (x\partial_x)^k \partial_y^\gamma u\in L^2(U)$ whenever $k+|\gamma|\leq m$. Therefore the Sobolev Embedding Theorem gives $\|x^m u\|_{C^l(M)} \leq C \|u\|_{H^m_b(M,\nu)}$ whenever $m>l+\tfrac12 \dim M$, from which the claim follows easily.
\end{proof}

\begin{lemma}
\label{lem:Schwartz space}
The intersection of all weighted \bfa-Sobolev spaces is the Schwartz space on $M$:
$$ \Cdot^\infty(M) = \bigcap_{\alpha, m\in\RR} x^\alpha H^m_\bfa (M,\nu), $$
and the topologies are equivalent.
The union of all weighted \bfa-Sobolev spaces is the space of tempered distributions:
$$ \Cdot^{-\infty}(M) = \bigcup_{\alpha, m\in\RR} x^\alpha H^m_\bfa (M,\nu). $$
\end{lemma}
\begin{proof}
It suffices to show that for each $\beta\in\RR,k\in\ZZ$ there are $\alpha,\alpha'\in\RR$, $m,m'\in\ZZ$ so that
$$ x^\alpha H^m_\bfa (M,\nu) \subset x^\beta H^k_b (M,\nu) \subset x^{\alpha'} H_\bfa^{m'}(M,\nu)$$
with continuous inclusions. For non-negative $k$ this follows from the obvious inclusions
$$ x^\alpha {}^\bfa\Diff^m (M) \subset x^\alpha {}^b \Diff^m (M) \subset x^{\alpha-m(a_1+a_2)} {}^\bfa\Diff^m (M)$$
and for negative $k$ using the duality of $H^k_\bfa(M,\nu)$ with $H^{-k}_\bfa(M,\nu)$ and similar for b-Sobolev spaces.
\end{proof}

\begin{proof}[Proof of Theorems \ref{th:psd-calc} and \ref{th:mainthm}]
Fix a positive regular density $\nu$ on $M$ and identify $P\in\Psi^m_\bfa(M)$ acting on half densities with $P_\nu=\nu^{-1/2}P\nu^{1/2}$ acting on functions. The general \bfa-differential operator \eqref{eq:difflocal} is then really $P_\nu$, for $P$ having the kernel \eqref{eq:diff kernel}, which is clearly in $\Psi^m_\bfa(M)$. Together with Theorem \ref{th:param} this proves Theorem \ref{th:psd-calc}.

In our new notation, the operator $P$ in Theorem \ref{th:mainthm} should be replaced by $P_\nu$.
For the metric $g=x^{2r}g_{\bfa b}$ we have $\dvol_g = x^{nr}\dvol_{g_{\bfa b}}$ where $n=\dim M$. Also, $\dvol_{g_{\bfa b}}$ is a positive $\bfa$-density, so it equals $x^{-\gamma_z}$ times a positive regular density, with $\gamma_z$ defined in \eqref{eq:def gamma}. We use this regular density as $\nu$ in the identification above. Then we have
$\dvol_g = x^s \nu$ with $s=nr-\gamma_z$ and so $x^\alpha H^k_\bfa(M,E,\dvol_g)=x^{\alpha-s/2}H^k_\bfa(M,E,\nu)$.
The theorem therefore follows from Corollary \ref{lem.sobolev mapping functions} and the corresponding statements for half-densities which we now prove.

By Theorem \ref{th:param}, $P$ has a parametrix, so $PQ=I+R$,
$QP=I+R'$ with $R,R'\in\xPsiinfty$. By Lemma \ref{lem:compactop} $R$
and $R'$ are compact operators on
$x^\alpha L^2(M,\Omega^{1/2})$ for any $\alpha$, so $P$ is Fredholm
on this space. If $Pu=f\in x^\alpha H^k_{\bfa}(M,\Omega^{1/2})$ and
$u\in x^{{\alpha'}}H^{k'}_{\bfa}(M,\Omega^{1/2})$ then applying $Q$ gives
$u+R'u = QPu=Qf$. Now  $Qf \in x^\alpha H^{(k+m)}_{\bfa}(M,\Omega^{1/2})$ by Theorem \ref{th:bounded
sobolev}, and $R'u$ is in the same space by the same theorem, so $u$
must also be.
\end{proof}

\begin{proof}[Proof of Theorem \ref{th:resolv}]
Write $\Delta=\Delta_{{\bfa b}}$.
Since $g_{\bfa b}$ is a complete metric, $\Delta$ is essentially selfadjoint on compactly supported smooth functions, and since it is non-negative, its spectrum is contained in $[0,\infty)$. (In the sequel to this paper we will see that the spectrum is equal to $[0,\infty)$.) Furthermore, the principal symbol of $\Delta$ is $\sigma(p,\xi) = |\xi|^2_{g^*(p)}$ times the identity, which is invertible
by definition of an $\bfa$-metric. Therefore, the Laplacian is \bfa-elliptic.

Therefore, the theorem follows from the following claim: If $P\in\Psi^m_\bfa(M)$, $m\geq0$, is elliptic and non-negative then $(P-\lambda)^{-1}\in\Psi^{-m}_\bfa(M)$ for all $\lambda\in\CC\setminus[0,\infty)$. This claim is proved as follows.
First, $P\geq0$ implies $\Nhat_\bfa(P)(p,\mu)\geq 0$ for each $p,\mu$.
Therefore, for each $\lambda\in\CC\setminus[0,\infty)$ the operator $\Nhat_\bfa(P-\lambda)(p,\mu)=\Nhat_\bfa(P)(p,\mu)-\lambda\in\Psi^m(F_{2,p})$ is invertible as operator in $L^2(F_{2,p})$, hence by Lemma \ref{lem.fullyelliptic} and the remark following it $P-\lambda$ is fully elliptic, hence its inverse is in $\Psi_\bfa^{-m}(M)$ by Lemma \ref{lem.inverse in calculus}.
\end{proof}

In a future paper, we will consider the Laplacian in more detail, and show that parametrices
can be constructed for it with respect to various weighted Sobolev spaces.



\section{Appendix:  Multi-fibred boundary structures}
In this appendix we show how the geometric setup of the main text may be generalized to more than two fibrations. The treatment here is more systematic, and also serves to fill in the details in parts of Section \ref{section:definitions}.

We introduce the structure of several nested fibrations at the boundary, extending to various orders to the interior of a manifold with boundary. This is to encode the idea of vector fields being tangent to the fibres of those fibrations at the boundary, such that the tangency to larger fibres is to higher order than to the tangency to smaller fibres.

\begin{definition}
A {\em tower of fibrations of manifolds with boundary} is a sequence of fibrations
${\bf \Phi} = (\Phi_0,\dots,\Phi_k)$,
\begin{equation}
\label{eq:tower of fibrations}
\Btilde_k \stackrel{\Phi_k} \to \Btilde_{k-1} \stackrel{\Phi_{k-1}} \to
\dots \stackrel{\Phi_1} \to \Btilde_0 \stackrel{\Phi_0} \to \Btilde_{-1} = [0,\eps)
\end{equation}
where each $\Btilde_i$ is a manifold with boundary.
\end{definition}
For simplicity we assume the $\Btilde_i$ to be connected. Then the fibre $\Phi_i^{-1} (p)$ is (up to diffeomorphism) independent of $p\in \Btilde_{i-1}$ and will be denoted $F_i$. For $i \geq j$ denote the compositions $\Phi_{i,j}=\Phi_j\circ \dots \circ \Phi_i: \Btilde_i \to \Btilde_{j-1}$, with fibres $F_{i,j}\subset \Btilde_i$. Then we get a sequence of nested foliations of $\Btilde_k$, given by the fibres of the maps $\Phi_k=\Phi_{k,k},\Phi_{k,k-1},\dots,\Phi_{k,0}$, the leaves being $F_k=F_{k,k},F_{k,k-1},\dots,F_{k,0}$ of successively increasing dimension. Restriction of the map $\Phi_k$ makes each $F_{k,i}$ a fibre bundle over $F_{k-1,i}$ for $i<k$.

Note that the map $\Phi_{k,0}:\Btilde_k \to [0,\eps)$ is a boundary
defining function, but it is natural to think of it also as a
fibration.

A tower of fibrations induces a tower of fibrations at the boundary, $B_i=\partial \Btilde_i$,
$$ B_k \stackrel{\phi_k} \to B_{k-1} \stackrel{\phi_{k-1}} \to
\dots \stackrel{\phi_1} \to B_0 (\to \{0\})
$$
with the same fibres. Such towers of fibrations (over closed manifolds) were also considered by Bismut and Cheeger in \cite{BisChe} under the name of multifibrations.

For any fibration $\phi:M\to B$, tangency of a vector field $V$ on
$M$ to the fibres of $\phi$ is equivalent to $Vf \equiv 0$ for all
functions $u$ constant on the fibres, i.e. $u\in C_\phi^\infty(M) :=
\phi^*C^\infty(B)$. Therefore, if $M,B$ are manifolds with boundary
then {\em tangency of $V$ to the $\phi$-fibres to order $a\in\NN$ at
the boundary} may be reasonably defined as $Vu=O(x^a)$ for all $u\in
C_\phi^\infty(M)$. Here and in the sequel we write $O(x^a)$ for a
section of a vector bundle that is  $x^a$ times a smooth section
(smooth up to the boundary), for some bdf $x$. This is independent
of the choice of  $x$.

Alternatively,  tangency of $V$ to the fibres of $\phi$ is equivalent to $\phi_*V=0$, where $\phi_*V$ is the section of $\phi^* TB$ defined by $(\phi_*V)(p) = d\phi_{|p}(V_{|p}) \in T_{\phi(p)} B$ for $p\in M$. Then tangency to order $a$ is equivalent to $\phi_*V=O(x^a)$.
Therefore we define:

\begin{definition}
Let $\bfPhi$ be a tower of fibrations as in \eqref{eq:tower of fibrations} and let $\bfa = (a_0,\dots,a_k) \in \NN^{k+1}$. Define the space of vector fields tangent to $\bfPhi$ to order $\bfa$ by
\begin{align}
\label{eq:def a phi V}
^{\bfa,\bfPhi}\calV &= \{
V\in \Gamma(T\Btilde_k):
\begin{aligned}[t]
 Vu &= O(x^{a_k}) && \text{ for all } u \in C_{\Phi_k}^\infty (\Btilde_{k})\\
 Vu &= O(x^{a_{k-1}+a_k}) && \text{ for all } u \in C_{\Phi_{k,k-1}}^\infty(\Btilde_{k}) \\
 &\vdots && \\
 Vu &= O(x^{a_0 + \dots + a_k}) && \text{ for all } u\in C_{\Phi_{k,0}}^\infty (\Btilde_{k})
     \}
\end{aligned} \\
\label{eq:def a phi V 2}
&= \{V\in \Gamma(T\Btilde_k):
(\Phi_{k,i})_*V = O(x^{a_i+\dots+a_k})\ \text{ for }i=0,\dots k \}
\end{align}
\end{definition}
The setup in the main text corresponds to $k=2$, $a_0=1$, and there we write simply $\bfa=(a_1,a_2)$.
Note that the last condition in \eqref{eq:def a phi V} implies that all vector fields are tangent to the boundary.

Examples of such structures appear previously in the b-calculus literature.  For instance, a tower of fibrations of height $k=0$ is just a boundary defining function $x:U \to [0,\eps)$ on $U=\Btilde_0$. If $a_0=1$ then $^{\bfa,x}\calV = {}^{\rm b}\calV(U)$ is the set of vector fields tangent to the boundary, or b-vector fields
as defined in \cite{Me-aps}. If $a_0=2$ then $^{\bfa,x}\calV = {}^c\calV(U)$ is the set of cusp vector fields, tangent to second order to the boundary, as defined in \cite{MaMe}. The set $^{\rm b}\calV(U)$ is independent of the choice of $x$ but $^{\rm c}\calV(U)$ depends on the choice of $x$ up to constant multiple and to second order at $\partial U$, i.e. $x,x'$ define the same structure if and only if $x' = cx+ O(x^2)$ for some constant $c>0$.
A tower of fibrations of height $k=1$ is given by a single fibration and a boundary defining function, and in the case $a_0=a_1=1$ is called a manifold with fibred boundary in \cite{MaMe}.

{\bf Local coordinate representations:} Let $\bfPhi$ be a tower of fibrations. We say coordinates $x, y^{(0)},y^{(1)},\dots, y^{(k)}$ on $\Btilde_k$ are {\em  compatible with $\bfPhi$} if $x$ is the pull-back via $\Phi_{k,0}$ of a coordinate on $\Btilde_{-1} = [0,\eps)$, the $y^{(0)}$ are pulled back via $\Phi_{k,1}$ from $\Btilde_0$ and so on. Each $y^{(i)}$ is a $(\dim F_i)$-tuple and may be thought of as coordinates on $F_i$, for each $i$ (but more precisely $y^{(i)}$ are just pull-backs of functions on $\Btilde_i$ which, complemented by $x$ and the pull-backs of $y^{(0)},\dots,y^{(i-1)}$ to $\Btilde_i$, are local coordinates on
$\Btilde_i$).

Choose compatible coordinates on a coordinate neighborhood $U\subset\Btilde_k$.
Write $\partial_{y^{(i)}}$ for $\partial_{y^{(i)}_1},\partial_{y^{(i)}_2},\dots$. Then $(\Phi_{k,i})_*\partial_{y^{(j)}} = \partial_{y^{(j)}}$ for $j<i$ (with $y^{(-1)}:=x$). Therefore, a vector field $\beta_{-1} \partial_x + \sum_{i=0}^k \beta_i \partial_{y^{(i)}}$ satisfies the conditions in \eqref{eq:def a phi V 2} if and only if $\beta_{i-1}=O(x^{a_i+\dots+a_k})$ for $i=0,\dots,k$, and this shows
\begin{equation}
\label{eq:a calV in coords}
^{\bfa,\bfPhi}\calV_{|U} = \Span_{C^\infty(\Btilde_k)_{|U}}
      \{ x^{a_0+\dots+a_k}\partial_x, \
         x^{a_1+\dots+a_k}\partial_{y^{(0)}},
         \dots,
         x^{a_k}\partial_{y^{(k-1)}},\
         \partial_{y^{(k)}} \}
\end{equation}

We now define the main object of interest.
\begin{definition}
Let $M$ be a manifold with boundary, $k\in\NN_0$ and $\bfa\in\NN^{k+1}$. A {\em multi-fibred boundary structure to order $\bfa$} (or \bfa-structure) on $M$  is a space of vector fields $^\bfa\calV(M) \subset \Gamma(TM)$ for which there is a tower of fibrations $\bfPhi$ with $\Btilde_k$ a neighborhood of the boundary of $M$ and $^\bfa\calV(M)_{|\Btilde_k} = {}^{\bfa,\bfPhi}\calV$. In this case we say that $^\bfa\calV(M)$ is {\em represented by} $\bfPhi$.
\end{definition}
Alternatively, an \bfa-structure may be characterized by the dual space of one-forms
$$ ^\bfa\Omega(M) = \{ \omega\in \Omega^1(\interior{M}):\, \omega(V) \text{ extends smoothly from }\interior{M}\text{ to } M,\text{ for all }V\in {}^\bfa\calV(M)\}
$$
or in terms of the algebra of functions
\begin{equation}
\label{eq:def a calF}^\bfa\calF(M) = \{ u\in C^\infty(M): Vu = O(x^{a_0+\dots+a_k})\text{ for all }V\in {}^\bfa\calV(M)\}.
\end{equation}
The algebra $^\bfa\calF(M)$ determines $^\bfa\calV(M)$ via
$$ ^\bfa\calV(M) = \{V\in \Gamma(TM):\, Vu = O(x^{a_0+\dots+a_k})\text{ for all }u\in {}^\bfa\calF(M) \}.$$

Let $^\bfa\calV(M)$ be represented by $\bfPhi$, and let $x, y^{(0)},y^{(1)},\dots, y^{(k)}$ be coordinates compatible with $\bfPhi$, defined on a neighborhood $U$ of a boundary point of $M$. Then it is clear from \eqref{eq:a calV in coords} that
\begin{align}
\label{eq:a Omega in coords}
^\bfa\Omega(M)_{|U} &= \Span_{C^\infty(M)_{|U}}
     \{ \frac{dx}{x^{a_0+\dots+a_k}},\
        \frac{dy^{(0)}} {x^{a_1+\dots+a_k}},
        \dots,
        \frac{dy^{(k-1)}} {x^{a_k}},\
        dy^{(k)}\}
        \\[1mm]
\label{eq:a calF in coords}
^\bfa\calF(M)_{|U} &= \begin{gathered}[t]\{ u_{-1}(x) + x^{a_0} u_0(x,y^{(0)}) + x^{a_0+a_1}u_1(x,y^{(0)},y^{(1)}) + \dots \\
+ x^{a_0+\dots+a_k}u_k(x,y^{(0)},\dots,y^{(k)}):\ u_{-1},\dots,u_k\text{ smooth }\}.
\end{gathered}
\end{align}
(To see \eqref{eq:a calF in coords}, write an arbitrary function in terms of its Taylor expansion in $x$ and check using \eqref{eq:def a calF} which variables each coefficient may depend on).
That is, a  function is in $^\bfa\calF(M)$ if and only if it is constant on the leaves of $\Phi_{k,i}$ up to an error $O(x^{a_0+\dots+a_i})$ for each $i=0,\dots,k$:
$$ ^\bfa\calF(M) = C_{\Phi_{k,0}}^\infty(M) + x^{a_0} C_{\Phi_{k,1}}^\infty(M) + \dots + x^{a_0+\dots+a_k} C^\infty(M),$$
where $x$ is the bdf $\Phi_{k,0}$, or more generally any bdf in $^\bfa\calF(M)$.

The basic geometric structure of interest to us is any one of the spaces $^\bfa\calV(M)$, $^\bfa\Omega(M)$ or $^\bfa\calF(M)$, but for
practical purposes it is often convenient to work with a $\bfPhi$ representing it. Now we need
to understand when two towers of fibrations $\bfPhi$, $\bfPhi'$ are {\em equivalent to order $\bfa$} in the sense that they define the same structure $^\bfa\calV(M)$. We first do this in coordinates. Let $x,y^{(0)},\dots,y^{(k)}$ be compatible coordinates for $\bfPhi$ and $x',y^{(0)'},\dots,y^{(k)'}$ for $\bfPhi'$. Since these are two coordinate systems on a neighborhood of the boundary in $M$ we may regard $x'$ and each $y^{(i)'}$ as functions of $x,y^{(0)},\dots,y^{(k)}$. For the structures to be equivalent, the spaces of 1-forms in \eqref{eq:a Omega in coords} must be the same for the two sets of coordinates. In particular, $ (x')^{-a_0-\dots-a_k} dx',\
         {(x')^{-a_1-\dots-a_k}}dy^{(0)'},
        \dots,
         {(x')^{-a_k}} dy^{(k-1)'},\
        dy^{(k)'}$
must lie in 
\eqref{eq:a Omega in coords}. Here one can replace $x'$ in the denominators by $x$ since the quotient of any two boundary defining functions is smooth.
One obtains, with smooth functions $u_i^j$,
\begin{align}
y^{(k)'}   &= u_k^k (x,y^{(0)},\dots,y^{(k)})&& && \notag \\
y^{(k-1)'} &= u_{k-1}^{k-1} (x,y^{(0)},\dots,y^{(k-1)})&& + O(x^{a_k}) && \notag\\
y^{(k-2)'} &= u_{k-2}^{k-2} (x,y^{(0)},\dots,y^{(k-2)})&& + x^{a_{k-1}}u_{k-1}^{k-2}(x,y^{(0)},\dots,y^{(k-1)})&&+ O(x^{a_{k-1}+a_k}) \notag \\
\vdots & &&&&
\label{eq:coord transform}\\
y^{(0)'} &= u_0^0 (x,y^{(0)}) &&+ x^{a_1} u_1^0(x,y^{(0)},y^{(1)}) &&+\dots + O(x^{a_1+\dots+a_k})\notag \\
x' &= u_{-1}^{-1} (x) &&+ x^{a_0} u_0^{-1}(x,y^{(0)}) &&+ \dots
+ O(x^{a_0+\dots+a_k}) \notag
\end{align}
Here $u_i^j$ is $\RR^{\dim F_j}$ valued.
It is not hard to see that these equations are also sufficient for equivalence of $\bfPhi,\bfPhi'$.
The terms $u_j^j$ here are the local coordinate form of a diffeomorphism between $\bfPhi$ and $\bfPhi'$, i.e.\ a set of diffeomorphisms $A_i:\Btilde_i\to\Btilde_i'$ commuting with the $\Phi_i, \Phi_i'$. In particular,
by the conditions on boundary defining functions, $u_{-1}^{-1}(x)=xv(x)$ for some smooth function $v$. Clearly, if $\bfPhi$, $\bfPhi'$ are diffeomorphic then they  are equivalent to any order since they define the same fibres. The higher terms describe what deviation from diffeomorphism is allowed to yield the same multi-fibration structure to order $\bfa$, that is, to what orders the fibres of $\Phi_i,\Phi_i'$ may differ (up to the coordinate changes $A_i$). The last equation simply says $x'\in {}^{\bfa}\calF(M)$.
In particular, for any $k$, the tower of fibrations at the boundary is determined by $^\bfa\calV(M)$, up to diffeomorphism.

In the examples above, we see that for $k=0$ and $a_0=2$, the function $u_{-1}^{-1}$ is just $cx$. Also, for $a_0=1$ and any $k$, the first term in the $x'$ formula can be absorbed into the second, and for $k=2$, $a_0=a_1=1$ this gives precisely the condition $x'\in x C_{\phi}^\infty(M)$ from \cite{MaMe}.

The following normal form shows that any multi-fibred boundary structure may be expressed in the apparently more special form used in the main part of the paper, Definition \ref{def:interior dfs}.
\begin{lemma} \label{lem:product}
Any multi-fibred boundary structure may be represented by a tower of fibrations with $\Btilde_i=B_i\times [0,1)$ and $\Phi_i=\phi_i\times\id$ for all $i$.
\end{lemma}
\begin{proof}
Given any representing tower of fibrations, first trivialize $\Phi_0:\Btilde_0\to [0,1)$. This can be done since $[0,1)$ is simply connected. This puts $\Btilde_0$ and $\Phi_0$ into product structure. Next, choose any trivialization $\Btilde_1\cong B_1\times[0,1)$ of $\Phi_0\circ\Phi_1$. Then $\Phi_1:B_1\times[0,1)\to B_0\times [0,1)$ is a family, parametrized by $x\in [0,1)$, of fibrations $B_1\to B_0$. These have to be equivalent via a smooth family of diffeomorphisms of $B_1$. Composing with these diffeomorphisms, we obtain product structure for $\Phi_1$ as well. Continue like this up to $\Btilde_k,\Phi_k$.
\end{proof}

\medskip {\bf Equivalence and fibre diagonals:} From now on we consider only the case $a_0=1$ for simplicity.
For a fibration $\phi:M\to B$ let $\Delta_\phi=\{(p,p'):\, \phi(p)=\phi(p')\}\subset M^2$ be the fibre diagonal. Two fibrations $\phi:M\to B$, $\phi':M\to B'$ have the same fibre diagonal if and only if they are diffeomorphic in the sense described above, i.e. if and only if they have the same fibres. Therefore, it is not surprising that equivalence of towers of fibrations, to some order $\bfa$ at the boundary, may be reformulated in terms of agreement to finite order of their fibre diagonals.
In Definition \ref{def:submfd finite order} we have introduced the notion of `agreement to finite order' only for p-submanifolds. However, (fibre) diagonals in $M^2$, for $M$ a manifold with boundary, are not p-submanifolds. Therefore, we first blow up the corner $\dM\times\dM$; then we are in the setting of p-submanifolds. \ownremark{Certainly one can extend the whole stuff about agreement to some order to submanifolds that are not p. And at some point it may be wise to do so. But not now.}

Given a tower of fibrations $M=\Btilde_k \stackrel{\Phi_k} \to \Btilde_{k-1} \stackrel{\Phi_{k-1}} \to
\dots \stackrel{\Phi_1} \to \Btilde_0 \stackrel{\Phi_0} \to \Btilde_{-1} = [0,\eps)$, with associated fibrations $\Phi_{k,i}:M \to \Btilde_{i-1}$, we have
$$M^2\supset \Delta_{\Phi_{k,0}}\supset \Delta_{\Phi_{k,1}}\supset\dots\supset \Delta_{\Phi_{k,k}} \supset \Delta_M.$$
Let $\beta_b:M^2_x=[M^2,(\dM)^2]\to M^2$ be the blow-down map for the $b$-blowup and let $\Deltatilde_i$ be the lift of $\Delta_{\Phi_{k,i}}$ under $\beta_b$. Then
\begin{equation}
\label{eq:chain of diagonals}
(M^2_x\supset)\, \Deltatilde_1\supset\Deltatilde_2\supset \dots \supset \Deltatilde_k
\end{equation}
is a chain of p-submanifolds, i.e. near any point there is one coordinate system in which all $\Deltatilde_i$ are given by the vanishing of some of the coordinates.
The submanifold $\Deltatilde_0$ is irrelevant in the sequel because $a_0=1$.

The p-submanifold $\Deltatilde_i$ is an interior extension of $\Delta_i := \Deltatilde_i\cap \ff_x$ which under $\beta_b$ is identified with the fibre diagonal $\Delta_{\phi_{k,i}}\subset \partial M\times \partial M$. The equivalence introduced in Definitions \ref{def:submfd finite order} and \ref{def:ideal order a} has a natural extension to such chains of submanifolds: Two interior extensions of the chain $\Delta_1\supset\dots\supset\Delta_k$ agree to order $(a_1,\dots,a_k)$ if and only if they define the same ideal
$$
\calI (\Deltatilde_1) + \calI(\ff_x)^{a_1}\calI(\Deltatilde_2) + \calI(\ff_x)^{a_1+a_2}\calI(\Deltatilde_3)+\dots +\calI(\ff_x)^{a_1+\dots+a_{k-1}} \calI(\Deltatilde_k) + \calI(\ff_x)^{a_1+\dots+a_k}.
$$
We then have:

\begin{theorem} \label{thm:equiv towers}
Let $\bfa=(1,a_1,\dots,a_k)$.
Two towers of fibrations on $M$ are equivalent to order $\bfa$ if and only if the associated chains \eqref{eq:chain of diagonals} of p-submanifolds are equivalent to order $(a_1,\dots,a_k)$ in the sense introduced above.
\end{theorem}
This is proved in \cite{GH2}. There we also show that a chain of p-submanifolds to order $(a_1,\dots,a_k)$ can be canonically resolved by first blowing up $\Delta_1$ to order $a_1$, then $\Delta_2$ to order $a_2$ etc. More precisely, when blowing up $\Delta_1$ we lift $\Delta_2,\dots,\Delta_k$ by taking the intersection of their extensions $\Deltatilde_2,\dots,\Deltatilde_k$ with the front face, and similarly for subsequent lifts. In particular, the lift of $\Delta_i$ has the same dimension as $\Delta_i$.\footnote{This is different from blowing up a chain of submanifolds (without increasing orders of definedness), as in \cite{Me-mwc}, where the lift say of $\Delta_2\subset\Delta_1$ is its preimage, hence has larger dimension than $\Delta_2$. } It is then easy to check that the centers of subsequent blow-ups are still well-defined to these orders after the previous blow-ups. We obtain:

\begin{corollary}
The double space $M^2_\bfa$ for a multi-fibred boundary structure to order $\bfa=(1,a_1,\dots,a_k)$ is canonically defined as the resolution of the chain $\Delta_1\supset\dots\supset\Delta_k$ to order $(a_1,\dots,a_k)$.
\end{corollary}
Recall that in the main text we use the notation $\bfa=(a_1,a_2)$ instead of $\bfa=(1,a_1,a_2)$.
\medskip

{\bf Reducing a multi-fibred boundary structure:}
Given a multi-fibred boundary structure (mfbs) $^\bfa\calV(M)$, we may reduce it by `forgetting' the smallest fibres. This may be iterated. Thus, let $l=0,\dots,k$ and choose a tower of fibrations
\eqref{eq:tower of fibrations} representing $\bfPhi$. Then the reduced mfbs $\Red_l ({}^\bfa\calV(M))$ is represented by the tower of fibrations
\begin{equation}
\label{eq:def reduced mfbs}
\Btilde_k \stackrel{\Phi_{k,l}}\to \Btilde_{l-1} \stackrel{\Phi_{l-1}}\to\dots \stackrel{\Phi_0}\to [0,1)
\end{equation}
with the orders $(a_0,\dots,a_l)$. This is easily seen to be independent of the choice of $\bfPhi$. In compatible local coordinates near a boundary point, $\Red_l({}^\bfa\calV(M))$ is spanned by
$$ x^{a_0+\dots+a_l}\partial_x,\ x^{a_1+\dots+a_l}\partial_{y^{(0)}},\dots,
x^{a_l} \partial_{y^{(l-1)}},\ \partial_{y^{(l)}},\dots,\partial_{y^{(k)}}.$$
In particular, for $l=0$ and $a_0=1$ one gets $^{\rm b}\calV(M)$.
\medskip

{\bf Vector bundles associated  with a multi-fibred boundary structure:}
\eqref{eq:a calV in coords} shows that $^\bfa\calV(M)$ is locally free and hence is the space of sections of a vector bundle, the {\em $\bfa$-tangent bundle $^\bfa TM$}. Thus, a local basis of $^\bfa TM$, at a boundary point, consists of the expressions in \eqref{eq:a calV in coords}. Interpreting $v\in {}^\bfa T_pM$ as a `usual' tangent vector gives a natural bundle map $\iota_\bfa: {}^\bfa TM \to TM$ which over the interior of $M$ is an isomorphism. Define the {\em $\bfa-$normal bundle }
\begin{equation}
\label{eq:def a-normal bundle}
^\bfa N \partial M := \ker \iota_{\bfa|\partial M}.
\end{equation}
It is spanned by the same vectors as $^\bfa TM$ except $\partial_{y^{(k)}}$.

$^\bfa N\partial M$ is actually the pull-back of a bundle from the base $B_{k-1}$. To see this, let $\bfPhi: \Btilde_k \stackrel{\Phi_k} \to \Btilde_{k-1} \stackrel{\Phi_{k-1}} \to
\dots \stackrel{\Phi_0} \to \Btilde_{-1} = [0,\eps)$ be a tower of fibrations  representing $^\bfa\calV(M)$, and let $\bfPhi'$ be the tower
$$\bfPhi': \Btilde_{k-1} \stackrel{\id} \to \Btilde_{k-1} \stackrel{\Phi_{k-1}} \to
\dots \stackrel{\Phi_0} \to \Btilde_{-1} = [0,\eps).$$
Thus, the total space of $\bfPhi'$ is $\Btilde_{k-1}$, and the smallest fibres are points.
Then canonically
\begin{equation}
\label{eq:normal bundle as pullback}
^\bfa N\partial M \cong (\phi_k)^* ({}^\bfa N'B_{k-1}) \quad\text{ where } {}^\bfa N'B_{k-1} :=
{}^{\bfa, \bfPhi'} T\Btilde_{k-1|B_{k-1}}.
\end{equation}
This is clear since in local coordinates, both sides are spanned by
$x^{a_0+\dots+a_k}\partial_x,\dots, x^{a_k}\partial_{y^{(k-1)}}$. More invariantly, any $V'\in {}^{\bfa,\bfPhi'}\calV$ lifts to a section of $V\in {}^{\bfa,\bfPhi}\calV$, and the lift is unique when one imposes the condition $\iota_\bfa V_{|B_k}=0$, so this defines the desired isomorphism.

Denote by $^{(a_0,\dots,a_l)} N\partial M$ the a-normal bundle of the reduction $\Red_l {}^\bfa \calV(M)$. It is spanned by
$$ x^{a_0+\dots+a_l}\partial_x,\dots,x^{a_l}\partial_{y^{(l-1)}}$$
and is the pull-back of a bundle on $B_{k-l-1}$.
In particular, for $l=0$ the basis is $x^{a_0}\partial_x$.

The $\bfa$-normal bundle can be used to give a description of the top order front face $\ff_\bfa$ of $M^2_\bfa$ (and of the previous front faces by cutting off the chain at $\Btilde_l$ for $l<k$), which lies at the core of the analysis of the normal operator in the $\bfa$-pseudodifferential calculus. Since $\ff_\bfa$ arises from the blow-up of a center canonically isomorphic with $\Delta_{\phi_k} = \partial M \times_{\phi_k} \partial M$, it is naturally fibred (via the blow-down map) over this space.

For $\bfa=(a_0,\dots,a_{k-1},a_k)$ let $\bfa+1 = (a_0,\dots,a_{k-1},a_k+1)$. We call a choice of mfbs ${}^{\bfa+1}\calV(M)$ a {\em one step refinement} of a mfbs ${}^\bfa\calV(M)$ if ${}^{\bfa+1}\calV(M)\subset{}^\bfa\calV(M)$.
A refinement allows us to consider ${}^\bfa N^*\partial M$ naturally as a subspace of ${}^\bfa T^*_{\partial M}M$ (rather than a quotient, which it always is), by defining its sections to be  $(x\alpha)_{|\partial M}$ for $\alpha$ a section of ${}^{\bfa+1}T^*M$, where $x$ is a compatible bdf. In local coordinates, ${}^\bfa T^*_{\partial M}M$ is spanned by the expressions on the right in \eqref{eq:a Omega in coords} and ${}^\bfa N^*\partial M$ by the same except $dy^{(k)}$.

Recall that in the main text the fibre diagonal $\Delta_{\Phi_k}$ is called $\Delta_z$, and its intersection with the front face is identified with the boundary fibre diagonal $\Delta_{\phi_k}$.
\begin{proposition} \label{prop:a front face general}
For a multi-fibred boundary structure ${}^\bfa\calV(M)$ there is a diffeomorphism of the interior of the front face of $M^2_\bfa$
\begin{equation}
\label{eq:ffa diffeo}
\interior{\ff}_\bfa \cong \ {}^\bfa N\partial M\times_{\phi_k}\partial M,
\end{equation}
as bundles over $\partial M \times_{\phi_k}\partial M$. The diffeomorphism is determined by, and determines, a one step refinement of ${}^\bfa\calV(M)$.

Suppose such an extension, hence identification \eqref{eq:ffa diffeo} is chosen.
If $g\in C^\infty(\interior{M})$ is such that $dg$ extends to a section of ${}^\bfa N^*\partial M$ then
\begin{equation}
\label{eq:dg meaning}
\beta_\bfa^*(\pi_l^*g - \pi_r^*g)\text{  extends  smoothly to }\intff_\bfa \text{ and equals }  dg \text{ there}
\end{equation}
where $\beta_\bfa:M^2_\bfa\to M^2$ is the blow-down map and $\pi_l,\pi_r:M^2\to M$ are projections to the left and right factor, and $dg$ is considered as a function on ${}^\bfa N\partial M\times_{\phi_k}\partial M$ constant in the second factor. Also, $\beta_\bfa^*(\pi_l^*g - \pi_r^*g)$ has finite order singularities at the other faces of $M^2_\bfa$.
\end{proposition}

\begin{proof}
We first give an invariant definition of the map ${}^\bfa N\partial M\times_{\phi_k}\partial M \to \intff_\bfa$ and then write it in local coordinates. Let $p,p'\in\partial M$ with $\phi_k(p)=\phi_k(p')$, so $(p,p')\in \partial M\times_{\phi_k}\partial M$. The fibre over $(p,p')$ of the vector bundle ${}^\bfa N\partial M\times_{\phi_k}\partial M \to M\times_{\phi_k}\partial M$, denoted  ${}^\bfa N_{(p,p')}\partial M\times_{\phi_k}\partial M$, is canonically identified with ${}^\bfa N_p\partial M$.
Given  $v\in {}^\bfa N_{(p,p')}\partial M\times_{\phi_k}\partial M\cong {}^\bfa N_{p}\partial M$,  choose an $\bfa$-vector field $V$ on $M$ with $V(p)=v$. The lift of $V$ to the left factor of $M^2$ and then to $M^2_\bfa$ is smooth, its restriction to the fibre $\ff_{\bfa,(p,p')}$ of $\ff_\bfa$ over $(p,p')$ depends only on $v$ and is tangent to this fibre, defining a complete vector field on it. These claims will be checked below in local coordinates. Let $A_v$ be the time one map of the flow of this restricted vector field.
Then $v\mapsto A_v$ defines an action of ${}^\bfa N_{(p,p')}\partial M\times_{\phi_k}\partial M $ on $\ff_{\bfa,(p,p')}$ and makes $\ff_{\bfa}$ an affine bundle modelled on the vector bundle ${}^\bfa N\partial M\times_{\phi_k}\partial M$.

Giving an isomorphism \eqref{eq:ffa diffeo} therefore corresponds to choosing a `zero section' in $\ff_\bfa$.
Given a refinement ${}^{\bfa+1}\calV(M)$ of ${}^\bfa\calV(M)$, the fibre diagonal $\Delta_{\phi_k}$, which is blown up last in the construction of $M^2_\bfa$, is defined to order $a_k+1$. Since in $M^2_\bfa$ it is blown up only to order $a_k$, its lift, the intersection of $\ff_\bfa$ with the lift of the interior fibre diagonal $\Delta_{\Phi_k}$, is still well-defined and projects diffeomorphically to $\Delta_{\phi_k}$ under the blow-down map, see the end of Subsection \ref{subsec:qhom}. This is the desired section of $\ff_\bfa\to \partial M \times_{\phi_k}\partial M$.

We now write this construction in local coordinates, where the various claims made will be obvious.
For the purpose of clarity we will restrict to the case of a single fibration, $\partial M \stackrel{\phi} \to B$, defined to order $a$. The general case is completely analogous. We denote base variables by $y$ and fibre variables by $z$. As in the main text we introduce coordinates $t=x'/x$ on the b-blow-up and $T=(1-t)/x^{a}$, $Y=(y-y')/x^a$ near the interior of the top front face $\ff_a$. Write $p=(y,z)$, $p'=(y,z')$. Then $T\in\RR$ and $Y\in\RR^b$ are coordinates on $\intff_{a,(p,p')}$. Recall that the fibre $^a N_p\partial M$ (usually denoted $^\phi N_p\partial M$) is spanned by $x^{1+a}\partial_x$ and $x^a\partial_y$. Let $\alpha\in\RR,\beta\in\RR^b$.
By the calculation \eqref{eq:vf pullback phi} $v=\alpha x^{1+a}\partial_x + \beta x^a\partial_y$ lifts to be smooth near $\intff_{a,(p,p')}$ and the lift is $\alpha\partial_T + \beta \partial_Y$ there, hence is tangent to this fibre, with time one map $A_v:(T,Y)\mapsto (T+\alpha,Y+\beta)$.

A change of coordinates $\ytilde= f(x,y) + x^a r(x,y,z)$,
$\xtilde = xh(x,y) + x^{1+a}s(x,y,z)$ as in \eqref{eq:coord transform} (where in $\xtilde$ the first term was absorbed into the second) yields
\begin{align*}
\Ytilde &= \frac{\ytilde -\ytilde'}{\xtilde^a} = Y h^{-a} d_y f  + (r(0,y,z) - r(0,y,z'))h^{-a}, \\
\Ttilde &= \frac{\xtilde-\xtilde'}{\xtilde^{1+a}} = T h^{-a} + Y  h^{-a-1}d_y h + (s(0,y,z) - s(0,y,z'))h^{-a-1}
\end{align*}
at $\ff_{a,(p,p')}$, where $h=h(0,y)$. This shows that the set  where $T=Y=0$ -- the intersection of the interior fibre diagonal with the front face -- is well defined if and only if the last terms vanish.
This is equivalent to having $r=xr_1$, $s=xs_1$, i.e.  only allowing coordinate changes compatible with an $(a+1)$ structure.

Summarizing, the diffeomorphism \eqref{eq:ffa diffeo} is given, in coordinates, by sending the point on $\intff_{a,(p,p')}$ with coordinates $T=\alpha$, $Y=\beta$ to $\alpha x^{1+a}\partial_x + \beta x^a\partial_y$.

It remains to prove \eqref{eq:dg meaning}.
Assume first $p=p'$. Let $v\in {}^\bfa N_p\partial M$ and again choose an $\bfa$-vector field $V$ with $V(p)=v$ as elements of ${}^\bfa N_p\partial M$. Denote the flow of $V$ by $t\mapsto \exp_{tV}$.
For $q\in \interior{M}$ we have
\begin{equation}
\label{eq:g minus g}
g(\exp_V(q)) - g(q) = \int_0^1 \frac{d}{dt} g(\exp_{tV}(q))\,dt=\int_0^1 dg(V)_{|\exp_{tV}(q)}\, dt.
\end{equation}
Denote $G=\beta_\bfa^*(\pi_l^*g-\pi_r^*g)$, $\Vtilde$ the lift of $V$ to $M^2_\bfa$ used above, and let $Q$ be the point on $\beta_\bfa^* \Delta_{\Phi_k}$ projecting to $(q,q)$, for a compatible choice of interior extension $\Phi_k$ of $\phi_k$. Unravelling the definitions of $G$ and $\Vtilde$ one sees that
the left side of \eqref{eq:g minus g} can be rewritten as \ownremark{$(\pi_l^*g-\pi_r^*g)(\exp_{V_l}(q,q))$ where $V_l$ is the lift of $V$ to the left factor of $M^2$ and then as}
$G(\exp_\Vtilde (Q))$. Now the right side of \eqref{eq:g minus g} is smooth in $v,q$ also for  $q\in\partial M$, so $G(\exp_\Vtilde (Q))$ is also smooth. Recall $\exp_\Vtilde(Q)=A_v(0)$ for $q\in\partial M$ by definition of $A_v$, and this covers all of  $\ff_{\bfa,(q,q)}$ as $v$ varies, so we get the smoothness claim. Also, for $q\in\partial M$ we have $\exp_{tV}(q)=q$ for all $t$ (since $V=0$ over $\partial M$ as section of $TM$), so the right side of \eqref{eq:g minus g} equals $dg(v)$ and \eqref{eq:dg meaning} follows.

To extend this to arbitrary $(p,p')\in \Delta_{\phi_k}$ it suffices to show that $g(q)-g(q')\to 0$ as $(q,q')\in\interior{\Delta}_{\Phi_k}$ approach the boundary. Now $q,q'$ can be connected by a curve $\gamma:[0,1]\to \interior{M}$ for which $\Phi_k(\gamma(t))$ is constant. Then $g(q)-g(q') = \int_0^1 dg(\gamma'(t))\, dt$. Since the $\Phi_k$ fibre component of $dg$ vanishes at $(p,p')\in \Delta_{\phi_k}$ this approaches zero as $(q,q')\to (p,p')$.
\end{proof}



\end{document}